\documentclass[a4paper]{article}

\usepackage{a4}
\usepackage{amsmath, amssymb}
\usepackage{amsthm}
\usepackage{mathtools}
\usepackage{graphicx,color}
\usepackage{float}
\usepackage{xspace}
\usepackage[scientific-notation=true]{siunitx}
\usepackage{enumitem}
\usepackage{booktabs}

\usepackage{tikz}
\usetikzlibrary{shapes.geometric, arrows}
\usetikzlibrary{positioning}
\usetikzlibrary{calc}
\usetikzlibrary{spy}
\usetikzlibrary{decorations.text}
\usetikzlibrary{arrows.meta}
\usepackage{pgfplots}
\usepackage{pgfplotstable}
\pgfplotsset{width=10cm,compat=1.18}
\usepackage[a4paper, total={6.5in, 8.5in}]{geometry}

\usepackage[justification=centering]{caption}

\usepackage{hyperref}
\usepackage[all]{hypcap}

\newcommand{\R}{\mathbb{R}}

\definecolor{light_gray}{gray}{0.75}
\definecolor{lighter_gray}{gray}{0.5}
\colorlet{light_blue}{blue!20}
\definecolor{dark_green}{rgb}{0.0, 0.6, 0.0}
\definecolor{royal_blue}{rgb}{0.0, 0.22, 0.66}
\definecolor{navy_blue}{rgb}{0.0, 0.0, 0.5}
\definecolor{crimson}{rgb}{0.79, 0.0, 0.09}
\definecolor{amethyst}{rgb}{0.6, 0.4, 0.8}
\definecolor{azure}{rgb}{0.0, 0.5, 1.0}

\theoremstyle{plain}
\newtheorem{lemma}{Lemma}
\newtheorem{theorem}[lemma]{Theorem}
\newtheorem{proposition}[lemma]{Proposition}
\newtheorem{corollary}[lemma]{Corollary}

\theoremstyle{definition}
\newtheorem{definition}[lemma]{Definition}
\newtheorem{example}[lemma]{Example}

\theoremstyle{remark}
\newtheorem{remark}[lemma]{Remark}

\newcommand{\panelwidth}{0.62\linewidth}
\pgfplotsset{
	plotsingle/.style={width=0.72\linewidth, height=0.45\linewidth},
	plotpair/.style={width=\linewidth, height=0.75\linewidth},
}

\title{Bilinearity Preserving Algebraic Flux Correction}
\author{Sarthak Sourav Dash\footnote{Department of Mathematics, Indian Institute of Technology Gandhinagar, Palaj, Gandhinagar, 382355, Gujarat, India \texttt{sarthak.dash@iitgn.ac.in}}, Abhinav Jha\footnote{Department of Mathematics, Indian Institute of Technology Gandhinagar, Palaj, Gandhinagar, 382355, Gujarat, India \texttt{abhinav.jha@iitgn.ac.in}}}

\date{}

\pgfplotstableread[row sep=\\]{
	x        Galerkin   LowOrder   Kuzmin     BJK        BJK()        \\
	0.0000   1.000000   1.000000   1.000000   1.000000   1.000000  \\
	0.0154   0.999998   1.000000   0.999998   1.000000   1.000000  \\
	0.0308   1.000004   1.000000   1.000000   1.000000   1.000000  \\
	0.0462   1.000000   1.000000   1.000007   1.000000   1.000000  \\
	0.0615   0.999993   1.000000   0.999993   1.000000   1.000000  \\
	0.0769   1.000007   1.000000   0.999997   1.000000   1.000000  \\
	0.0923   1.000003   0.999999   1.000013   1.000000   1.000000  \\
	0.1077   0.999982   0.999997   0.999989   1.000000   1.000000  \\
	0.1231   1.000021   0.999993   0.999996   1.000000   1.000000  \\
	0.1385   0.999990   0.999984   1.000017   1.000000   1.000000  \\
	0.1538   0.999999   0.999962   0.999986   0.999999   1.000000  \\
	0.1692   0.999991   0.999914   0.999998   0.999998   1.000000  \\
	0.1846   1.000048   0.999811   1.000018   0.999996   0.999999  \\
	0.2000   0.999900   0.999605   0.999981   0.999991   0.999999  \\
	0.2154   1.000147   0.999207   1.000005   0.999984   0.999997  \\
	0.2308   0.999804   0.998474   1.000015   0.999971   0.999995  \\
	0.2462   1.000290   0.997186   0.999970   0.999950   0.999990  \\
	0.2615   0.999514   0.995024   1.000024   0.999917   0.999984  \\
	0.2769   1.000831   0.991554   1.000014   0.999868   0.999973  \\
	0.2923   0.998635   0.986220   0.999928   0.999796   0.999958  \\
	0.3077   1.002157   0.978353   1.000079   0.999696   0.999936  \\
	0.3231   0.996637   0.967200   1.000134   0.999557   0.999905  \\
	0.3385   1.005262   0.951978   0.999338   0.999369   0.999863  \\
	0.3538   0.991757   0.931944   0.997119   0.999117   0.999808  \\
	0.3692   1.012785   0.906473   0.992658   0.998780   0.999734  \\
	0.3846   0.980553   0.875141   0.984615   0.998323   0.999637  \\
	0.4000   1.028793   0.837796   0.970883   0.997675   0.999510  \\
	0.4154   0.958847   0.794601   0.948405   0.996635   0.999338  \\
	0.4308   1.055355   0.746052   0.913083   0.994387   0.999087  \\
	0.4462   0.931398   0.692956   0.859875   0.987848   0.998593  \\
	0.4615   1.058284   0.636384   0.783168   0.967076   0.995842  \\
	0.4769   0.950105   0.577591   0.677534   0.899744   0.973414  \\
	0.4923   0.610789   0.517929   0.538881   0.680370   0.787925  \\
	0.5077   0.275041   0.458744   0.392251   0.210852   0.119812  \\
	0.5231  -0.062773   0.401298   0.271935   0.066950   0.015363  \\
	0.5385   0.055817   0.346686   0.183093   0.022764   0.002827  \\
	0.5538  -0.031425   0.295792   0.119877   0.009097   0.001254  \\
	0.5692   0.019815   0.249259   0.076354   0.004741   0.000878  \\
	0.5846  -0.012107   0.207483   0.047259   0.003139   0.000664  \\
	0.6000   0.007417   0.170628   0.028307   0.002306   0.000513  \\
	0.6154  -0.004508   0.138657   0.016243   0.001755   0.000399  \\
	0.6308   0.002755   0.111365   0.008738   0.001354   0.000312  \\
	0.6462  -0.001720   0.088424   0.004207   0.001050   0.000243  \\
	0.6615   0.001096   0.069426   0.001620   0.000815   0.000188  \\
	0.6769  -0.000685   0.053916   0.000322   0.000630   0.000145  \\
	0.6923   0.000396   0.041425  -0.000126   0.000486   0.000112  \\
	0.7077  -0.000209   0.031498  -0.000063   0.000373   0.000085  \\
	0.7231   0.000119   0.023708   0.000053   0.000284   0.000064  \\
	0.7385  -0.000095   0.017669   0.000008   0.000215   0.000048  \\
	0.7538   0.000087   0.013043  -0.000030   0.000162   0.000036  \\
	0.7692  -0.000061   0.009538   0.000013   0.000121   0.000027  \\
	0.7846   0.000019   0.006912   0.000013   0.000090   0.000019  \\
	0.8000   0.000011   0.004965  -0.000018   0.000066   0.000014  \\
	0.8154  -0.000012   0.003536   0.000003   0.000048   0.000010  \\
	0.8308  -0.000007   0.002498   0.000011   0.000035   0.000007  \\
	0.8462   0.000022   0.001750  -0.000010   0.000025   0.000005  \\
	0.8615  -0.000017   0.001217  -0.000003   0.000018   0.000004  \\
	0.8769  -0.000000   0.000839   0.000010   0.000013   0.000002  \\
	0.8923   0.000012   0.000575  -0.000006   0.000009   0.000002  \\
	0.9077  -0.000010   0.000391  -0.000003   0.000006   0.000001  \\
	0.9231  -0.000001   0.000264   0.000007   0.000004   0.000001  \\
	0.9385   0.000008   0.000177  -0.000004   0.000003   0.000001  \\
	0.9538  -0.000005   0.000118  -0.000002   0.000002   0.000000  \\
	0.9692  -0.000000   0.000078   0.000003   0.000001   0.000000  \\
	0.9846   0.000002   0.000051  -0.000001   0.000001   0.000000  \\
	1.0000   0.000000   0.000000   0.000000   0.000000   0.000000  \\
}\cutlinedata
\pgfplotstableread[row sep=\\]{
	x        Galerkin   LowOrder   Kuzmin     BJK        BJK()         \\
	0.0000   1.000000   1.000000   1.000000   1.000000   1.000000   \\
	0.0154   0.999998   1.000000   0.999998   1.000000   1.000000   \\
	0.0308   1.000004   1.000000   1.000000   1.000000   1.000000   \\
	0.0462   1.000000   1.000000   1.000007   1.000000   1.000000   \\
	0.0615   0.999993   1.000000   0.999993   1.000000   1.000000   \\
	0.0769   1.000007   1.000000   0.999997   1.000000   1.000000   \\
	0.0923   1.000003   0.999999   1.000013   1.000000   1.000000   \\
	0.1077   0.999982   0.999997   0.999989   1.000000   1.000000   \\
	0.1231   1.000021   0.999993   0.999996   1.000000   1.000000   \\
	0.1385   0.999990   0.999984   1.000017   1.000000   1.000000   \\
	0.1538   0.999999   0.999962   0.999986   0.999999   1.000000   \\
	0.1692   0.999991   0.999914   0.999998   0.999998   1.000000   \\
	0.1846   1.000048   0.999811   1.000018   0.999996   0.999999   \\
	0.2000   0.999900   0.999605   0.999981   0.999991   0.999999   \\
	0.2154   1.000147   0.999207   1.000005   0.999984   0.999997   \\
	0.2308   0.999804   0.998474   1.000015   0.999971   0.999995   \\
	0.2462   1.000290   0.997186   0.999970   0.999950   0.999990   \\
	0.2615   0.999514   0.995024   1.000024   0.999917   0.999984   \\
	0.2769   1.000831   0.991554   1.000014   0.999868   0.999973   \\
	0.2923   0.998635   0.986220   0.999928   0.999796   0.999958   \\
	0.3077   1.002157   0.978353   1.000079   0.999696   0.999936   \\
	0.3231   0.996637   0.967200   1.000134   0.999557   0.999905   \\
	0.3385   1.005262   0.951978   0.999338   0.999369   0.999863   \\
	0.3538   0.991757   0.931944   0.997119   0.999117   0.999808   \\
	0.3692   1.012785   0.906473   0.992658   0.998780   0.999734   \\
	0.3846   0.980553   0.875141   0.984615   0.998323   0.999637   \\
	0.4000   1.028793   0.837796   0.970883   0.997675   0.999510   \\
	0.4154   0.958847   0.794601   0.948405   0.996635   0.999338   \\
	0.4308   1.055355   0.746052   0.913083   0.994387   0.999087   \\
	0.4462   0.931398   0.692956   0.859875   0.987848   0.998593   \\
	0.4615   1.058284   0.636384   0.783168   0.967076   0.995842   \\
	0.4769   0.950105   0.577591   0.677534   0.899744   0.973414   \\
	0.4923   0.610789   0.517929   0.538881   0.680370   0.787925   \\
	0.5077   0.275041   0.458744   0.392251   0.210852   0.119812   \\
	0.5231   -0.062773  0.401298   0.271935   0.066950   0.015363   \\
	0.5385   0.055817   0.346686   0.183093   0.022764   0.002827   \\
	0.5538   -0.031425  0.295792   0.119877   0.009097   0.001254   \\
	0.5692   0.019815   0.249259   0.076354   0.004741   0.000878   \\
	0.5846   -0.012107  0.207483   0.047259   0.003139   0.000664   \\
	0.6000   0.007417   0.170628   0.028307   0.002306   0.000513   \\
	0.6154   -0.004508  0.138657   0.016243   0.001755   0.000399   \\
	0.6308   0.002755   0.111365   0.008738   0.001354   0.000312   \\
	0.6462   -0.001720  0.088424   0.004207   0.001050   0.000243   \\
	0.6615   0.001096   0.069426   0.001620   0.000815   0.000188   \\
	0.6769   -0.000685  0.053916   0.000322   0.000630   0.000145   \\
	0.6923   0.000396   0.041425   -0.000126  0.000486   0.000112   \\
	0.7077   -0.000209  0.031498   -0.000063  0.000373   0.000085   \\
	0.7231   0.000119   0.023708   0.000053   0.000284   0.000064   \\
	0.7385   -0.000095  0.017669   0.000008   0.000215   0.000048   \\
	0.7538   0.000087   0.013043   -0.000030  0.000162   0.000036   \\
	0.7692   -0.000061  0.009538   0.000013   0.000121   0.000027   \\
	0.7846   0.000019   0.006912   0.000013   0.000090   0.000019   \\
	0.8000   0.000011   0.004965   -0.000018  0.000066   0.000014   \\
	0.8154   -0.000012  0.003536   0.000003   0.000048   0.000010   \\
	0.8308   -0.000007  0.002498   0.000011   0.000035   0.000007   \\
	0.8462   0.000022   0.001750   -0.000010  0.000025   0.000005   \\
	0.8615   -0.000017  0.001217   -0.000003  0.000018   0.000004   \\
	0.8769   -0.000000  0.000839   0.000010   0.000013   0.000002   \\
	0.8923   0.000012   0.000575   -0.000006  0.000009   0.000002   \\
	0.9077   -0.000010  0.000391   -0.000003  0.000006   0.000001   \\
	0.9231   -0.000001  0.000264   0.000007   0.000004   0.000001   \\
	0.9385   0.000008   0.000177   -0.000004  0.000003   0.000001   \\
	0.9538   -0.000005  0.000118   -0.000002  0.000002   0.000000   \\
	0.9692   -0.000000  0.000078   0.000003   0.000001   0.000000   \\
	0.9846   0.000002   0.000051   -0.000001  0.000001   0.000000   \\
	1.0000   0.000000   0.000000   0.000000   0.000000   0.000000   \\
}\cutlinezoomeddata

\pgfplotstableread[row sep=\\]{
	DOFs     LowOrder   Kuzmin     BJK        BJK()        \\
	121      0.723562   0.474817   0.274325   0.162309  \\
	256      0.590358   0.345918   0.182665   0.114892  \\
	441      0.511307   0.290164   0.137158   0.086500  \\
	676      0.443527   0.244782   0.109847   0.068508  \\
	961      0.406681   0.216760   0.091645   0.057522  \\
	1296     0.369361   0.191950   0.078628   0.049846  \\
	1681     0.347191   0.176577   0.068883   0.044102  \\
	2116     0.322634   0.161571   0.061294   0.040555  \\
	2601     0.308075   0.151807   0.055233   0.036496  \\
	3136     0.290038   0.141581   0.050268   0.032937  \\
	3721     0.279579   0.134186   0.046138   0.030068  \\
	4356     0.265982   0.126320   0.042679   0.027598  \\
	5041     0.257440   0.120436   0.039731   0.025441  \\
	5776     0.246771   0.114208   0.037167   0.023181  \\
	6561     0.240178   0.109523   0.034935   0.022048  \\
	7396     0.231175   0.104747   0.032958   0.020754  \\
	8281     0.225642   0.101130   0.031211   0.019420  \\
	9216     0.218425   0.097100   0.029640   0.018424  \\
	10201    0.213722   0.094038   0.028236   0.017393  \\
}\convergencedata

\pgfplotstableread[row sep=\\]{
	x         Galerkin   LowOrder   Kuzmin     BJK        BJK()         \\
	0.0000    1.000000   1.000000   1.000000   1.000000   1.000000   \\
	0.0019    0.981745   1.000000   1.000000   1.000000   1.000000   \\
	0.0054    0.992591   1.000000   1.000000   1.000000   1.000000   \\
	0.0099    0.979950   1.000000   1.000000   1.000000   1.000000   \\
	0.0153    0.985924   1.000000   1.000000   1.000000   1.000000   \\
	0.0213    0.989416   1.000000   1.000000   1.000000   1.000000   \\
	0.0280    0.988920   1.000000   1.000000   1.000000   1.000000   \\
	0.0353    1.011505   1.000000   1.000000   1.000000   1.000000   \\
	0.0432    0.994911   1.000000   1.000000   1.000000   1.000000   \\
	0.0515    1.029003   1.000000   1.000000   1.000000   1.000000   \\
	0.0603    0.986764   1.000000   1.000000   1.000000   1.000000   \\
	0.0696    1.026368   1.000000   1.000000   1.000000   1.000000   \\
	0.0793    0.969247   1.000000   1.000000   1.000000   1.000000   \\
	0.0894    1.017930   1.000000   1.000000   1.000000   1.000000   \\
	0.1000    0.972245   1.000000   1.000000   1.000000   1.000000   \\
	0.1109    1.018335   1.000000   1.000000   1.000000   1.000000   \\
	0.1221    0.995731   1.000000   1.000000   1.000000   1.000000   \\
	0.1338    1.003142   1.000000   1.000000   1.000000   1.000000   \\
	0.1457    1.012227   1.000000   1.000000   1.000000   1.000000   \\
	0.1580    0.969539   0.999999   1.000000   1.000000   1.000000   \\
	0.1707    1.034254   0.999997   1.000000   1.000000   1.000000   \\
	0.1836    0.957615   0.999990   1.000000   1.000000   1.000000   \\
	0.1969    1.054943   0.999972   1.000000   1.000000   1.000000   \\
	0.2105    0.956233   0.999923   1.000000   1.000000   1.000000   \\
	0.2244    1.026303   0.999808   1.000000   1.000000   1.000000   \\
	0.2385    0.977070   0.999550   1.000000   1.000000   1.000000   \\
	0.2530    0.993524   0.999009   1.000000   1.000000   1.000000   \\
	0.2677    1.046592   0.997943   1.000000   1.000000   1.000000   \\
	0.2827    0.960096   0.995962   1.000000   1.000000   1.000000   \\
	0.2980    1.031484   0.992482   1.000000   1.000000   1.000000   \\
	0.3136    0.970510   0.986690   1.000000   1.000000   1.000000   \\
	0.3294    0.923860   0.977534   0.999999   1.000000   1.000000   \\
	0.3454    1.259689   0.963762   0.999995   1.000000   1.000000   \\
	0.3617    0.569133   0.944010   0.999981   1.000000   1.000000   \\
	0.3783    1.734144   0.916951   0.999927   1.000000   1.000000   \\
	0.3951   -0.255754   0.881492   0.999730   0.999997   1.000000   \\
	0.4122    2.762628   0.836976   0.999017   0.999977   0.999998   \\
	0.4295   -1.302180   0.783362   0.996490   0.999829   0.999972   \\
	0.4470    4.415169   0.721341   0.987695   0.998730   0.999677   \\
	0.4648   -4.062052   0.652349   0.957570   0.990728   0.996305   \\
	0.4827    6.973177   0.578476   0.855982   0.933437   0.958303   \\
	0.5010   -3.039827   0.502266   0.518345   0.528870   0.533488   \\
	0.5194   -5.264285   0.426880   0.155240   0.074881   0.047825   \\
	0.5381   -0.054405   0.355273   0.045893   0.010463   0.004248   \\
	0.5569   -0.107671   0.289475   0.013395   0.001443   0.000374   \\
	0.5760   -0.028804   0.230899   0.003861   0.000197   0.000033   \\
	0.5953    0.221872   0.180306   0.001099   0.000026   0.000003   \\
	0.6149   -0.209504   0.137860   0.000309   0.000004   0.000000   \\
	0.6346    0.027685   0.103227   0.000086   0.000000   0.000000   \\
	0.6545    0.089028   0.075718   0.000024   0.000000   0.000000   \\
	0.6747   -0.078351   0.054423   0.000006   0.000000   0.000000   \\
	0.6950    0.036920   0.038344   0.000002   0.000000   0.000000   \\
	0.7155   -0.020115   0.026491   0.000000   0.000000   0.000000   \\
	0.7363    0.004703   0.017954   0.000000   0.000000   0.000000   \\
	0.7572    0.022160   0.011941   0.000000   0.000000  -0.000000   \\
	0.7783   -0.035706   0.007797   0.000000   0.000000   0.000000   \\
	0.7997    0.020751   0.005000   0.000000   0.000000  -0.000000   \\
	0.8212    0.005125   0.003150   0.000000   0.000000  -0.000000   \\
	0.8429   -0.017981   0.001951   0.000000   0.000000   0.000000   \\
	0.8648    0.013343   0.001188   0.000000   0.000000   0.000000   \\
	0.8869   -0.002317   0.000711   0.000000   0.000000   0.000000   \\
	0.9091   -0.004675   0.000419   0.000000   0.000000  -0.000000   \\
	0.9316    0.005128   0.000243   0.000000   0.000000  -0.000000   \\
	0.9542   -0.002026   0.000139   0.000000   0.000000  -0.000000   \\
	0.9770   -0.000408   0.000079   0.000000   0.000000  -0.000000   \\
	1.0000    0.000000   0.000000   0.000000   0.000000   0.000000   \\
}\cutlinedatanu

\pgfplotstableread[row sep=\\]{
	x         Galerkin   LowOrder   Kuzmin     BJK        BJK()         \\
	0.3454    1.259689   0.963762   0.999995   1.000000   1.000000   \\
	0.3617    0.569133   0.944010   0.999981   1.000000   1.000000   \\
	0.3783    1.734144   0.916951   0.999927   1.000000   1.000000   \\
	0.3951   -0.255754   0.881492   0.999730   0.999997   1.000000   \\
	0.4122    2.762628   0.836976   0.999017   0.999977   0.999998   \\
	0.4295   -1.302180   0.783362   0.996490   0.999829   0.999972   \\
	0.4470    4.415169   0.721341   0.987695   0.998730   0.999677   \\
	0.4648   -4.062052   0.652349   0.957570   0.990728   0.996305   \\
	0.4827    6.973177   0.578476   0.855982   0.933437   0.958303   \\
	0.5010   -3.039827   0.502266   0.518345   0.528870   0.533488   \\
	0.5194   -5.264285   0.426880   0.155240   0.074881   0.047825   \\
	0.5381   -0.054405   0.355273   0.045893   0.010463   0.004248   \\
	0.5569   -0.107671   0.289475   0.013395   0.001443   0.000374   \\
	0.5760   -0.028804   0.230899   0.003861   0.000197   0.000033   \\
	0.5953    0.221872   0.180306   0.001099   0.000026   0.000003   \\
	0.6149   -0.209504   0.137860   0.000309   0.000004   0.000000   \\
	0.6346    0.027685   0.103227   0.000086   0.000000   0.000000   \\
	0.6545    0.089028   0.075718   0.000024   0.000000   0.000000   \\
}\cutlinezoomeddatanu

\pgfplotstableread[row sep=\\]{
	DOFs      LowOrder   Kuzmin     BJK        BJK()         \\
	121       0.660683   0.302928   0.210472   0.200037   \\
	256       0.524355   0.210982   0.141760   0.134375   \\
	441       0.482192   0.170370   0.111608   0.105426   \\
	676       0.419172   0.136788   0.088804   0.083855   \\
	961       0.375889   0.114282   0.073738   0.069602   \\
	1296      0.345222   0.098153   0.063043   0.059490   \\
	1681      0.319230   0.086019   0.055057   0.051942   \\
	2116      0.298782   0.076562   0.048866   0.046093   \\
	2601      0.282433   0.068981   0.043927   0.041428   \\
	3136      0.275359   0.063725   0.040479   0.038169   \\
	3721      0.262089   0.058388   0.037034   0.034916   \\
	4356      0.250579   0.053876   0.034129   0.032174   \\
	5041      0.240551   0.050012   0.031647   0.029831   \\
	5776      0.231242   0.046665   0.029501   0.027806   \\
	6561      0.223168   0.043737   0.027627   0.026038   \\
	7396      0.219876   0.041554   0.026223   0.024713   \\
	8281      0.213030   0.039218   0.024732   0.023307   \\
	9216      0.206565   0.037130   0.023402   0.022052   \\
	10201     0.200679   0.035253   0.022207   0.020926   \\
}\convergencedatanu

\pgfplotstableread[row sep=\\]{
	x        mu_1       mu_1.41421 mu_3       mu_5       mu_10      mu_50      mu_200     \\
	0.0000   1.000000   1.000000   1.000000   1.000000   1.000000   1.000000   1.000000   \\
	0.0002   1.000000   1.000000   1.000000   1.000000   1.000000   1.000000   1.000000   \\
	0.0009   1.000000   1.000000   1.000000   1.000000   1.000000   1.000000   1.000000   \\
	0.0021   1.000000   1.000000   1.000000   1.000000   1.000000   1.000000   1.000000   \\
	0.0038   1.000000   1.000000   1.000000   1.000000   1.000000   1.000000   1.000000   \\
	0.0059   1.000000   1.000000   1.000000   1.000000   1.000000   1.000000   1.000000   \\
	0.0085   1.000000   1.000000   1.000000   1.000000   1.000000   1.000000   1.000000   \\
	0.0116   1.000000   1.000000   1.000000   1.000000   1.000000   1.000000   1.000000   \\
	0.0151   1.000000   1.000000   1.000000   1.000000   1.000000   1.000000   1.000000   \\
	0.0192   1.000000   1.000000   1.000000   1.000000   1.000000   1.000000   1.000000   \\
	0.0237   1.000000   1.000000   1.000000   1.000000   1.000000   1.000000   1.000000   \\
	0.0286   1.000000   1.000000   1.000000   1.000000   1.000000   1.000000   1.000000   \\
	0.0341   1.000000   1.000000   1.000000   1.000000   1.000000   1.000000   1.000000   \\
	0.0400   1.000000   1.000000   1.000000   1.000000   1.000000   1.000000   1.000000   \\
	0.0464   1.000000   1.000000   1.000000   1.000000   1.000000   1.000000   1.000000   \\
	0.0533   1.000000   1.000000   1.000000   1.000000   1.000000   1.000000   1.000000   \\
	0.0606   1.000000   1.000000   1.000000   1.000000   1.000000   1.000000   1.000000   \\
	0.0684   1.000000   1.000000   1.000000   1.000000   1.000000   1.000000   1.000000   \\
	0.0767   1.000000   1.000000   1.000000   1.000000   1.000000   1.000000   1.000000   \\
	0.0854   1.000000   1.000000   1.000000   1.000000   1.000000   1.000000   1.000000   \\
	0.0947   1.000000   1.000000   1.000000   1.000000   1.000000   1.000000   1.000000   \\
	0.1044   1.000000   1.000000   1.000000   1.000000   1.000000   1.000000   1.000000   \\
	0.1146   1.000000   1.000000   1.000000   1.000000   1.000000   1.000000   1.000000   \\
	0.1252   1.000000   1.000000   1.000000   1.000000   1.000000   1.000000   1.000000   \\
	0.1363   1.000000   1.000000   1.000000   1.000000   1.000000   1.000000   1.000000   \\
	0.1479   1.000000   1.000000   1.000000   1.000000   1.000000   1.000000   1.000000   \\
	0.1600   1.000000   1.000000   1.000000   1.000000   1.000000   1.000000   1.000000   \\
	0.1725   1.000000   1.000000   1.000000   1.000000   1.000000   1.000000   1.000000   \\
	0.1856   1.000000   1.000000   1.000000   1.000000   1.000000   1.000000   1.000000   \\
	0.1991   1.000000   1.000000   1.000000   1.000000   1.000000   1.000000   1.000000   \\
	0.2130   1.000000   1.000000   1.000000   1.000000   1.000000   1.000000   1.000000   \\
	0.2275   1.000000   1.000000   1.000000   1.000000   1.000000   1.000000   1.000000   \\
	0.2424   1.000000   1.000000   1.000000   1.000000   1.000000   1.000000   1.000000   \\
	0.2578   1.000000   1.000000   1.000000   1.000000   1.000000   1.000000   1.000000   \\
	0.2736   1.000000   1.000000   1.000000   1.000000   1.000000   1.000000   1.000000   \\
	0.2899   1.000000   1.000000   1.000000   1.000000   1.000000   1.000000   1.000000   \\
	0.3067   1.000000   1.000000   1.000000   1.000000   1.000000   1.000000   1.000002   \\
	0.3240   1.000000   1.000000   1.000000   1.000000   1.000000   1.000000   1.000000   \\
	0.3418   1.000000   1.000000   1.000000   1.000000   1.000000   1.000000   1.000008   \\
	0.3600   0.999999   1.000000   1.000000   1.000000   1.000000   1.000000   1.000023   \\
	0.3787   0.999992   0.999997   1.000000   1.000000   1.000000   1.000000   0.999976   \\
	0.3979   0.999948   0.999977   0.999998   1.000000   1.000000   1.000000   0.999809   \\
	0.4175   0.999662   0.999824   0.999971   0.999994   0.999999   1.000000   0.999817   \\
	0.4376   0.997838   0.998674   0.999659   0.999890   0.999981   1.000000   1.003897   \\
	0.4582   0.986476   0.990251   0.996069   0.998152   0.999420   0.999972   1.005343   \\
	0.4793   0.917082   0.929610   0.955303   0.969346   0.982808   0.996204   1.004368   \\
	0.5008   0.502949   0.503270   0.504085   0.504690   0.505487   0.503996   0.504362   \\
	0.5228   0.083853   0.071231   0.045304   0.031106   0.017468   0.003864  -0.005114   \\
	0.5453   0.013723   0.009903   0.004004   0.001887   0.000594   0.000029  -0.006121   \\
	0.5683   0.002206   0.001353   0.000348   0.000113   0.000020   0.000000  -0.004553   \\
	0.5917   0.000348   0.000182   0.000030   0.000007   0.000001   0.000000   0.000208   \\
	0.6156   0.000054   0.000024   0.000003   0.000000   0.000000   0.000000   0.000232   \\
	0.6400   0.000008   0.000003   0.000000   0.000000   0.000000   0.000000   0.000067   \\
	0.6649   0.000001   0.000000   0.000000   0.000000   0.000000   0.000000  -0.000014   \\
	0.6902   0.000000   0.000000   0.000000   0.000000   0.000000   0.000000  -0.000008   \\
	0.7160   0.000000   0.000000   0.000000   0.000000   0.000000   0.000000   0.000001   \\
	0.7422   0.000000   0.000000   0.000000  -0.000000   0.000000   0.000000  -0.000001   \\
	0.7690   0.000000   0.000000   0.000000  -0.000000  -0.000000   0.000000  -0.000001   \\
	0.7962   0.000000   0.000000   0.000000  -0.000000  -0.000000  -0.000000   0.000000   \\
	0.8239   0.000000   0.000000  -0.000000   0.000000  -0.000000  -0.000000  -0.000000   \\
	0.8521   0.000000   0.000000  -0.000000   0.000000  -0.000000  -0.000000  -0.000000   \\
	0.8807   0.000000   0.000000  -0.000000   0.000000   0.000000  -0.000000  -0.000000   \\
	0.9098   0.000000   0.000000   0.000000  -0.000000   0.000000  -0.000000  -0.000000   \\
	0.9394   0.000000   0.000000   0.000000  -0.000000   0.000000  -0.000000  -0.000000   \\
	0.9695   0.000000   0.000000   0.000000  -0.000000  -0.000000  -0.000000  -0.000000   \\
	1.0000   0.000000   0.000000   0.000000   0.000000   0.000000   0.000000   0.000000   \\
}\cutlinedatamu

\pgfplotstableread[row sep=\\]{
	x        mu_1       mu_1.41421 mu_3       mu_5       mu_10      mu_50      mu_200     \\
	0.3418   1.000000   1.000000   1.000000   1.000000   1.000000   1.000000   1.000008   \\
	0.3600   0.999999   1.000000   1.000000   1.000000   1.000000   1.000000   1.000023   \\
	0.3787   0.999992   0.999997   1.000000   1.000000   1.000000   1.000000   0.999976   \\
	0.3979   0.999948   0.999977   0.999998   1.000000   1.000000   1.000000   0.999809   \\
	0.4175   0.999662   0.999824   0.999971   0.999994   0.999999   1.000000   0.999817   \\
	0.4376   0.997838   0.998674   0.999659   0.999890   0.999981   1.000000   1.003897   \\
	0.4582   0.986476   0.990251   0.996069   0.998152   0.999420   0.999972   1.005343   \\
	0.4793   0.917082   0.929610   0.955303   0.969346   0.982808   0.996204   1.004368   \\
	0.5008   0.502949   0.503270   0.504085   0.504690   0.505487   0.503996   0.504362   \\
	0.5228   0.083853   0.071231   0.045304   0.031106   0.017468   0.003864  -0.005114   \\
	0.5453   0.013723   0.009903   0.004004   0.001887   0.000594   0.000029  -0.006121   \\
	0.5683   0.002206   0.001353   0.000348   0.000113   0.000020   0.000000  -0.004553   \\
	0.5917   0.000348   0.000182   0.000030   0.000007   0.000001   0.000000   0.000208   \\
	0.6156   0.000054   0.000024   0.000003   0.000000   0.000000   0.000000   0.000232   \\
	0.6400   0.000008   0.000003   0.000000   0.000000   0.000000   0.000000   0.000067   \\
	0.6649   0.000001   0.000000   0.000000   0.000000   0.000000   0.000000  -0.000014   \\
}\cutlinezoomeddatamu

\pgfplotstableread[row sep=\\]{
	DOFs      mu_1        mu_1_41421  mu_3        mu_5        mu_10       mu_50       mu_200      \\
	121       0.261283    0.255126    0.243055    0.236801    0.231042    0.225540    0.221803    \\
	256       0.184866    0.180151    0.170972    0.166219    0.161821    0.157578    0.154780    \\
	441       0.132798    0.129280    0.122517    0.119060    0.115883    0.112839    0.110836    \\
	676       0.109947    0.106869    0.101006    0.098045    0.095364    0.092852    0.091199    \\
	961       0.089199    0.086681    0.081893    0.079475    0.077285    0.075232    0.073863    \\
	1296      0.077979    0.075766    0.071603    0.069510    0.067608    0.065806    0.064652    \\
	1681      0.066957    0.065043    0.061426    0.059618    0.057979    0.056430    0.055450    \\
	2116      0.060580    0.058833    0.055537    0.053876    0.052369    0.050955    0.050038    \\
	2601      0.053693    0.052138    0.049207    0.047737    0.046399    0.045144    0.044343    \\
	3136      0.049454    0.048013    0.045304    0.043947    0.042724    0.041575    0.040841    \\
	3721      0.044773    0.043465    0.041005    0.039774    0.038663    0.037622    0.036971    \\
	4356      0.041814    0.040589    0.038286    0.037134    0.036091    0.035109    0.034513    \\
	5041      0.038415    0.037287    0.035168    0.034108    0.033150    0.032248    0.031679    \\
	5776      0.036204    0.035136    0.033133    0.032133    0.031230    0.030385    0.029818    \\
	6561      0.034235    0.033224    0.031328    0.030382    0.029526    0.028721    0.028201    \\
	7396      0.031927    0.030983    0.029213    0.028330    0.027531    0.026781    0.026273    \\
	8281      0.030383    0.029482    0.027794    0.026952    0.026192    0.025480    0.025000    \\
	9216      0.028552    0.027705    0.026117    0.025325    0.024610    0.023940    0.023494    \\
	10201     0.027312    0.026500    0.024981    0.024223    0.023539    0.022897    0.022490    \\
}\convergencedatamudependency
\pgfplotstableread{
	DOFs    mu_1    mu_sqrt2  mu_3    mu_5    mu_10   mu_50   mu_200
	121     2165    2171        2228    2286    2656    4273    15000
	256     1987    2218        2719    2446    2820    5677    15000
	441     2175    2299        2374    2431    3108    7151    15000
	676     2238    2340        2899    3332    3513    8057    15000
	961     2352    2827        2615    3771    4348    9008    15000
	1296    2558    2546        3861    3811    4498    9979    15000
	1681    2580    3146        3793    3886    4873    10540   15000
	2116    2327    2552        3620    3904    5171    11696   15000
	2601    2337    2522        2681    3356    5253    12442   15000
	3136    2227    2549        2682    3520    6057    13179   15000
	3721    2231    2538        2687    3386    6751    13501   15000
	4356    2320    2331        2793    3908    7112    13397   15000
	5041    768     2407        2823    3890    7909    12339   15000
	5776    2478    2578        3403    4509    8831    10948   15000
	6561    2495    2664        3710    4821    9022    12510   15000
	7396    2511    2680        3972    5063    9143    13895   15000
	8281    2430    2673        4156    5225    9167    14900   15000
	9216    2445    2669        4073    5585    9698    15000   15000
	10201   2452    3298        4673    5541    9573    15000   15000
}\dynamicdampingiterations
\begin{document}
\maketitle

\begin{abstract}
Algebraic flux correction schemes for convection--diffusion--reaction equations rely on linearity preservation to avoid unnecessary artificial diffusion: the limiter is inactive whenever the discrete solution is affine. On quadrilateral and hexahedral meshes, this is insufficient, since the finite element space also contains bilinear functions that are reproduced exactly on axis-aligned meshes. We show that a mesh patch admits bilinearity preservation if and only if the central node lies in the convex hull of its neighbours after lifting each neighbour by the product of its coordinate differences, and we derive a sharp explicit value for the limiter parameter on such patches when the cell edges are parallel to the coordinate axes.  We further characterize the meshes on which a mapped bilinear finite element space reproduces a bilinear function exactly, thereby establishing a fundamental limitation on what any limiter can achieve on general quadrilaterals. Numerical experiments confirm the theory: on axis-aligned grids, the proposed limiter reproduces a bilinear solution to the accuracy of the nonlinear iteration, whereas the linearity-preserving limiter does not. A sensitivity study shows that the parameter required for bilinearity preservation lies close to the upper end of the range for which the nonlinear problem can be solved by the fixed-point iteration considered here.
\end{abstract}



\section{Introduction}

Convection--diffusion--reaction (CDR) equations describe the transport of a scalar
quantity: the temperature of a fluid, the concentration of a chemical species, a
pollutant, or a turbulence variable by a moving medium, together with its diffusive
spreading and its production or depletion through reaction.  In most applications the scalar
equation does not appear in isolation but as one component of a larger coupled system,
so that the qualitative properties of its discretization are inherited by the whole
model: a discrete solution exhibiting negative concentrations or negative turbulent
kinetic energies may render the coupled problem meaningless, or simply cause the
nonlinear solver to break down \cite{JR10}. Preserving, at the discrete level, the bounds that the
continuous problem satisfies is therefore not a cosmetic requirement but a
prerequisite for a usable simulation. These equations are defined by
\begin{equation}
	\begin{aligned}
		-\varepsilon \Delta u + \boldsymbol{b} \cdot \nabla u + \sigma u &= f && \text{in } \Omega, \\
		u &= u_{\mathrm{D}} && \text{on } \partial\Omega,
	\end{aligned}
	\label{MainEq}
\end{equation}
where $\Omega\subset \mathbb{R}^d$, $d\geq 1$ is a bounded connected domain with Lipschitz boundary. Here $\varepsilon>0$ denotes the diffusion coefficient, $\boldsymbol{b}$ is the convective field, $\sigma$ is the reaction term, $f$ is the source/sink term, and $u_{\mathrm{D}}$ is the Dirichlet boundary condition.

In the applications just mentioned the diffusion coefficient $\varepsilon$ is
typically very small in comparison with the magnitude of the convection field
$\boldsymbol{b}$ and the length scale of the domain ($L$), i.e., the problem is
\emph{convection dominated} ($\|\boldsymbol{b}\|_{\infty}L\gg \varepsilon$). The solution of Eq.~\eqref{MainEq} then generically develops
boundary and interior layers, which are narrow subregions across which the solution changes abruptly, while being smooth
elsewhere; resolving these layers on a uniform mesh is out of reach in any realistic
computation \cite{RST08}. 
It is classical that the standard Galerkin finite element method is inadequate in this
regime, the discrete solution is polluted by spurious oscillations which are not confined to a neighbourhood of the layers but spread over the entire domain. The algebraic reason is that the Galerkin matrix loses the $M$-matrix structure that the differential operator possesses through the maximum principle, so that the \emph{discrete maximum principle} (DMP), the discrete counterpart of the maximum principle satisfied by the solution of Eq.~\eqref{MainEq}, fails, \cite{BJK25}.

The standard remedy is stabilization. A large family of methods: streamline upwind
Petrov--Galerkin (SUPG) \cite{BH82},  continuous interior penalty
\cite{BE07}, 
and local projection stabilization \cite{Kn10} adds a consistent term to the variational
formulation that is \emph{linear} in the discrete solution. These methods provide
global stability and quasi-optimal error estimates in norms that measure the smooth
part of the solution, but they do not remove the local overshoots and undershoots in
the vicinity of layers, and their behaviour depends on user-chosen parameters. This limitation is structural rather than accidental: Godunov's classical result shows that a linear monotone discretization for a hyperbolic problem cannot achieve an accuracy higher than first order. Thus, in seeking to combine monotonicity with higher-order accuracy, one is naturally led to nonlinear discretizations, even when the underlying differential equation is linear.  This observation motivated  algebraic stabilizations considered in this work; see \cite{BJK24} for a recent survey of finite element methods respecting the DMP.

Among the nonlinear approaches, \emph{algebraic flux correction} (AFC) schemes have
proved particularly successful. Their roots lie in the flux-corrected transport
algorithms of Zalesak \cite{Zal79}, and their finite element form was
developed by Kuzmin and co-authors \cite{KM05, Ku06, Ku09}. Rather than modifying the variational formulation, AFC acts
directly on the algebraic system $A\boldsymbol{u}=\boldsymbol{g}$. One first adds a
symmetric artificial diffusion matrix which eliminates the positive off-diagonal
entries of $A$, producing a low-order operator that satisfies a DMP but is only
first-order accurate and smears layers heavily. One then reinstates as much of the
associated antidiffusive fluxes $f_{ij}=d_{ij}(u_j-u_i)$ as is compatible with
monotonicity, through solution-dependent limiters $\alpha_{ij}(\boldsymbol{u})\in[0,1]$.
The result is a nonlinear algebraic problem whose numerical analysis---solvability,
validity of the DMP, and error estimates---has been carried out only recently in
\cite{BJK16} and is presented systematically in the monograph
\cite{BJK25}.

A concept that has emerged as central in this analysis is \emph{linearity preservation}: the limiter should switch off completely, i.e.,\ no artificial diffusion should be introduced, whenever the discrete solution is affine on the patch surrounding a node. Its role is not merely aesthetic. Linearity preservation is what prevents the scheme from degrading the solution in the smooth, well-resolved part of
the domain, and both numerical evidence and the available analysis indicate that it is
closely tied to optimal convergence rates and to a sharp resolution of layers
\cite{BJK17,BJK25}. The limiter introduced by Barrenechea, John and Knobloch (henceforth the BJK limiter) satisfies the DMP on arbitrary meshes and for arbitrary data while being linearity preserving, the latter under an explicit geometric condition on the parameters $\mu_i$ entering its definition.

This brings us to the question addressed in the present paper. Quadrilateral and hexahedral meshes with $\mathbb{Q}_1$ elements are widely used in computational fluid dynamics: their tensor-product structure facilitates anisotropic and layer-adapted refinement, they can be aligned with the convection field, and they can achieve a given accuracy with fewer degrees of freedom than simplicial meshes. On such meshes, however, the local finite element space $\mathbb{Q}_1$ strictly contains $\mathbb{P}_1$, with the additional bilinear term $xy$ in two dimensions. Consequently, linearity preservation leaves one component of the $\mathbb{Q}_1$ space unprotected. This is particularly relevant because every function in $\mathbb{Q}_1$ is harmonic and hence satisfies the classical maximum principle, while on axis-aligned rectangular meshes such functions are represented exactly by the finite element space. These functions therefore provide a natural class on which a limiter should remain inactive. A merely linearity-preserving limiter, however, may still introduce artificial diffusion for such functions. The effect is visible in the problem with exact solution $u=10xy$ considered in Section~\ref{sec:numerics}: the BJK limiter produces an error of the same order as the discretization error, whereas the limiter proposed here reproduces the solution to machine precision.

Extending linearity preservation to bilinearity preservation is not a routine matter,
and the reason is instructive. In the affine case the underlying mechanism is
convexity: if the central node lies in the convex hull of its neighbours, then the
value of an affine function at that node lies between the minimum and the maximum of
its values at the neighbours, and consequently a finite parameter $\mu_i$ exists for
\emph{every} admissible patch. For bilinear functions this is false. For $u=-xy$ on a
symmetric four-quadrilateral patch the value at the central node is the largest value
on the whole patch (Example~\ref{ex:counterexample}), so no finite $\mu_i$ can exist.
Bilinearity preservation is thus attainable only on a proper subclass of meshes, and
identifying that subclass is the first task.

The results of this paper are the following.

\begin{enumerate}[leftmargin=*, itemsep=0.4em]

	\item We characterize the meshes on which bilinearity preservation can be required
	at all. A patch admits it precisely when it is \emph{bilinearly convex}: lifting
	each surrounding node by the product of its two coordinates relative to the central
	node, the central node must lie in the convex hull of the lifted neighbours
	(Theorem~\ref{thm:BMPiffBC}). Unlike in the affine case, this excludes some
	perfectly ordinary meshes. 

	\item We give a sharp explicit value for the limiter parameter on such patches, in
	closed form when the edges of the patch are parallel to the coordinate axes
	(Theorem~\ref{thm:mu4quad}). On a uniform mesh its value is $3$, against $\sqrt{2}$
	for linearity preservation alone; in three dimensions the corresponding value is
	$7$. The resulting scheme differs from the existing linearity-preserving one only
	through this constant, so its theory \cite{BJK25} and its implementation carry
	over unchanged. 

	\item We show that a mapped bilinear finite element space represents a bilinear
	function exactly if and only if the mesh cells are rectangles with edges parallel to
	the coordinate axes (Theorem~\ref{thm:exactrep}). This is a property of the space,
	not of the stabilization, and it bounds what any limiter can achieve on general
	quadrilaterals.

\end{enumerate}

The theory is confirmed on a problem with a bilinear exact solution, on a
three-dimensional problem exhibiting trilinearity preservation with optimal empirical
convergence orders, and on the interior-layer benchmark of Hughes, Mallet and Mizukami
\cite{HMM86}, where the proposed limiter smears the layer less than the
linearity-preserving limiter, the limiter of Kuzmin \cite{KM05} and the
unlimited low-order scheme.

\medskip
\noindent\textbf{Outline.} The paper is organized as follows. The remainder of this section states the model problem and fixes the notation. Section~\ref{sec:AFC} recalls the algebraic flux correction framework for the CDR equation, the matrix-based discrete maximum principles, and the BJK limiter together with the known results on DMP satisfaction and linearity preservation. Section~\ref{sec:bilinear} contains the main results: Subsection~\ref{sec:meshclass} classifies the mesh patches on which a bilinear maximum principle can hold and establishes its equivalence with bilinear convexity; Subsection~\ref{sec:explicitmu} derives the explicit bound for $\mu_i$ on such patches;  characterizes the meshes on which a bilinear function is exactly representable; and specializes the geometric parameter to axis-aligned patches in two and three dimensions and shows its optimality. Section~\ref{sec:numerics} presents the numerical experiments, and Section~\ref{sec:conclusion} summarizes the results and discusses open questions and extensions, in particular to time-dependent problems and to the virtual element framework.

\subsection{Preliminaries}
\label{sec:model}
We use the standard notation $L^2(\Omega)$, $H^k(\Omega)$, $W^{1,\infty}(\Omega)$ for the usual Lebesgue and Sobolev spaces, and $(\cdot,\cdot)$ for the inner product of $L^2(\Omega)$ as well as of $L^2(\Omega)^d$.

In Eq.\eqref{MainEq},  $\varepsilon>0$ and $\sigma\ge 0$ are constants,
$\boldsymbol{b}\in W^{1,\infty}(\Omega)^d$ with $\nabla\cdot\boldsymbol{b}=0$,
$f\in L^2(\Omega)$, and
$u_{\mathrm{D}}\in H^{1/2}(\partial\Omega)\cap C(\partial\Omega)$. 

Setting $V=H^1_0(\Omega)$, the variational formulation of \eqref{MainEq} reads: find
$u\in H^1(\Omega)$ with $u-u_{\mathrm{D},\mathrm{ext}}\in V$ such that
\begin{equation}
	a(u, v) = g(v) \qquad \forall\, v \in V,
	\label{WeakEq}
\end{equation}
where $u_{\mathrm{D},\mathrm{ext}}\in H^1(\Omega)$ is an extension of
$u_{\mathrm{D}}$, the bilinear form $a:V\times V\to\mathbb{R}$ is given by
\begin{equation*}
	a(u, v) = (\varepsilon \nabla u, \nabla v) + (\boldsymbol{b} \cdot \nabla u + \sigma u, v),
\end{equation*}
and the linear form $g:V\to\mathbb{R}$ by $g(v)=(f,v)$. Under the stated assumptions
$a$ is bounded and elliptic on $V$, so that the Lax--Milgram theorem yields existence
and uniqueness of the solution of Eq.~\eqref{WeakEq}; see, e.g.,
\cite{RST08}. We assume throughout that these criteria are met.

We will consider an admissible triangulation $\mathcal{T}_h$ of the domain $\Omega$ and let the patch surrounding the node $i$ be

\begin{equation}
\omega_i = \bigcup \{K \in \mathcal{T}_h : \mathbf{x}_i \in K\}.
\end{equation}
where $h = \max\limits_{K \in \mathcal{T}_h } \text{diam}(K) $ and the surrounding nodes for a node $i$ as:
\begin{equation}
	S_i = \{ j \in \{1, \dots, n\} \setminus \{i\} : j \in \omega_i \}, \quad i = 1, \dots, m, \label{eq:S_i}
\end{equation}
where $m$ are the number of non-Dirichlet nodes and $n$ are the total number of nodes.

\section{Algebraic Flux Correction Schemes}\label{sec:AFC}

\subsection{Algebraic Flux Correction for CDR equations}
Suppose that the weak solution $u$ of Eq.~\eqref{WeakEq} (see \cite[Sec.~2.2.1]{RST08}) exists in the space $V=H_0^1(\Omega).$ Let $V_h\subset V$ be a conforming finite element subspace spanned by the basis functions $\{\phi_j\}_{j=1}^n$. The Galerkin finite element discretization of Eq.~\eqref{WeakEq} seeks $u_h=\sum_{j=1}^n u_j\phi_j\in V_h$ such that
\begin{align}
\sum_{j=1}^{n}a(\phi_j,\phi_i)u_j &= g(\phi_i),
&& i\in\{1,\ldots,m\},\
u_i &= u_i^b,
&& i\in\{m+1,\ldots,n\}.
\end{align}
Here, the bilinear form $a:V_h\times V_h\to\mathbb{R}$ and the linear form $g:V_h\to\mathbb{R}$ are the restrictions of the corresponding forms in the weak formulation to $V_h$.

The boundary value $u_i^b$ is the value of $u_{\mathrm D}$ at the $i$th boundary node. Defining

$$
[A]_{ij}=a_{ij}=a(\phi_j,\phi_i),
\qquad
g_i=g(\phi_i),
$$

the discrete problem can be written in matrix form as
\begin{equation}
A\mathbf u=\mathbf g,
\label{eq:Lin_Sys}
\end{equation}
with the appropriate modifications to account for the Dirichlet boundary conditions. The resulting linear system is uniquely solvable, for example, if the bilinear form is positive definite on $V_h$, i.e.,

$$
a(v_h,v_h)>0 \qquad\forall\,v_h\in V_h\setminus\{0\}.
$$

The algebraic flux correction (AFC) approach modifies the standard discrete system by introducing an artificial diffusion matrix whose entries may depend on the discrete solution. Denoting these entries by $b_{ij}(\mathbf u)$, the modified system reads
\begin{align}
\sum_{j=1}^{n}\bigl(a_{ij}+b_{ij}(\mathbf u)\bigr)u_j
&=g_i,
&& i\in\{1,\ldots,m\},\label{eq:AFC_Sys}\
u_i&=u_i^b,
&& i\in\{m+1,\ldots,n\}.
\end{align}

To ensure that the added operator has the structure of a diffusion operator, the artificial diffusion matrix is required to satisfy
\begin{align*}
b_{ij}(\mathbf u)&=b_{ji}(\mathbf u),
&& i,j=1,\ldots,n,\
b_{ij}(\mathbf u)&\leq0,
&& i\neq j,\
\sum_{j=1}^{n}b_{ij}(\mathbf u)&=0,
&& i=1,\ldots,n.
\end{align*}

For a Galerkin discretization, it is often the case that
\begin{equation}
a_{ii}>0,\qquad\sum_{j=1}^{n}a_{ij}\geq0, \qquad i=1,\ldots,m.
\end{equation}
To guarantee a discrete maximum principle, the modified matrix must additionally satisfy suitable sign conditions at local extrema. Following \cite[Theorem~7.2]{BJK25}, we distinguish strict and nonstrict extrema.

\begin{definition}
The AFC system Eq.~\eqref{eq:AFC_Sys} is said to satisfy the matrix-based DMP property for \emph{strict extrema} if, for every $\mathbf u=(u_1,\ldots,u_n)^\top\in\mathbb R^n$ and every $i\in\{1,\ldots,m\}$, whenever $u_i$ is a strict local extremum with respect to the stencil $S_i$, i.e.,

$$
u_i>u_j\quad\forall\,j\in S_i \qquad\text{or}\qquad u_i<u_j\quad\forall\,j\in S_i,
$$

then

$$
a_{ij}+b_{ij}(\mathbf u)\leq0 \qquad\forall\,j\in S_i.
$$

\label{def:strict}
\end{definition}

\begin{definition}
The AFC system Eq.~\eqref{eq:AFC_Sys} is said to satisfy the matrix-based DMP property for \emph{nonstrict extrema} if, for every $\mathbf u=(u_1,\ldots,u_n)^\top\in\mathbb R^n$ and every $i\in\{1,\ldots,m\}$, whenever $u_i$ is a local extremum with respect to the stencil $S_i$, i.e.,

$$
u_i\geq u_j\quad\forall\,j\in S_i \qquad\text{or}\qquad u_i\leq u_j\quad\forall\,j\in S_i,
$$

then

$$
a_{ij}+b_{ij}(\mathbf u)\leq0 \qquad\forall\,j\in S_i\text{ with }u_j\neq u_i.
$$

\label{def:nonstrict}
\end{definition}

An AFC scheme with the BJK limiter provides one construction of the artificial diffusion matrix satisfying these conditions. By appropriately limiting the antidiffusive contributions, the resulting nonlinear system can preserve the discrete maximum principle while reducing the artificial diffusion introduced by the corresponding low-order scheme.

\subsection{AFC Scheme BJK Limiters}
The BJK limiter was introduced in Barrenechea et al. \cite{BJK17} designed to have properties like DMPs and linearity preservation for general meshes and data.
It starts by defining an artificial diffusion matrix $(d_{ij})_{i,j=1}^n$:
\begin{equation*}
	d_{ij} = d_{ji} = -\max \{a_{ij}, 0, a_{ji}\} \quad \forall i \neq j, \quad d_{ii} = -\sum_{j \neq i} d_{ij} \,. 
\end{equation*}

To reduce the amount of the artificial diffusion, the matrix $(a_{ij})_{i,j=1}^n$ defined in Eq.~\eqref{eq:Lin_Sys} is modified by setting
\[
a_{ji} := 0 \quad \text{if } a_{ij} < 0, \quad i = 1, \dots, m, \ j = m+1, \dots, n.
\]

This modification influences only the definition of the AFC matrix Eq.~\eqref{eq:AFC_Sys} through the artificial diffusion matrix $(d_{ij})_{i,j=1}^n$.

To define the limiter, let us first denote, for $i = 1, \dots, m$,
\begin{equation}
	u_i^{\max} = \max_{j \in S_i \cup \{i\}} u_j, \quad u_i^{\min} = \min_{j \in S_i \cup \{i\}} u_j, \quad q_i = \sum_{j \in S_i} d_{ij} , \label{def:minmax}
\end{equation}
and $s^+ = \max\{0,s\}$, $s^- = \min\{0,s\} \; \forall s\in \R$. Then, for $i = 1, \dots, m$, one computes
\begin{align}
	P_i^+ &= \sum_{j \in S_i} f_{ij}^+, & P_i^- &= \sum_{j \in S_i} f_{ij}^-,\label{eq:BJKInit}  \\
	Q_i^+ &= q_i (u_i - u_i^{\max}), & Q_i^- &= q_i (u_i - u_i^{\min}),  \\
	R_i^+ &= \min \left\{1, \frac{\mu_i Q_i^+}{P_i^+} \right\}, & R_i^- &= \min \left\{1, \frac{\mu_i Q_i^-}{P_i^-} \right\},\label{eq:Rpm}
\end{align}
where $\mu_i > 0$ are fixed constants and $f_{ij}$ are the anti-diffusive fluxes given by
\begin{equation*}
	f_{ij} = d_{ij} (u_j - u_i), \quad i, j = 1, \dots, n.
\end{equation*}

If $P_i^+$ or $P_i^-$ vanishes, one sets $R_i^+ = 1$ or $R_i^- = 1$, respectively. Furthermore,
\begin{equation}
	R_i^+ = 1, \quad R_i^- = 1, \quad i = m + 1, \dots, n. \label{eq:Def_1}
\end{equation}

Then, one again defines
\begin{equation*}
	\widetilde{\alpha}_{ij} = \begin{cases} 
		R_i^+ & \text{if } f_{ij} > 0, \\ 
		1 & \text{if } f_{ij} = 0, \\ 
		R_i^- & \text{if } f_{ij} < 0, 
	\end{cases} \quad i, j = 1, \dots, n. 
\end{equation*}

Then the symmetrization is performed and one finally sets
\begin{equation}
	\alpha_{ij} = \alpha_{ji} = \min \left\{\widetilde{\alpha}_{ij}, \widetilde{\alpha}_{ji}\right\} , \quad i, j = 1, \dots, n. \label{eq:Def_6}
\end{equation}
Its obvious that $\alpha_{ij} \in [0, 1]$ for all $i, j = 1, \dots, n$.

Rewriting the system of equations, one arrives at the nonlinear algebraic problem
\begin{align*}
	\sum_{j=1}^n a_{ij} u_j + \sum_{j=1}^n (1 - \alpha_{ij}(\mathbf{u})) d_{ij} (u_j - u_i) &= g_i, \quad i = 1, \dots, m,  \\
	u_i &= u_i^b,   \quad i = m + 1, \dots, n. 
\end{align*}

It is assumed that the functions $\alpha_{ij}(\mathbf{u})(u_j - u_i)$ continuously depend on $\mathbf{u} = (u_1, \dots, u_n)^\top \in \mathbb{R}^n$. 
One can write $(b_{ij}(\mathbf{u}))_{i,j=1}^n$ as
\begin{equation}
	b_{ij}(\mathbf{u}) = (1 - \alpha_{ij}(\mathbf{u})) d_{ij} \quad \forall i \neq j, \quad b_{ii}(\mathbf{u}) = -\sum_{j \neq i} b_{ij}(\mathbf{u}).
	\label{eq:AFCMat}
\end{equation}

\vspace{1em}

The solvability of the resulting nonlinear system and its convergence to the
analytic solution are established in \cite[Theorems~10.1 and 10.5]{BJK25}.
In particular, Theorem~10.5 provides the corresponding convergence estimate
in the norm specified there.

\subsection{DMP Satisfaction and Linearity Preservation}
Here we layout the results for DMP satisfaction and linearity preservation of the BJK limiter from \cite{BJK25} for completeness.

\begin{theorem}{(DMP Properties for the AFC Scheme with BJK Limiter)}\label{Thm:DMP_BJK}
	\textit{The system Eq.~\eqref{eq:AFC_Sys} with the AFC matrix defined using the BJK limiter given by Eq.~\eqref{eq:Def_1}--Eq.~\eqref{eq:Def_6} satisfies the DMP properties formulated in Definitions \ref{def:strict} and \ref{def:nonstrict} for the sets $S_i$ defined by Eq.~\eqref{eq:S_i} .}
\end{theorem} 

\begin{proof}
	See \cite[Theorem~10.35]{BJK25}.
\end{proof}

The above theorem implies the validity of DMPs without any additional assumptions on the matrix $(a_{ij})_{i,j=1}^n$. Thus, in particular, the DMPs hold for arbitrary meshes and data satisfying the general assumptions.

\begin{theorem}
	\textit{For any $i \in \{1, \dots, m\}$, let $\mu_i > 0$ be such that}
	\begin{equation}
		p_i - p_i^{\min} \le \mu_i (p_i^{\max} - p_i) \quad \forall p \in \mathbb{P}_1(\omega_i), \label{eq:condition}
	\end{equation}
	\textit{where $p_j = p(\mathbf{x}_j)$, $j = 1, \dots, n$, and $p_i^{\max} = \max\limits_{j \in S_i \cup \{i\}} p_j, \quad p_i^{\min} = \min\limits_{j \in S_i \cup \{i\}} p_j $,  are defined similar to Eq.~\eqref{def:minmax}. Then the nonlinear problem Eq.~\eqref{eq:AFC_Sys} with the AFC matrix Eq.~\eqref{eq:AFCMat} defined using the BJK limiter given by Eq.~\eqref{eq:BJKInit}--Eq.~\eqref{eq:Def_6} is linearity preserving.}
\label{Thm:LP}	
\end{theorem}
\begin{proof}
  See \cite[Lemma~10.9]{BJK25}.
\end{proof}

The theorem implies that, whenever the discrete solution is affine on the patch surrounding node $i$, the limiter is inactive at that node, i.e., $b_{ij}(\mathbf u)=0$ for all $j\in S_i$. Furthermore, the authors have given an explicit lower bound for $\mu_i$ depending on mesh patch geometry for linearity preservation.

\begin{theorem}
	Consider conforming $\mathbb{P}_1$ or $\mathbb{Q}_1$ finite elements on any admissible triangulation of $\Omega$. Let the sets $S_i$ be defined as in Eq.~\eqref{eq:S_i}  and consider any constants $\mu_i$ satisfying
	\begin{equation}
		\mu_i \ge \frac{\max_{\mathbf{x}_j \in \partial\omega_i} \|\mathbf{x}_i - \mathbf{x}_j\|_2}{\mathrm{dist}\left(\mathbf{x}_i, \partial\omega_i^{\text{conv}}\right)}, \quad i = 1, \dots, m,
	\end{equation}
	where $\omega_i^{\text{conv}}$ is the convex hull of $\omega_i$. The nonlinear problem Eq.~\eqref{eq:AFC_Sys} with the AFC matrix Eq.~\eqref{eq:AFCMat} defined using the BJK limiter given by Eq.~\eqref{eq:BJKInit}--Eq.~\eqref{eq:Def_6} is linearity preserving.
\end{theorem}
\begin{proof}
 See \cite[Theorem~6.1]{BJK17} or \cite[Theorem~10.42]{BJK25}.
\end{proof}

The main goal of this paper is to extend the linearity-preserving BJK limiter to bilinear functions. Since functions in $\mathbb Q_1$ are harmonic and, on suitable meshes, are exactly represented by the finite element space, we seek constants \(\mu_i\) for which the limiter remains inactive on such functions. Our proposed limiter will differ from the BJK limiter only in the definition of $\mu_i$ to incapacitate this property. In the spirit of BJK limiter, we will call it the  ``BJK($ \mathbb{Q}_1 $) limiter" . 
 
\begin{theorem}	
For any $i \in \{1, \dots, m\}$, let $\mu_i > 0$ be such that
	\begin{equation}
		p_i - p_i^{\min} \le \mu_i (p_i^{\max} - p_i) \quad \forall p \in \mathbb{Q}_1(\omega_i), \label{eq:conditionQ_1}
	\end{equation}
where $p_j = p(\mathbf{x}_j)$, $j = 1, \dots, n$, and $p_i^{\min}$, $p_i^{\max}$ are defined similar to Eq.~\eqref{def:minmax}. Then the nonlinear problem Eq.~\eqref{eq:AFC_Sys} with the AFC matrix Eq.~\eqref{eq:AFCMat} defined using the limiter given by Eq.~\eqref{eq:BJKInit}--Eq.~\eqref{eq:Def_6} is bilinearity preserving.	\label{Cor:BLP}	
\end{theorem}
\begin{proof}
	The proof follows the exact same steps of Theorem~\ref{Thm:LP}.
\end{proof}

However, not all mesh can have such a constant unlike linear functions. We will start the next section by classifying meshes which can exhibit these properties for a bilinear function.
\section{Bilinearity Preserving Limiters }\label{sec:bilinear}
\subsection{Meshes Capturing Bilinear Maximum Principle }\label{sec:meshclass}

One can observe that, for a linear function, Eq.~\eqref{eq:condition} is always satisfied in any mesh where the internal node is in the convex hull of a polygonal patch $\omega_i$. Although a general bilinear function $p \in \mathbb{Q}_1(\mathbb{R}^2)$ is harmonic ($\Delta p = 0$) and satisfies the classical local maximum principles, it does not exhibit maximum and minimum in the boundary nodes in general. For example, consider the following situation which is not very difficult to occur. 

\begin{example}\label{ex:counterexample}
Consider the function $u = -xy$ in the parallelogram mesh patch with four parallelogram elements (see Fig.~\ref{fig:1}) with the central node at the origin. The surrounding nodes are in the first and third quadrant only. Hence the values at those points are negative. In such a case, the maximum occurs at the central node (0 is the maximum value here). If there were a constant $\mu_i$ as in Eq.~\eqref{eq:condition}, then the right hand side would be zero, whereas the left hand side would be positive since the minimum occurs in the surrounding nodes. Hence, no such constant exists.\label{CounterExample}
\end{example}	

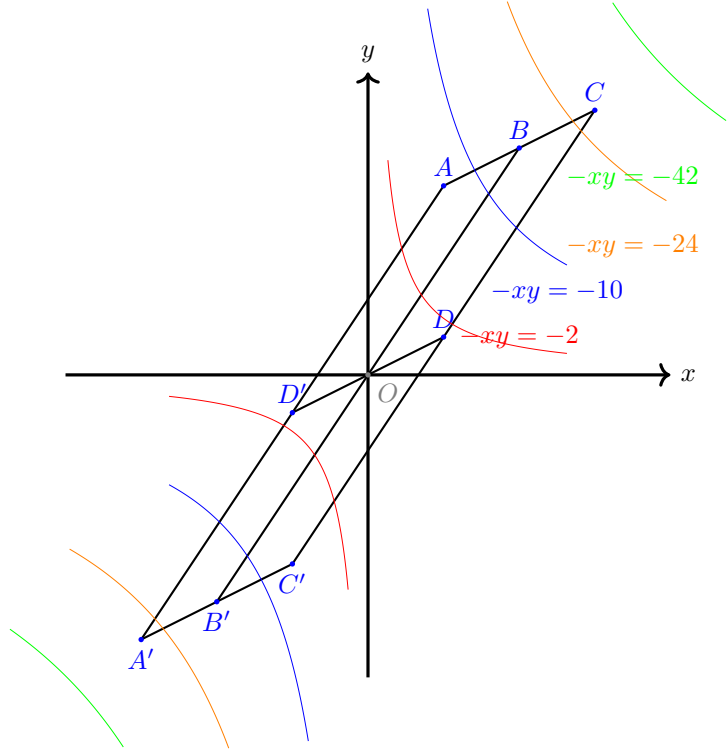
\begin{figure}[tbp]	
	\begin{center}
	\begin{tikzpicture}[scale=0.5]
				
	\draw[->,very thick,black] (-8,0) -- (8,0) node[right] {$x$};
	\draw[->,very thick,black] (0,-8) -- (0,8) node[above] {$y$};

	\coordinate (A) at (2,5);
	\coordinate (B) at (4,6);
	\coordinate (C) at (6,7);
				
	\coordinate (Ap) at (-6,-7);
	\coordinate (Bp) at (-4,-6);
	\coordinate (Cp) at (-2,-5);
				
	\coordinate (O) at (0,0);
	\coordinate (D') at (-2,-1);
	\coordinate (D) at (2,1);
	\draw[thick] (Ap)--(A)--(B)--(C);
	\draw[thick] (Bp)--(O)--(B);
	\draw[thick] (Cp)--(C);
				
	\draw[thick] (Ap)--(Bp)--(Cp);
	\draw[thick] (D')--(O)--(D);				
				
	\fill[blue] (A) circle (2pt) node[above] {$A$};
	\fill[blue] (B) circle (2pt) node[above] {$B$};
	\fill[blue] (C) circle (2pt) node[above] {$C$};
	\fill[blue] (D) circle (2pt) node[above] {$D$};
	\fill[blue] (D') circle (2pt) node[above] {$D'$};
	\fill[blue] (Ap) circle (2pt) node[below] {$A'$};
	\fill[blue] (Bp) circle (2pt) node[below] {$B'$};
	\fill[blue] (Cp) circle (2pt) node[below] {$C'$};			
	\fill[gray] (O) circle (2pt) node[below right] {$O$};
				
	\begin{axis}[
			at={(0,0)},
			anchor=origin,
			axis lines=none,
			xmin=-8,xmax=8,
			ymin=-8,ymax=8,
			samples=200,
			clip=false
			]
					
	\addplot[	domain=3:10,smooth,blue]{66/x};
					
	\addplot[	domain=-3:-10,smooth,blue	]{66/x};
					
	\addplot[	domain=-1:-10,smooth,	red]{12.9/x};
					
	\addplot[	domain=1:10,smooth,red]	{12.9/x};
					
	\addplot[	domain=7:15,	smooth,orange]	{157/x};
					
	\addplot[	domain=-7:-15,smooth,orange]{157/x};
					
	\addplot[	domain=12.3:18,smooth,green]{275/x};
					
	\addplot[	domain=-12.3:-18,smooth,green]{275/x};
	\end{axis}
				
	\node[blue] at (5,2.2) {$-xy=-10$};
	\node[orange] at (7,3.4) {$-xy=-24$};
	\node[red] at (4,1) {$-xy=-2$};
	\node[green] at (7,5.2) {$-xy=-42$};
\end{tikzpicture}					
\end{center}
\caption{A parallelogram patch and the level curves of $-xy$ passing through nodes.}\label{fig:1}
\end{figure}

Therefore, we establish the necessary and sufficient condition for a patch that satisfies the discrete analogue of classical local maximum principle,  i.e., maximum and minimum occur only at the boundary nodes, for any bilinear map in this sub-section.
 
\begin{definition}
We will say a mesh patch $\omega_i$ satisfies \textbf{bilinear maximum principle} if 
\begin{equation}
p_i^{\min} \leq p_i \leq p_i^{\max} \qquad \forall p \in \mathbb{Q}_1(\omega_i), \label{eq:BMP}
\end{equation}
where $p_i^{\min} = \min\limits_{j \in S_i} p_j$ and $p_i^{\max}=\max\limits_{j \in S_i} p_j$, and $p_i$ is the value of $p$ at the $i$th node. Moreover equality in either inequality holds if and only if $p$ is constant.  A mesh will be said to satisfy bilinear maximum principle if and only if each mesh patch $\omega_i$ satisfies Eq.~\eqref{eq:BMP} corresponding to $i=1,2,\cdots,m$.	\label{Def:BilMax}
 \end{definition}
 
 \begin{remark}
	We clarify that this definition is a condition on a mesh patch or a mesh, in contrast to the discrete maximum principle which is a condition on the system of equations to be solved. Also, note that the definition of $p^{\min}_i$ and $p^{\max}_i$ does not consider the value at $p_i$, i.e., it is the maximum (and minimum) of the values of the surrounding nodes $S_i$. 
\end{remark}
The condition given by Eq.~\eqref{eq:BMP} on a mesh patch ensures the existence of a constant $\mu_i$ as required in the condition \eqref{eq:conditionQ_1}. On the other hand, this condition is necessary as we saw in the case of Example~\ref{CounterExample}.

\begin{proposition}\label{Prop:Existence}
For functions $p \in \mathbb{Q}_1(\omega_i) $, the Eq.~\eqref{eq:conditionQ_1} holds if and only if Eq.~\eqref{eq:BMP} holds.
\end{proposition}

\begin{proof}
\textbf{($\implies$):}
For constant $p\in \mathbb{Q}_1(\omega_i) $, Eq.~\eqref{eq:BMP} and Eq.~\eqref{eq:conditionQ_1} holds trivially.

Consider a non-constant $p \in \mathbb{Q}_1(\omega_i)$ satisfying $p_i = p_i^{\max}$. Since $p$ is non-constant, at least one node in $S_i$ satisfies $p_j < p_i$, implying $p_i - p_i^{\min} > 0$. If there is a constant $\mu_i > 0$ satisfying Eq.~\eqref{eq:conditionQ_1}, then 
$$
p_i - p_i^{\min} \le \mu_i (p_i^{\max} - p_i) = \mu_i (0) = 0 \implies p_i - p_i^{\min} \le 0.
$$

Hence, by contradiction $p_i < p_i^{\max}$ is necessary for $\mu_i > 0$ to exist. One can establish $p_i^{\min} < p_i$ with a similar argument. Hence the maximum occurs at the boundary nodes  and Eq,~\eqref{eq:BMP} holds.
	
\vspace{0.5em}

\textbf{($\impliedby$):}
Define a function $R:\mathbb{Q}_1(\omega_i)\rightarrow\mathbb{R}; R(p) = \frac{p_i - p_i^{\min}}{p_i^{\max} - p_i}$ for non-constant $p$. This function is invariant under translation $p \to p + c$ and scaling $p \to \lambda p$ ($\lambda\neq 0 $). Hence we need to consider only the space of non-constant bilinear functions by setting $p_i = 0$ and $\|p\|_{\infty,S_i} = \max_{j \in S_i} |p_j| = 1$. Let
$$
W = \{ p \in \mathbb{Q}_1(\omega_i) \mid p(0,0) = 0, \, \|p\|_{\infty,S_i} = 1, \text{ $p$ is not constant } \}.
$$
$W$ can be parameterized by coefficients $(b,c,d)$, and hence identical to a compact subset of $\mathbb{R}^3$. Compactness of $W$ follows from the fact that $\|\cdot \|_{\infty,S_i}$ is a norm on $\mathbb{Q}_1(\omega_i)$.  Now, $p_i^{\max} > 0$ for all $p \in W$, so the denominator $p_i^{\max} - p_i = p_i^{\max}$ is strictly positive for $p \in W$. Furthermore, the discrete maximum (and minimum) functions are continuous, and therefore $R(p)$ is a composition of continuous functions. Hence $R$ is continuous on the compact set $W$. Thus $R(p)$ is bounded and achieves the maximum in $W$ by the extreme value theorem. One can define
$$
\mu_i = \max_{p \in W} R(p) < \infty.
$$
This proves the existence of a finite constant $\mu_i > 0$.
\end{proof}

\begin{remark}
The strictness in Definition~\ref{Def:BilMax} is what makes the denominator in the proof above strictly positive, and by Theorem~\ref{thm:BMPiffBC} it corresponds to the origin lying in the relative interior of the lifted hull.  Containment in the hull alone is not sufficient: if some neighbour carries 	weight zero in every representation of the origin, there is a non-constant 	$p \in \mathbb{Q}1(\omega_i)$ with $p_i = p_i^{\max} = 0$ and $p_i^{\min} < 0$, and no finite $\mu_i$ satisfies	Eq.~\eqref{eq:conditionQ_1}.
\end{remark}

Now, Eq.~\eqref{eq:BMP} is equivalent to saying that for any bilinear function having a zero at the internal node, there exists a node where it is non-negative and a node where it is non-positive.

\begin{lemma}
A mesh satisfies Bilinear Maximum Principle if and only if it satisfies
\begin{equation}
p_i^{\min} \leq 0 \leq p_i^{\max} \qquad \forall p \in \mathbb{Q}_1(\mathbb{R}^2),\; \text{and} \; p_i =0.   \label{eq:BMP2}
\end{equation}
with $p_i^{\min} = \min\limits_{j \in S_i} p_j$, $p_i^{\max}=\max\limits_{j \in S_i} p_j$.\label{lem:BMP2}
\end{lemma}

\begin{proof}
We establish the equivalence by proving both directions of the implication.
	
\textbf{($\implies$):} 
Assume the mesh satisfies Eq.~\eqref{eq:BMP} for all $p \in \mathbb{Q}_1(\mathbb{R}^2)$. Let $p$ be an arbitrary bilinear function in $\mathbb{Q}_1(\mathbb{R}^2)$ that  satisfies $p_i = 0$. Substituting $p_i = 0$ into Eq.~\eqref{eq:BMP}:
$$
p_i^{\min} \leq 0 \leq p_i^{\max}
$$
Thus, Eq.~\eqref{eq:BMP2} clearly holds.
	
\textbf{($\impliedby$):} 
Assume Eq.~\eqref{eq:BMP2} holds, meaning that any bilinear function vanishing at the $i$th node satisfies $q_i^{\min} \leq 0 \leq q_i^{\max}$. 
	
Let $p \in \mathbb{Q}_1(\mathbb{R}^2)$ be a general bilinear function with an arbitrary value $p_i$ at the $i$th node. We construct an auxiliary function $q$ by subtracting the constant value $p_i$ from $p$:
$$
q(\mathbf{x}) = p(\mathbf{x}) - p_i
$$
	
Because the space $\mathbb{Q}_1(\mathbb{R}^2)$ contains constant functions, the difference of two bilinear functions remains a bilinear function, ensuring $q \in \mathbb{Q}_1(\mathbb{R}^2)$. Evaluating $q$ at the $i$th node gives:
$$
q_i = p_i - p_i = 0
$$
	
Since $q \in \mathbb{Q}_1(\mathbb{R}^2)$ and $q_i = 0$, it satisfies Eq.~\eqref{eq:BMP2} and by the assumption:
\begin{equation}
q_i^{\min} \leq 0 \leq q_i^{\max} \label{eq:q_ineq}
\end{equation}

By definition, the nodal value of $q$ at any neighboring node $j$ is given by $q_j = u_j - u_i$. Because subtracting a constant shifts all nodal values uniformly without altering their relative order, the minimum and maximum neighboring values of $q$ can be expressed directly in terms of $p$:
\begin{align*}
q_i^{\min} &= \min_{j \in S_i} q_j = \min_{j \in S_i} (p_j - p_i) = p_i^{\min} - p_i, \\
q_i^{\max} &= \max_{j \in S_i} q_j = \max_{j \in S_i} (p_j - p_i) = p_i^{\max} - p_i.
\end{align*}

Substituting these relationships back into the inequality \eqref{eq:q_ineq} yields:
$$
p_i^{\min} - p_i \leq 0 \leq p_i^{\max} - p_i
$$
Adding $p_i$ to all parts of the inequality results in:
$$
p_i^{\min} \leq p_i \leq p_i^{\max}
$$

Since $p \in \mathbb{Q}_1(\mathbb{R}^2)$ was chosen arbitrarily, Eq.~\eqref{eq:BMP} holds for all bilinear functions.
\end{proof}

\begin{definition}
We call a mesh patch $\omega_i$ \textbf{bilinearly convex} if the convex hull
of the lifted points
$$
\mathbf{P}_j = \bigl((x_j-x_i),\,(y_j-y_i),\,(x_j-x_i)(y_j-y_i)\bigr)^{\top},	\qquad j \in S_i,
$$
contains the origin of $\mathbb{R}^3$, and \textbf{strictly bilinearly convex} if the origin lies in the relative interior of that convex hull, that is, if there exist weights $\alpha_j > 0$ with
\begin{equation}\label{eq:convexrep}
\sum_{j \in S_i} \alpha_j = 1 	\qquad\text{and}\qquad 	\sum_{j \in S_i} \alpha_j \mathbf{P}_j = \mathbf{0}.	
\end{equation}
\label{Def:BilConv}
\end{definition}

Geometrically, if the nodes of the patch are lifted onto the graph of $xy$ in a
coordinate system centred at the internal node, the patch is bilinearly convex
when the convex hull of the lifted nodes contains the lifted internal node.
Both conditions are decided by a small linear feasibility problem in the weights
$\alpha_j$.
 
The next theorem states the equivalence of the conditions \emph{bilinear maximum principle}~ \ref{Def:BilMax} and \emph{bilinear convexity}~\ref{Def:BilConv} on a mesh patch.
 
\begin{theorem}
A mesh patch $\omega_i$ satisfies
\begin{equation*}
	p_i^{\min} \le p_i \le p_i^{\max} \qquad \forall\, p \in \mathbb{Q}_1(\omega_i)
\end{equation*}
if and only if it is bilinearly convex, and satisfies the bilinear maximum principle of Definition~\ref{Def:BilMax}, with equality only for constant $p$, 	if and only if it is strictly bilinearly convex.\label{thm:BMPiffBC}
\end{theorem}

\begin{proof}
Without loss of generality, let the coordinate system be shifted such that the central node $i$ is located at the origin, therefore $(x_i, y_i) = (0,0)$. Under this translation, any general bilinear function $p \in \mathbb{Q}_1(\mathbb{R}^2)$ that vanishes at the internal node ($p_i = 0$) takes the form:
$$
p(x,y) = bx + cy + dxy
$$
Let $ \mathbf{P_j} = (x_j, y_j, x_j y_j)^\top \in \mathbb{R}^3$ denote the lifted coordinates of the neighboring nodes $j \in S_i$, and let $\mathcal{C} = \text{conv}(\{\mathbf{P_j}\}_{j \in S_i})$ be the convex hull of these points. Since $S_i$ is a finite index set of neighboring nodes, $\mathcal{C}$ is a compact (closed and bounded) convex set in $\mathbb{R}^3$.

\textbf{($\impliedby$):} 
Suppose the mesh is bilinearly convex. Then the origin $\mathbf{0} = (0,0,0)^\top \in \mathcal{C}$ by definition. Algebraically, this is equivalent to the existence of a set of convex coefficients $\{\alpha_j\}_{j \in S_i}$ such that $\alpha_j \geq 0$, $\sum_{j \in S_i} \alpha_j = 1$, and
$$
\sum_{j \in S_i} \alpha_j \mathbf{P_j} = \mathbf{0}.
$$
Taking the inner product with an arbitrary $\mathbf{v} = (b,c,d)^\top \in \mathbb{R}^3$ (representing coefficients), both sides of the equality gives
$$
\mathbf{v} \cdot \sum_{j \in S_i} \alpha_j \mathbf{P_j} = \sum_{j \in S_i} \alpha_j (\mathbf{v} \cdot \mathbf{P_j}) = \sum_{j \in S_i} \alpha_j p(x_j, y_j) = 0
$$
	Since $\alpha_j \geq 0$ and adds up to $1$, the function values at the nodes $p(x_j, y_j)$ cannot all be strictly positive or negative unless they are all zero. Consequently, there exists at least one $j_+ \in S_i$ with  $p_{j_+}\geq 0$ and at least one $j_- \in S_i$ with  $p_{j_-}\leq 0$. This implies $p_i^{\min} \leq 0 \leq p_i^{\max}$, ( i.e. Eq.~\eqref{eq:BMP2}) which in turn, by Lemma~\ref{lem:BMP2},  is equivalent to Eq.~\eqref{eq:BMP}.
	
\textbf{ ($\implies$):} 
We will prove the converse by contraposition. Suppose the mesh patch is \textit{not} bilinearly convex. This implies  $\mathbf{0} \notin \mathcal{C}$.
	
Since $\mathcal{C}$ is a closed convex set and disjoint from  $\{\mathbf{0}\}$ (a closed convex set) , by the Hyperplane Separation Theorem, there exists a hyperplane strictly separating them. Which implies there exists a non-zero vector $\mathbf{v} = (b, c, d)^\top \in \mathbb{R}^3$ and $\epsilon > 0$ such that:
$$
\mathbf{v} \cdot \mathbf{0} < \epsilon < \mathbf{v} \cdot \mathbf{p} \quad \forall \mathbf{p} \in \mathcal{C}.
$$

Evaluating this at the extreme points $\mathbf{P_j} \in \mathcal{C}$ for all $j \in S_i$ gives:
\begin{equation}
\mathbf{v} \cdot \mathbf{P_j} > 0 \quad \forall j \in S_i. \label{eq:separation}
\end{equation}
Now, we construct a bilinear function $p \in \mathbb{Q}_1(\mathbb{R}^2)$ corresponding to the vector $\mathbf{v}$ :
$$
p(x,y) = bx + cy + dxy.
$$
By construction, this function vanishes at the central node ($p_i = p(0,0) = 0$). Evaluating $p$ at any neighboring node $j \in S_i$ yields its value as a simple dot product:
$$
p(x_j, y_j) = bx_j + cy_j + dx_j y_j = \mathbf{v} \cdot \mathbf{P_j}.
$$
Applying the strict separation inequality Eq.~\eqref{eq:separation}, we find that:
$$
p(x_j, y_j) > 0 \quad \forall j \in S_i.
$$

Because the function value is strictly positive at every neighboring node in the patch, the minimum value over the neighborhood must also be strictly positive:
$$
p_i^{\min} = \min_{j \in S_i} p(x_j, y_j) > 0.
$$

This contradicts Eq.~\eqref{eq:BMP2} and consequently condition Eq.~\eqref{eq:BMP} by Lemma~\ref{lem:BMP2} . Thus, if a mesh satisfies condition Eq.~\eqref{eq:BMP2}, it is bilinearly convex.

It remains to identify the strict case. If the patch is strictly bilinearly convex, let $p \in \mathbb{Q}_1(\omega_i)$ be non-constant with $p_i = 0$. Then $p$ does not vanish at every node of $S_i$, and since $\sum_{j \in S_i} \alpha_j p_j = 0$ with all $\alpha_j > 0$, the values $p_j$ cannot all be of one sign, so $p_i^{\min} < 0 < p_i^{\max}$. Conversely, if the origin lies in $\mathcal{C}$ but not in its relative interior, there is a supporting hyperplane, that is, a non-zero $\mathbf{v} = (b,c,d)^\top$ with $\mathbf{v} \cdot \mathbf{P_j} \geq 0$ for every $j \in S_i$ and $\mathbf{v} \cdot \mathbf{P}_{j_0} > 0$ for at least one $j_0 \in S_i$. The associated $p(x,y) = bx + cy + dxy$ is non-constant, vanishes at the internal node, and satisfies $p_i^{\min} = 0 = p_i$, so the bilinear maximum principle fails in its strict form.
\end{proof}

\subsection{Explicit Calculation for Bilinearity Preserving Limiters}\label{sec:explicitmu}
Although Proposition~\ref{Prop:Existence} ensures the existence of $\mu_i$, the next lemma gives an explicit choice for the same.

\begin{lemma}[Geometric parameter for bilinearly convex patches]
Let $\omega_i$ be a bilinearly convex mesh patch, let $\{\alpha_j\}_{j\in S_i}$ be convex weights representing the origin as in Eq.~\eqref{eq:convexrep}, and 	set $T = \{ j \in S_i : \alpha_j > 0 \}$. If, for every $p \in \mathbb{Q}_1(\omega_i)$, the minimum $p_i^{\min}$ is attained at a node of $T$, then
\begin{equation*}
		p_i - p_i^{\min} \;\le\; \mu_i \,\bigl( p_i^{\max} - p_i \bigr),
		\qquad
		\mu_i = \frac{1}{\alpha_{\min}} - 1,
		\qquad
		\alpha_{\min} = \min_{j \in T} \alpha_j .
\end{equation*}
\label{general_formula}
\end{lemma}

\begin{proof}
Similar to the Proposition~\ref{Prop:Existence}, it is enough to look at $\frac{-p_i^{\min}}{p_i^{\max}}$ for any shifted bilinear function $p(x,y) = bx + cy + dxy$.
	
As the origin belongs to the convex hull of the lifted nodal values, there exist convex weights $\{\alpha_j\}_{j \in S_i}$ such that $\sum_{j \in S_i} \alpha_j = 1$ with $\alpha_j \geq 0$, satisfying:
$$
\sum_{j \in S_i} \alpha_j p_j = 0.
$$
	
Let the minimum over the patch occur at the $k$th node, which by hypothesis
belongs to $T$, so that $\alpha_k > 0$. Separating the node with minimum value from the summation gives:
$$
\sum_{j \in S_i \setminus \{k\}} \alpha_j p_j + \alpha_k p_i^{\min} = 0.
$$
	
Since $p_j \leq p_i^{\max} \quad  \forall j \in S_i$, replacing $p_j$'s with the maximum bound gives the inequality:
$$
\left( \sum_{j \in S_i \setminus \{k\}} \alpha_j \right) p_i^{\max} \geq -\alpha_k p_i^{\min}.
$$
	
Since $\alpha_k$ and $p_i^{\max}$ are positive, dividing both sides yields:
$$
\frac{\sum_{j \in S_i \setminus \{k\}} \alpha_j}{\alpha_k} \geq \frac{-p_i^{\min}}{p_i^{\max}}.
$$
	
Using the partition of unity property ($\sum_{j \in S_i \setminus \{k\}} \alpha_j = 1 - \alpha_k$), we obtain the final bound:
$$
\frac{-p_i^{\min}}{p_i^{\max}} \leq \frac{1 - \alpha_k}{\alpha_k} = \frac{1}{\alpha_k} - 1 \leq \frac{1}{\alpha_{\min}} - 1,
$$
where $\alpha_{\min} = \min_{j \in T} \alpha_j > 0$.
\end{proof}

\subsubsection{Exact Representation}\label{sec:exacrep}

The difficulty of representing physical bilinear functions using mapped \(\mathbb Q_1\) finite elements is well known (see \cite{ABF02}). We denote $\text{Aff}(\hat{K})$ for the class of affine functions mapping $\hat{K}$ to the physical elements $K \in \mathcal{T}_h $ and $\text{Bi}(\hat{K})$ for the class of isoparametric maps. 

We define the mapping classes as
\begin{equation}
\text{Aff}(\hat{K}) = \left\{ F: \hat{K} \to \mathbb{R}^2 \;\middle|\; F(\xi, \eta) = \begin{pmatrix} c_1 + a_1\xi + a_2\eta \\ c_2 + b_1\xi + b_2\eta \end{pmatrix} \;\middle|\;
\begin{aligned}
	&c_1, c_2, a_1, a_2, b_1, b_2 \in \mathbb{R}, \\
	&a_1 b_2 - a_2 b_1 \neq 0
\end{aligned} \right\},
	\label{eq:Aff_map_class}
\end{equation}
and
\begin{equation}
	\text{Bi}(\hat{K}) = \left\{ F: \hat{K} \to \mathbb{R}^2 \;\middle|\; F(\xi, \eta) = \begin{pmatrix} c_1 + a_1\xi + a_2\eta + a_3\xi\eta \\ c_2 + b_1\xi + b_2\eta + b_3\xi\eta \end{pmatrix} \;\middle|\;
	\begin{aligned}
		&c_i, a_i, b_i \in \mathbb{R}, \\
		&J(\xi,\eta) \neq 0, \; \forall (\xi,\eta) \in \hat{K}
	\end{aligned} \right\},
	\label{eq:Bi_map_class}
\end{equation}
where $J$ is the Jacobian of the transformation.  Here is an simple example of non-representability of bilinear function on a bilinear element mapped by a simple $\pi/4$ rotation map from $ \text{Aff}(\hat{K})$.  

\textbf{Example:} Consider the reference element as the unit square $\hat{K} = [-1,1]^2$ with local coordinates $(\xi, \eta)$. The standard bilinear finite element reference space is
$$
\hat{V}_{\hat{K}} = \text{span}\{1, \xi, \eta, \xi\eta\}.
$$
Consider a physical element $K$ which is a $\pi/4$ anti-clockwise rotation of $\hat{K}$ (see Fig.~\ref{fig:rotation}).  The corresponding map $F: \hat{K} \to K$ is defined as:
\begin{align*}
x(\xi, \eta) &= \frac{1}{\sqrt{2}}(\xi - \eta), \\
y(\xi, \eta) &= \frac{1}{\sqrt{2}}(\xi + \eta).
\end{align*}
Let's consider the function $u(x,y) = xy$ and it's approximation in $V_K$. The pullback of this function to the reference element is given by
$$
\hat{u}(\xi, \eta) = u(x(\xi, \eta), y(\xi, \eta)) = \left( \frac{1}{\sqrt{2}}(\xi - \eta) \right) \left( \frac{1}{\sqrt{2}}(\xi + \eta) \right).
$$
Expanding this product yields:
\begin{equation}
	\hat{u}(\xi, \eta) = \frac{1}{2}(\xi^2 - \eta^2).
\end{equation}
The resulting function contains the quadratic terms ($\xi^2$, $\eta^2$). Since $\xi^2 - \eta^2 \notin \hat{V}_{\hat{K}}$, the pullback $\hat{u}(\xi, \eta)$ does not belong to the reference space $\hat{V}_{\hat{K}}$. Therefore, the mapped finite element $V_K$ space cannot represent the bilinear function $u = xy$ exactly.

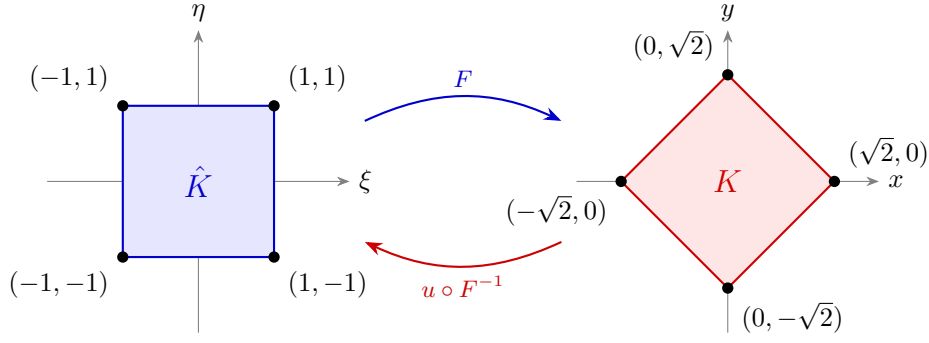
\begin{figure}[tbp]

\begin{center}
	\begin{tikzpicture}[
		>=Stealth,
		node distance=5cm,
		point/.style={circle, fill=black, inner sep=1.5pt},
		axes/.style={gray, thin, ->},
		refDomain/.style={thick, fill=blue!10, draw=blue!80!black},    
		mappedDomain/.style={thick, fill=red!10, draw=red!80!black}
		]
		
		\begin{scope}[local bounding box=ref_element]
			\draw[axes] (-2, 0) -- (2, 0) node[right, black] {$\xi$};
			\draw[axes] (0, -2) -- (0, 2) node[above, black] {$\eta$};
			
			\draw[refDomain] (-1, -1) rectangle (1, 1);
			\node[blue!80!black, font=\large] at (0,0) {$\hat{K}$};
			
			\node[point, label={below left:$(-1,-1)$}] at (-1,-1) {};
			\node[point, label={below right:$(1,-1)$}] at (1,-1) {};
			\node[point, label={above right:$(1,1)$}] at (1,1) {};
			\node[point, label={above left:$(-1,1)$}] at (-1,1) {};
		\end{scope}
		
		\begin{scope}[xshift=7cm, local bounding box=phys_element]
			\draw[axes] (-2, 0) -- (2, 0) node[right, black] {$x$};
			\draw[axes] (0, -2) -- (0, 2) node[above, black] {$y$};
			
			\draw[mappedDomain] (0,-1.41) -- (1.41,0) -- (0,1.41) -- (-1.41,0) -- cycle;
			\node[red!80!black, font=\large] at (0,0) {$K$};
			
			\node[point, label={below right:$(0,-\sqrt{2})$}] at (0,-1.41) {};
			\node[point, label={above right:$(\sqrt{2},0)$}] at (1.41,0) {};
			\node[point, label={above left:$(0,\sqrt{2})$}] at (0,1.41) {};
			\node[point, label={below left:$(-\sqrt{2},0)$}] at (-1.41,0) {};
		\end{scope}
		
		\draw[->, thick, blue!80!black] 
		(2.2, 0.8) to[bend left=25] 
		node[midway, above, font=\small] {$F$} 
		(4.8, 0.8);
		
		\draw[->, thick, red!80!black] 
		(4.8, -0.8) to[bend left=25] 
		node[midway, below, font=\small] {$u \circ F^{-1}$ } 
		(2.2, -0.8);
		
	\end{tikzpicture}
\end{center}
.\caption{Affine Mapping of a $\mathbb{Q}_1$ element by a $\dfrac{\pi}{4}$ rotation}
\label{fig:rotation}
\end{figure}

If the exact solution of Eq.~\eqref{MainEq} is bilinear, then a Galerkin finite element method using \(\mathbb Q_1\) elements on a rotated affinely mapped mesh, such as the one considered above, cannot in general represent it exactly. The next theorem establishes necessary and sufficient conditions for exact representation.

\begin{theorem}\label{thm:exactrep}
	A global space $V_h$ of $\mathbb{Q}_1$ finite elements can exactly represent a physical bilinear function $u(x,y) = a+ bx+ cy+  dxy$ with $d\neq 0$ if and only if the physical elements have their edges aligned with the coordinate axes.
\end{theorem}

\begin{proof}
Let $K$ be a physical element of the mesh and $\hat{K}$ be the standard $[-1,1]^2$ reference element. Let $F: \hat{K} \to K$ be an arbitrary affine mapping from the reference square to a physical parallelogram:
\begin{align*}
x(\xi, \eta) &= c_1 + a_1\xi + a_2\eta, \\
y(\xi, \eta) &= c_2 + b_1\xi + b_2\eta.
\end{align*}
Let the mapping be non-degenerate, then Jacobian determinant has to be non-zero: $J = a_1b_2 - a_2b_1 \neq 0$.
	
Pulling back an arbitrary bilinear function $u(x,y) = a + bx + cy + dxy$ to the reference space gives:
\begin{equation}
\hat{u}(\xi, \eta) = a + b(c_1 + a_1\xi + a_2\eta) + c(c_2 + b_1\xi + b_2\eta) + d(c_1 + a_1\xi + a_2\eta)(c_2 + b_1\xi + b_2\eta).\label{eq:expansion}
\end{equation}
Expanding this expression, observe that the first, second, third as well as the decoupled products from the cross-term belong to $\text{Span}\{1, \xi, \eta, \xi\eta\}$. Isolating the remaining higher-order quadratic terms:
\begin{equation}
\hat{u}(\xi, \eta) = [\text{terms belonging to } \hat{\mathbb{Q}}_1(\hat{K})] + da_1b_1\xi^2 + da_2b_2\eta^2.	\label{eq:terms}
\end{equation}
	
\textbf{($\implies$):} If $\hat{u}(\xi, \eta)$ is exactly representable, we need $\hat{u}(\xi, \eta) \in \hat{\mathbb{Q}}_1(\hat{K})$. Therefore, the coefficients of the unrepresentable terms must vanish identically, $i.e$,
$$
da_1b_1 = 0 \quad \text{and} \quad da_2b_2 = 0.
$$
Since $d \neq 0$, this requires $a_1b_1 = 0$ and $a_2b_2 = 0$. Given that $J = a_1b_2 - a_2b_1 \neq 0$, let's look at the consequent algebraic constraints:
	\begin{itemize}
		\item If $a_1 = 0$, then $b_1 \neq 0$ (to keep $J \neq 0$). This implies $a_2 \neq 0$, which in turn implies $b_2 = 0$. The mapping reduces to $x = c_1 + a_2\eta$ and $y = c_2 + b_1\xi$.
		\item If $b_1 = 0$, then $a_1 \neq 0$ (to keep $J \neq 0$). This forces $b_2 \neq 0$, which in turn forces $a_2 = 0$. The mapping reduces to $x = c_1 + a_1\xi$ and $y = c_2 + b_2\eta$.
	\end{itemize}
	In both cases, the physical $x$-coordinate depends only on one local coordinate, and the physical $y$-coordinate depends only on the other. This implies same $x$-coordinate along two opposing edges and same $y$-coordinate along the other two opposing edges. Geometrically, $K$ is a rectangle whose edges are parallel to the $x$ and $y$ axes.
	
	\textbf{($\impliedby$):} If the element edges are axis-aligned, assume without loss of generality, the mapping is $x = c_1 + a_1\xi$ and $y = c_2 + b_2\eta$ (implying $a_2 = 0$ and $b_1 = 0$). Substituting these expressions into Eq.~\eqref{eq:expansion} evaluates $d a_1(0)\xi^2 + d(0)b_2\eta^2 = 0$. Consequently, the unrepresentable components vanish completely, confirming that the pullback $\hat{u}(\xi, \eta)$ belongs to $ \hat{\mathbb{Q}}_1 $. If it is true for all mesh elements, this implies the considered bilinear function belongs to the global finite element space, i.e., $u \in V_h$.	
\end{proof}

It immediately follows from the Definitions~\ref{eq:Aff_map_class}--\ref{eq:Bi_map_class} that, if we use the isoparametric mapping class $\text{Bi}(\hat{K})$ for the same, the pullback function $\hat{u}(\xi, \eta)$ involves extra terms similar to Eq.~\eqref{eq:terms} and hence the bilinear functions can't be represented using this mapping class as well. 

\begin{corollary}\label{Cor:Isoparametric}
 A global space $V_h$ of $\mathbb{Q}_1$ finite elements mapped through the isoparametric mapping class $\text{Bi}(\hat{K})$ cannot exactly represent a physical bilinear function $u(x,y) = a+ bx+ cy+ dxy$ with $d\neq 0$.
\end{corollary}

\begin{corollary}\label{Cor:GalerkinExact}
Suppose that a consistent Galerkin finite element problem has a unique discrete solution and a bilinear analytical solution
$u(x,y) = a+bx+cy+dxy$ with $d\neq 0$. Then the discrete finite element solution $u_h$ using $\mathbb{Q}_1$ elements is exact, i.e., $u_h=u$, if and only if the mesh is axis-aligned.
\end{corollary}

\begin{proof}
By Theorem~\ref{thm:exactrep}, the exact solution $u$ belongs to the global finite element space $V_h$ if and only if the mesh is axis-aligned.

If the mesh is axis-aligned, then $u\in V_h$. Since the Galerkin method is consistent, $u$ satisfies the discrete variational problem. By uniqueness of the discrete solution, it follows that $u_h=u.$

Conversely, if $u_h=u$, then $u\in V_h$. By Theorem~\ref{thm:exactrep}, the mesh must be axis-aligned.
\end{proof}

\subsubsection{Geometric Parameter for Axes Aligned Meshes}\label{sec:axesaligned}
In Lemma~\ref{general_formula} we compute $\mu_i$ as $1/\alpha_{\min}-1$ also we noticed in the previous subsection that while working with mapped elements only elements that are axis-aligned can be represented by $\mathbb{Q}_1$ elements. Now, we give the exact representation of $\mu_i$ for an axis aligned grid. For simplicity, let us take a case of a 9-node axis-aligned rectangular patch be defined by the grid lines $x \in \{-c, 0, a\}$ and $y \in \{-d, 0, b\}$ with $a, b, c, d > 0$ (see Fig.~\ref{Rect}). The central node $i$ is located at the origin $(0,0)$. The four corner nodes of the patch are denoted as $V_1 = (a,b)$, $V_2 = (-c,b)$, $V_3 = (-c,-d)$, and $V_4 = (a,-d)$. From now on we refer to this as 4-Quad patch.
\begin{figure}[tbp]
\begin{center}
	\begin{tikzpicture}[
		>=Stealth,
		dot/.style={circle, fill=black, inner sep=2pt}
		]
		
		\def\a{3}
		\def\b{2.5}
		\def\c{2.5}
		\def\d{2}
		
		\fill[gray!10] (-\c, -\d) rectangle (\a, \b);
		
		\draw[thick, black] (-\c, -\d) rectangle (\a, \b);
		
		\draw[->, thick] (-\c-1.5, 0) -- (\a+1.5, 0) node[right] {$x$};
		\draw[->, thick] (0, -\d-1.5) -- (0, \b+1.5) node[above] {$y$};
		
		\node[dot, label={above right:$V_1 = (a,b)$}] at (\a, \b) {};
		\node[dot, label={above left:$V_2 = (-c,b)$}] at (-\c, \b) {};
		\node[dot, label={below left:$V_3 = (-c,-d)$}] at (-\c, -\d) {};
		\node[dot, label={below right:$V_4 = (a,-d)$}] at (\a, -\d) {};
		
		\node[dot, label={below right:$(a,0)$}] at (\a, 0) {};
		\node[dot, label={below left:$(-c,0)$}] at (-\c, 0) {};
		\node[dot, label={above left:$(0,b)$}] at (0, \b) {};
		\node[dot, label={below left:$(0,-d)$}] at (0, -\d) {};
		
		\node[dot, label={[xshift=0.1cm, yshift=0.1cm]above right:$i = (0,0)$}] at (0, 0) {};
		
	\end{tikzpicture}
\end{center}
\caption{Example of a general rectangular 4-Quad patch.}
\label{Rect}
\end{figure}
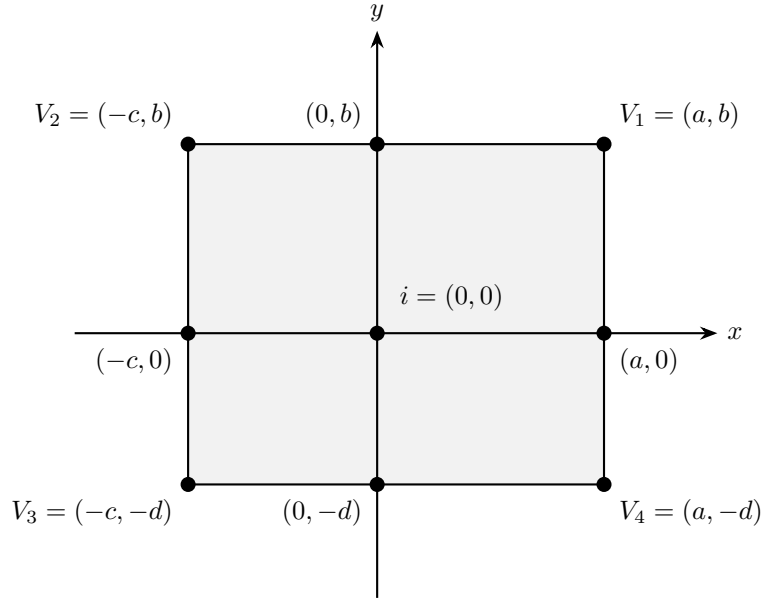

\begin{lemma}[Extrema of a Bilinear Function]\label{lem:extrema}
	For a bilinear function $p(x,y) \in \mathbb{Q}_1$ defined on the rectangular domain $R = [-c, a] \times [-d, b]$, the maximum $p^{\max}$ and the minimum $p^{\min}$ over $R$ are attained at corner points of the rectangular patch; that is, there exist corners $V, V'$ among $V_1, V_2, V_3, V_4$ with $p(V) = p^{\max}$ and $p(V') = p^{\min}$.
\end{lemma}

\begin{proof}
	Let $p(x,y) = \alpha + \beta x + \gamma y + \delta xy$. Since $\Delta p = 0$, the extrema of $p$ lie on the boundary of the patch. Along the horizontal line $y = b$ the restriction $p(x,b) = (\alpha + \gamma b) + (\beta + \delta b)x$ is affine in $x$, so its extrema over $x \in [-c,a]$ are attained at $x = -c$ or $x = a$; the same holds along $y = -d$. Exchanging the roles of the variables, along $x = a$ and $x = -c$ the extrema are attained at $y = -d$ or $y = b$. Hence the extrema over the boundary, and therefore over $R$, are attained at the four corners.
\end{proof}
\begin{lemma}
	Let an axis-aligned rectangular domain centered at the origin be defined by the vertices $V_1=(a,b)$, $V_2=(-c,b)$, $V_3=(-c,-d)$, and $V_4=(a,-d)$ with $a,b,c,d > 0$. Denote the lifted points by $\tilde{V}_k = (x_k, y_k, x_k y_k)$ in $\R^3$. Then the origin $(0,0,0)$ can be expressed uniquely as the convex combination $\sum_{k=1}^{4} \alpha_k \tilde{V}_k = (0,0,0)$ with $\sum_{k=1}^{4} \alpha_k = 1$. The resulting weights $\alpha_k$ are  the ratios of the 2D quadrant areas diametrically opposite to each respective corner to the area of the entire domain.
	\label{lem:explicit_weight_2D}
\end{lemma}

\begin{proof}
	The conditions for the convex coefficient require satisfying the following linear system:
	\begin{align*}
		\alpha_1 + \alpha_2 + \alpha_3 + \alpha_4 &= 1  \\
		\alpha_1 a - \alpha_2 c - \alpha_3 c + \alpha_4 a &= 0, \\
		\alpha_1 b + \alpha_2 b - \alpha_3 d - \alpha_4 d &= 0 , \\
		\alpha_1 ab - \alpha_2 cb + \alpha_3 cd - \alpha_4 ad &= 0 .
	\end{align*}
	
	Solving this linear system provides
	\begin{align}
		\alpha_1  &= \frac{cd}{(a+c)(b+d)} \label{alpha_f}, \\
		\alpha_2  &= \frac{ad}{(a+c)(b+d)}, \\
		\alpha_3  &= \frac{ab}{(a+c)(b+d)}, \\
		\alpha_4  &= \frac{cb}{(a+c)(b+d)}.\label{alpha_l}
	\end{align}
	
\end{proof}

\begin{figure}[tbp]
	\centering
	\includegraphics[width=\panelwidth]{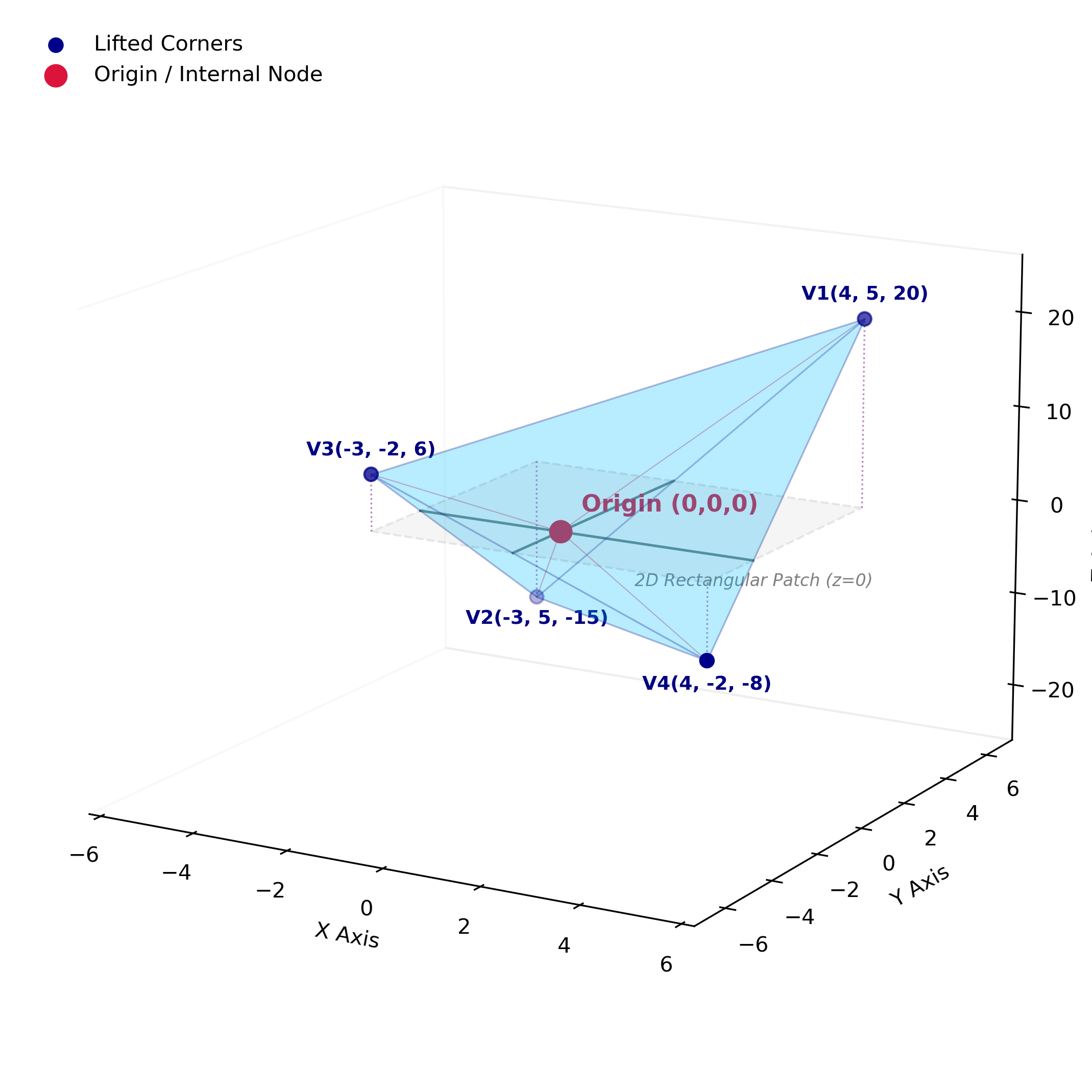}
	\caption{Tetrahedron with lifted corners for specific case of $a=4, b=5,c=3,$ and $d=2$.}
	\label{tetra}
\end{figure}

Geometrically, the above result means that the tetrahedron by the four corner points contains the origin(see Fig.~\ref{tetra}). Observe that the above coefficients are the unique barycentric coefficients.  

Next we define the linear eccentricity of a patch.
\begin{definition}
	We define the linear eccentricity of the patch as $R_x = \frac{\min(a, c)}{a + c}$ and $R_y = \frac{\min(b, d)}{b + d}$.
\end{definition}

Note that the minimum among the weights in ($\ref{alpha_f})-(\ref{alpha_l})$ is $\alpha_{\min} = R_x R_y$.

\begin{theorem}[Geometric Parameter $\mu_i$ for an axes aligned 4-Quad Patch]\label{thm:mu4quad}
	For any bilinear function $p(x,y) \in \mathbb{Q}_1(\mathbb{R}^2) $ evaluated over the 4-Quad patch the following bound holds:
	\begin{equation}\label{eq:geom_param}
		(p_i - p_i^{\min}) \leq \left( \frac{1}{R_x R_y} - 1 \right) (p_i^{\max} - p_i),\quad \mathrm{for}\ \ \ i\in \{1,2,\cdots, m\}.
	\end{equation}
\label{Thm:Explicit_bound}
\end{theorem}

\begin{proof}
By Lemma~\ref{lem:explicit_weight_2D} the origin is represented by the weights
Eq.~\eqref{alpha_f}--\eqref{alpha_l}, which are supported on the four corners
$V_1,\dots,V_4$ and strictly positive there, so that $T = \{V_1,\dots,V_4\}$ in
the notation of Lemma~\ref{general_formula}. By Lemma~\ref{lem:extrema} the
extrema of $p$ over the patch are attained at corners, and the corners are nodes
of $S_i$; hence $p_i^{\min}$ is attained on $T$ and the hypothesis of
Lemma~\ref{general_formula} is satisfied. The smallest of the four weights is
$\alpha_{\min} = R_x R_y$, and substituting this into Lemma~\ref{general_formula}
gives the assertion.
\end{proof}

We can generalize the above results to higher dimension. The following result describes the 3D-counterpart.

\begin{lemma}
	Let an axis-aligned rectangular parallelepiped domain containing the origin be defined by the 8 vertices $V_k = (x_k, y_k, z_k)$ with $x_k \in \{-c, a\}$, $y_k \in \{-d, b\}$, and $z_k \in \{-f, e\}$ for $a,b,c,d,e,f > 0$. Denote the lifted points in $\mathbb{R}^7$:$$\tilde{V}_k = (x_k, y_k, z_k, x_k y_k, x_k z_k, y_k z_k, x_k y_k z_k)$$The origin $(0,0,0,0,0,0,0)\in \mathbb{R}^7$ can be expressed uniquely as the convex combination $\sum_{k=1}^{8} \alpha_k \tilde{V}_k = \mathbf{0}$ with $\sum_{k=1}^{8} \alpha_k = 1$. The resulting weights $\alpha_k$ are the ratios of the 3D octant volumes diametrically opposite to each respective corner to the volume of the parallelepiped.
\end{lemma}

\begin{proof}The convex coefficients $\alpha_k$ must satisfy an $8 \times 8$ linear system corresponding to unity and the 7 components of the lifted space,i.e.,

	    \begin{align}
		\sum_{k=1}^{8} \alpha_k &= 1, \label{eq:unity3d} \\
		\sum_{k=1}^{8} \alpha_k x_k = 0, \quad \sum_{k=1}^{8} \alpha_k y_k &= 0, \quad \sum_{k=1}^{8} \alpha_k z_k = 0, \\
		\sum_{k=1}^{8} \alpha_k x_k y_k = 0, \quad \sum_{k=1}^{8} \alpha_k x_k z_k &= 0, \quad \sum_{k=1}^{8} \alpha_k y_k z_k = 0, \\
		\sum_{k=1}^{8} \alpha_k x_k y_k z_k &= 0.
		\end{align}
	The unique solution is a multiplication of ratios from each axis. For any vertex $V_k = (x_k, y_k, z_k)$, define the 1D weight functions:
	$$p_x(x_k) = \begin{cases} \frac{c}{a+c}, &  x_k = a \\ \frac{a}{a+c}, &  x_k = -c \end{cases}, \; p_y(y_k) = \begin{cases} \frac{d}{b+d}, & y_k = b \\ \frac{b}{b+d}, &  y_k = -d \end{cases},\; p_z(z_k) = \begin{cases} \frac{f}{e+f}, & z_k = e \\ \frac{e}{e+f}, &  z_k = -f \end{cases}$$
	
	The weights are given by $\alpha_k = p_x(x_k) p_y(y_k) p_z(z_k)$, which equals:
	$$\alpha_k = \frac{ |\tilde{P_k}| }{(a+c)(b+d)(e+f)}$$
	where $|\tilde{P_k}|$ is the volume of the parallelepiped diametrically opposite to $V_k$.
\end{proof}
\begin{remark}
	Similarly,  Theorem~\ref{Thm:Explicit_bound} can be extended to 3D by multiplying the $z$-axis eccentricity $R_z$ to $R_x R_y$. So the 3D geometric factor for axes aligned meshes will be 
	\begin{equation}
		\mu_i = \left( \frac{1}{R_x R_y R_z} - 1 \right).
		\label{eq:3DFormula}
	\end{equation}
	\label{Rem:3D}
\end{remark}

To demonstrate that the geometric bound established in Theorem~\ref{Thm:Explicit_bound} is optimal, we consider an example where the inequality holds with strict equality.

\begin{example}
	Let a 4-Quad patch be defined on $\Omega = [-1,1]^2$. The grid lines are defined by $x \in \{-1, 0, 1\}$ and $y \in \{-1, 0, 1\}$, positioning the central node $i$ at the origin $(0,0)$ as usual. This configuration yields $a = b = c = d = 1$. Consider the bilinear function $p(x,y) \in \mathbb{Q}_1(\mathbb{R}^2)$: $p(x,y) = x + y - xy$.

Evaluating $p(x,y)$ at the central node:
\begin{equation*}
	p_i = p(0,0) = 0.
\end{equation*}

We compute the values of the function at the remaining 8 nodes across the patch:
\begin{align*}
p(1,1) &=  1, \quad &p(-1,1) &=  1, \\
p(-1,-1) &= -3, \quad &p(1,-1) &=  1, \\
p(1,0) &= 1, \quad &p(-1,0) &= -1, \\
p(0,1) &= 1, \quad &p(0,-1) &= -1.
\end{align*}

The global maximum and minimum values over the 9-node patch thus are :
\begin{equation*}
p_i^{\min} = -3 \quad \text{and} \quad p_i^{\max} = 1.
\end{equation*}

We now evaluate both sides of the inequality in Eq.~\eqref{eq:geom_param}. The left-hand side yields:
\begin{equation*}
 p_i - p_i^{\min} = 0 - (-3) = 3.
\end{equation*}

For the right-hand side, we compute the eccentricity ratios and the geometric parameter:
\begin{equation*}
R_x = \frac{\min(1,1)}{1+1} = 0.5, \quad R_y = \frac{\min(1,1)}{1+1} = 0.5.
\end{equation*}

\begin{equation*}
\left( \frac{1}{R_x R_y} - 1 \right) = \left( \frac{1}{0.5 \times 0.5} - 1 \right) = 4 - 1 = 3.
\end{equation*}

Using this multiplier  the RHS of Eq.~\eqref{eq:geom_param} evaluates to:
\begin{equation}
	\left( \frac{1}{R_x R_y} - 1 \right) (p_i^{\max} - p_i) = 3 \times (1 - 0) = 3.
\end{equation}
This shows the optimality of the geometric parameter.

\end{example}

\section{Numerical Results}
\label{sec:numerics}

\subsection{Discretization, Solver and Implementation}
\label{subsec:setup}

The AFC scheme \eqref{eq:AFC_Sys} is a nonlinear algebraic system, and its numerical
solution is a nontrivial matter in its own right; a systematic comparison of iterative
schemes for exactly this class of problems was carried out by Jha and John
\cite{JJ19}, who found that the simple fixed point iterations outperform Newton-type
methods, because the iteration matrix retains properties that the linear solver can
exploit. Following their terminology, we use the \emph{fixed point right-hand side}
iteration method, combined with dynamic damping. The nonlinear loop was terminated when the residual of the system of equation dropped below $10^{-10}$ or when $10^6$ iterations were reached. The nonlinear loop converged in all the simulations reported below, with the exception of the largest values of $\mu_i$ considered in the sensitivity study of Section~\ref{subsec:musens}; these cases are identified explicitly there. We present results with respect to three limiters, BJK from \cite{BJK17}, Kuzmin from \cite{BJK16}, and the BJK ($\mathbb{Q}_1$) developed in this paper.

All computations were performed on a machine with an Intel Core i9-14900KS processor
(32 threads) and $128$ GB of RAM, running Linux.  The simulations were carried on the Julia code \texttt{Ferrite.jl} \cite{Ferrite_jl}.

\subsubsection{Grids}
\label{subsec:grids}

Three families of quadrilateral grids on $[0,1]^2$ are used, shown in
Fig.~\ref{fig:meshes}. Unless stated otherwise, the initial grid consists $30 \times 30$ cells in the
experiments.

\begin{itemize}[leftmargin=*]
	\item \emph{Uniform grids.} The tensor-product is created grid with nodes
	$x_k = y_k = k/N$, $k = 0,\dots,N$. Every interior patch is a $4$-Quad
	patch with $a = b = c = d = h$, so that $R_x = R_y = 1/2$ and
	Theorem~\ref{Thm:Explicit_bound} gives $\mu_i = 3$ at every interior node.

	\item \emph{Non-uniform axis-aligned grids.} The grid lines remain parallel to the
	coordinate axes but are graded, the coordinates being
	$x_k = y_k = (k/n)^{p}$, $k = 0,\dots,n$, with $p = 3/2$ unless stated otherwise.
	The resulting patches are axis-aligned but asymmetric, so that
	$R_x, R_y < 1/2$ and $\mu_i$ varies from node to node. This family is included
	precisely because it separates axis-alignment, which is what
	Theorem~\ref{thm:exactrep} requires, from uniformity, which it does not.

	\item \emph{Randomly perturbed grids.} Starting from the uniform grid, every
	interior vertex is displaced by an independent random vector drawn uniformly from
	$[-0.2h, 0.2h]^2$, where $h$ is the uniform mesh width; boundary vertices are kept
	fixed so that the domain is reproduced exactly. The cells are then general
	quadrilaterals, mapped from the reference square by bilinear (isoparametric) maps,
	and by Theorem~\ref{thm:exactrep} and Corollary~\ref{Cor:Isoparametric} a bilinear
	function is no longer contained in $V_h$. The perturbation is small enough that all
	cells remain convex and the grid stays admissible.
\end{itemize}
For three dimensional problems the grids are defined in the same manner as two dimensional grids. For brevity we omit the details.
\begin{figure}[tbp]
	\centering
	\includegraphics[width=\textwidth]{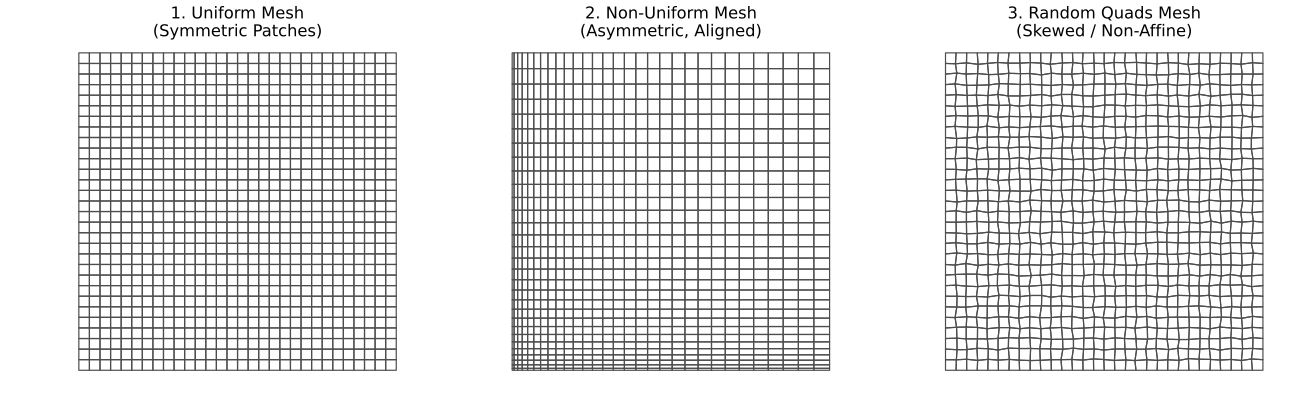}
	\caption{Grids used in the bilinearity-preservation test of
		Section~\ref{subsec:bilinear}: uniform (left), non-uniform axis-aligned
		(centre), and randomly perturbed by $20\%$ of the mesh width (right).}
	\label{fig:meshes}
\end{figure}

\subsection{Problems with Multilinear Solutions}
\label{subsec:bilinear}
Here we verify the multilinearity preservation property of the proposed limiter by constructing problems with bilinear and trilinear solutions.
\subsubsection{Problem with Bilinear Solution}
This example is chosen so as to isolate the property of bilinearity preservation. We
take $\varepsilon = 10^{-7}$ and $\boldsymbol{b}(x,y) = (y, x)^{\top}$ on
$\Omega = (0,1)^2$, with $\sigma = 0$. Note that
$\nabla \cdot \boldsymbol{b} = 0$, as required in Section~\ref{sec:model}. The exact
solution is prescribed as
\begin{equation}
	u(x,y) = 10\,xy,
	\label{eq:exact_bilinear}
\end{equation}
the Dirichlet datum $u_{\mathrm{D}}$ is its restriction to $\partial\Omega$, and the
right-hand side $f$ of Eq.~\eqref{MainEq} is computed accordingly.

Fig.~\ref{fig:10xy_uniform}--\ref{fig:10xy_random} compare the absolute error of
the BJK limiter (left) with that of the BJK($ \mathbb{Q}_1 $) limiter (right) on the three grids of
Fig.~\ref{fig:meshes}, in the order in which they appear there.

\begin{figure}[tbp]
	\centering
	\includegraphics[width=\textwidth]{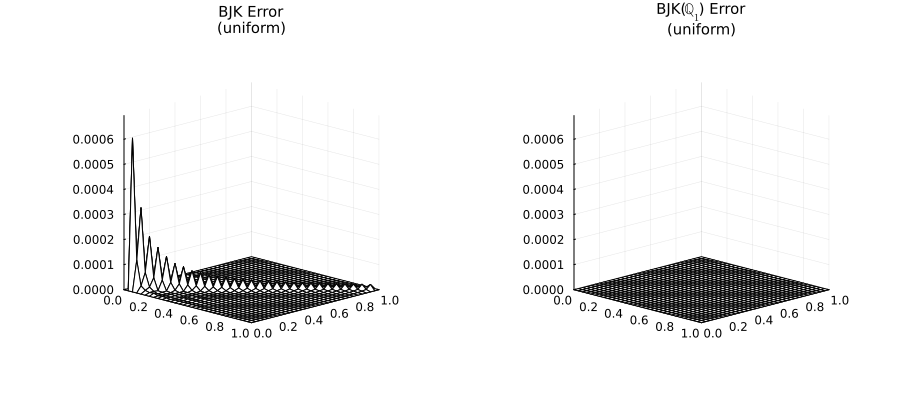}
	\caption{Example~\ref{subsec:bilinear}: Absolute error for the solution $u = 10xy$ on the uniform grid: BJK
		limiter (left), BJK($ \mathbb{Q}_1 $) limiter (right).}
	\label{fig:10xy_uniform}
\end{figure}

\begin{figure}[tbp]
	\centering
	\includegraphics[width=\textwidth]{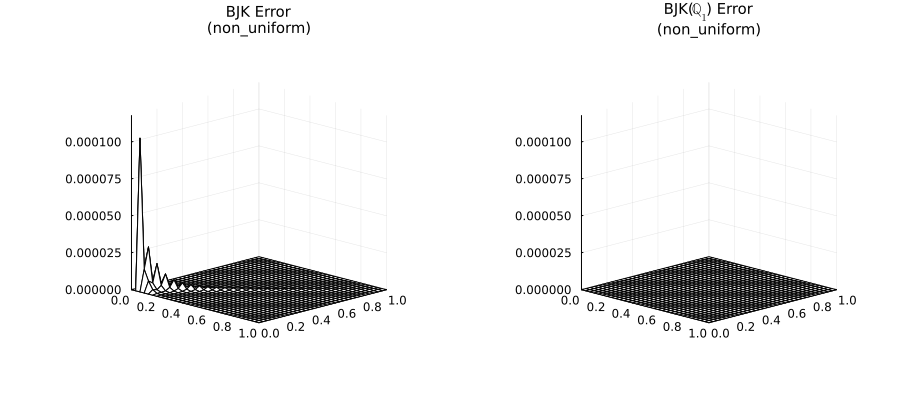}
	\caption{Example~\ref{subsec:bilinear}: Absolute error for the solution $u = 10xy$ on the non-uniform axis-aligned grid: BJK limiter (left), BJK($ \mathbb{Q}_1 $) limiter (right).}
	\label{fig:10xy_nonuniform}
\end{figure}

\begin{figure}[tbp]
	\centering
	\includegraphics[width=\textwidth]{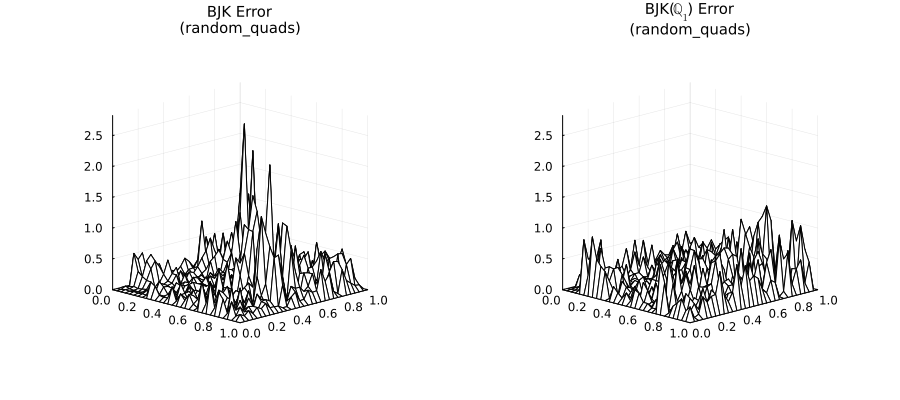}
	\caption{Example~\ref{subsec:bilinear}: Absolute error for the solution $u = 10xy$ on the randomly perturbed grid: BJK limiter (left), BJK($ \mathbb{Q}_1 $) limiter (right).}
	\label{fig:10xy_random}
\end{figure}

\subsubsection{Problem with Trilinear Solution}
This example demonstrates the extension of the proposed limiter to multi-linearity preservation. We have taken $\varepsilon = 10^{-7}$ and $\boldsymbol{b}(x,y,z) = (yz,xz,xy)^{\top}$ on
$\Omega = (0,1)^3$, with $\sigma = 0$. 

The exact solution is 
\begin{equation}
	u(x,y,z) = 10\,xyz.
	\label{eq:exact_trilinear}
\end{equation}
The Dirichlet boundary condition is the restriction of $10xyz$ to the boundary cube and $f$ is set by substituting the exact solution into the left hand side of \eqref{MainEq}. Fig.~\ref{fig:10xyz_uniform}--\ref{fig:10xyz_random} depict the absolute error function in the cross section $z=0.5$.

\begin{figure}[tbp]
	\centering
	\includegraphics[width=\textwidth]{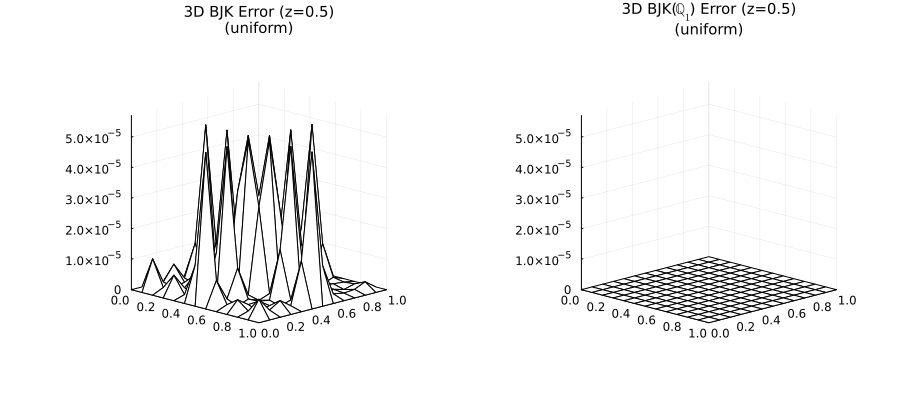}
	\caption{Example~\ref{subsec:bilinear}: Absolute error for the solution $u = 10xyz$ on the uniform grid for $z=0.5$: BJK
		limiter (left), BJK($ \mathbb{Q}_1 $) limiter (right).}
	\label{fig:10xyz_uniform}
\end{figure}

\begin{figure}[tbp]
	\centering
	\includegraphics[width=\textwidth]{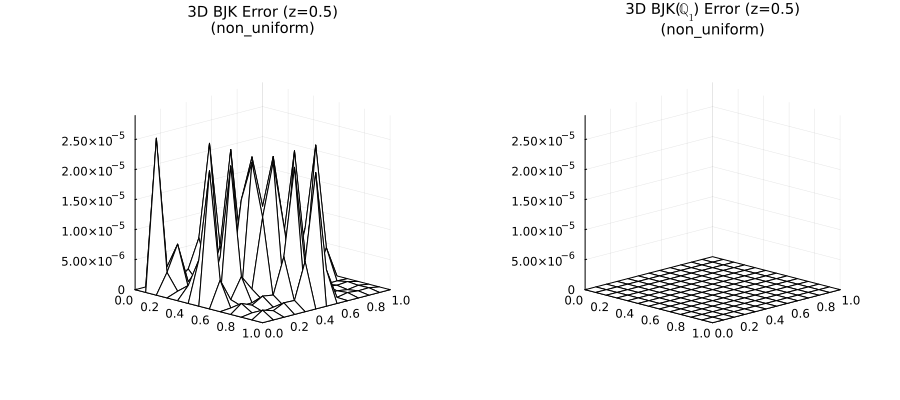}
	\caption{Example~\ref{subsec:bilinear}: Absolute error for the solution $u = 10xyz$ on the non-uniform axis-aligned grid for $z=0.5$: BJK limiter (left), BJK($ \mathbb{Q}_1 $) limiter (right).}
	\label{fig:10xyz_nonuniform}
\end{figure}

\begin{figure}[tbp]
	\centering
	\includegraphics[width=\textwidth]{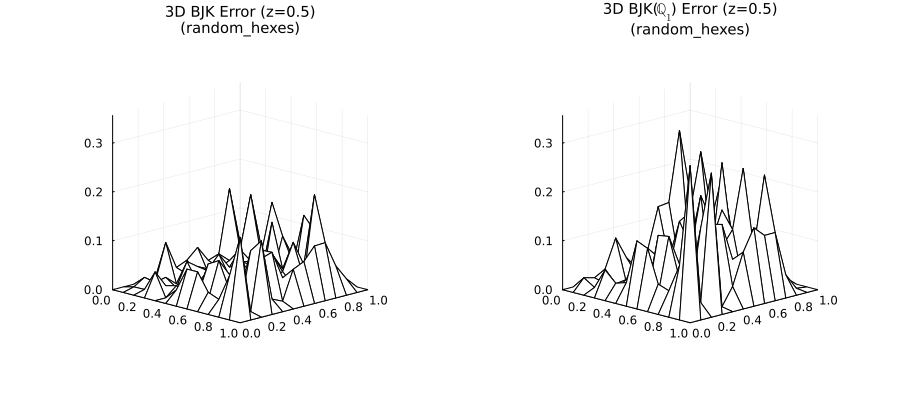}
	\caption{Example~\ref{subsec:bilinear}: Absolute error for the solution $u = 10xyz$ on the randomly perturbed grid for $z=0.5$: BJK limiter (left), BJK($ \mathbb{Q}_1 $) limiter (right).}
	\label{fig:10xyz_random}
\end{figure}
In both of these cases, on the two axis-aligned grids the error of the BJK($\mathbb{Q}_1 $) limiter is zero up to the accuracy
of the nonlinear iteration whereas the BJK limiter produces an error of the size of the discretization error. This is exactly what the theory predicts: by
Theorem~\ref{thm:exactrep} the bilinear function Eq.~\eqref{eq:exact_bilinear} and the trilinear function Eq.~\eqref{eq:exact_trilinear} belongs to
$V_h$ on an axis-aligned grid, by Theorem~\ref{Thm:Explicit_bound} the BJK($ \mathbb{Q}_1 $) limiter is inactive on it, and hence by Corollary~\ref{Cor:GalerkinExact} the discrete solution
coincides with $u$. The BJK limiter, being only linearity preserving, remains active
and adds artificial diffusion where none is needed.

On the randomly perturbed grid neither limiter reproduces the solution, and the two
errors are of comparable magnitude. This is not a failure of the BJK($ \mathbb{Q}_1 $) limiter but of the
finite element space: by Corollary~\ref{Cor:Isoparametric} a bilinear function is not
contained in $V_h$ on non-axis-aligned quadrilaterals, so no choice of limiter can
recover it.

\subsection{A Three-Dimensional Problem with Polynomial Solution}
\label{subsec:3D}

We now consider $\Omega = (0,1)^3$ with $\varepsilon = 10^{-7}$,
$\boldsymbol{b} = (1.0, 0.2, 0.1)^{\top}$ and $\sigma = 0$, and the exact solution
\begin{equation}
	u(x,y,z) = xyz\,(1-x)(1-y)(1-z),
	\label{eq:exact_3d}
\end{equation}
which vanishes on $\partial\Omega$, so that $u_{\mathrm{D}} = 0$; the right-hand side
$f$ is computed from Eq.~\eqref{eq:exact_3d}. 
The grids are uniform and axis-aligned (similar to Grid~1 in Fig.~\ref{fig:meshes}).

As described in Remark~\ref{Rem:3D}, the construction extends to the trilinear case
with no essential change. For a uniform axis-aligned hexahedral grid all three
eccentricities equal $1/2$, so Eq.~\eqref{eq:3DFormula} yields $\mu_i = 7$ at every
interior node. Fig.~\ref{fig:3d_error} shows the absolute error on the cut planes
$z = 0.5$ and $y = 0.5$ for the BJK limiter (left) and the BJK($ \mathbb{Q}_1 $) limiter (right). The BJK($ \mathbb{Q}_1 $) error is smaller throughout. 

\begin{figure}[tbp]
	\centering
	\includegraphics[width=\textwidth]{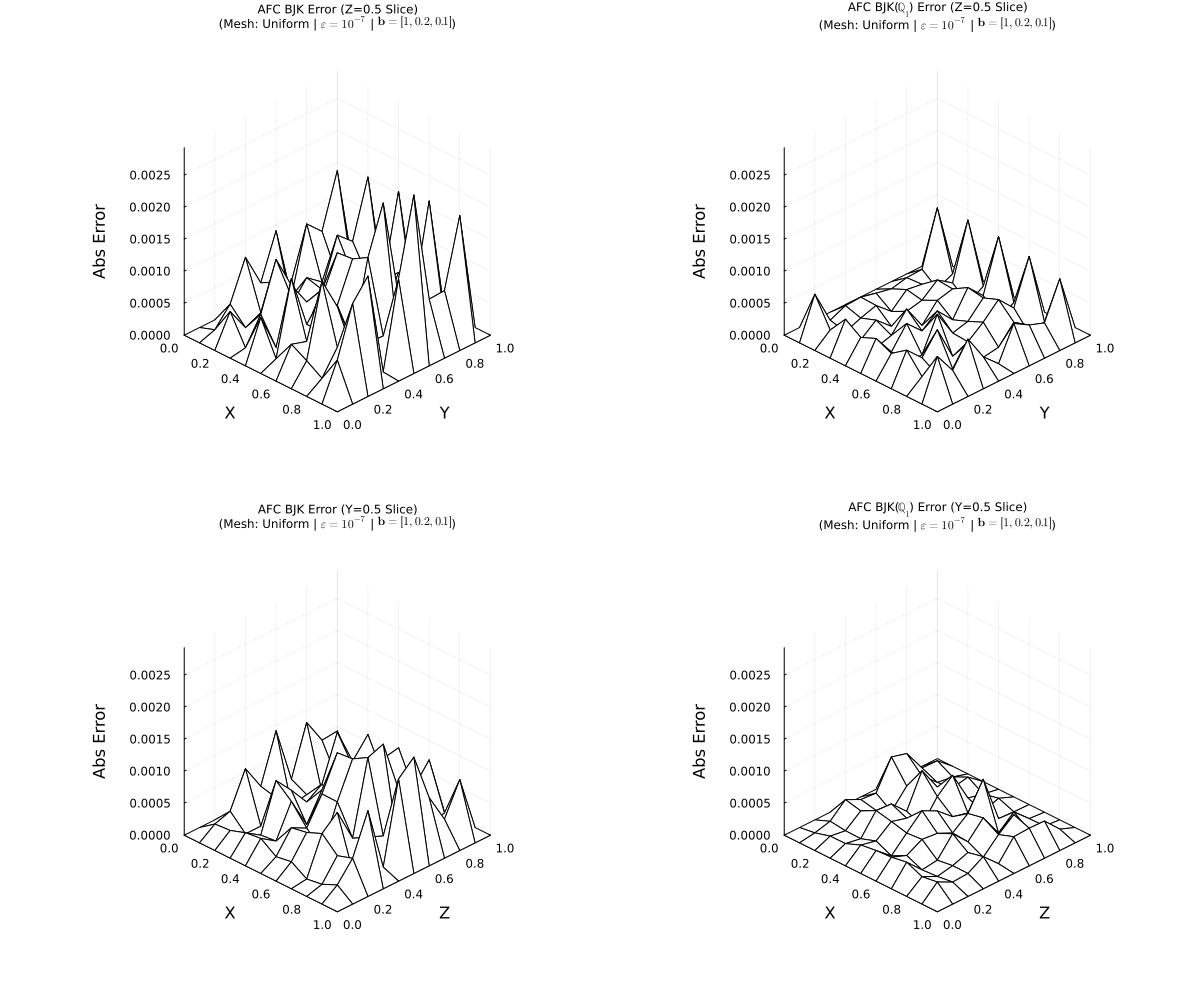}
	\caption{Example~\ref{subsec:3D}: Absolute error of the BJK and BJK($ \mathbb{Q}_1 $) limiters for the exact solution \eqref{eq:exact_3d} on the cut planes $z = 0.5$ (top) and $y = 0.5$ (bottom).}
	\label{fig:3d_error}
\end{figure}

Table~\ref{tab:3d_convergence} reports errors and empirical orders of convergence
(EOC) for the BJK and the BJK($ \mathbb{Q}_1 $) limiter on uniform grids ranging from $10\times10\times10$ to $60\times60\times60$ cells. The
error in the $L^2$ norm converges with order $2$ and the error in the $H^1$ norm with
order $1$, i.e.\ with the optimal rates for $\mathbb{Q}_1$ elements for both of them. In all cases, the BJK($ \mathbb{Q}_1 $) error is less than that of BJK.

\begin{table}[htbp]
	\centering
	\footnotesize
	\setlength{\tabcolsep}{3pt}
	\begin{tabular}{|c|c|c|c|c|c|c|c|c|c|}
		\hline
		Mesh & $h$ & BJK($ \mathbb{Q}_1 $) ($L^2$) & EOC & BJK ($L^2$) & EOC & BJK($ \mathbb{Q}_1 $) ($H^1$) & EOC & BJK ($H^1$) & EOC \\
		\hline
		$10\times10\times10$ & 0.1000 & 1.0853e-04 & -    & 2.1138e-04 & -    & 4.2537e-03 & -    & 5.5149e-03 & -    \\ \hline
		$15\times15\times15$ & 0.0667 & 3.8941e-05 & 2.53 & 5.7814e-05 & 3.20 & 2.4362e-03 & 1.37 & 2.7969e-03 & 1.67 \\ \hline
		$20\times20\times20$ & 0.0500 & 1.7930e-05 & 2.70 & 2.2964e-05 & 3.21 & 1.7332e-03 & 1.18 & 1.8181e-03 & 1.50 \\ \hline
		$25\times25\times25$ & 0.0400 & 9.9291e-06 & 2.65 & 1.2622e-05 & 2.68 & 1.3567e-03 & 1.10 & 1.3915e-03 & 1.20 \\ \hline
		$30\times30\times30$ & 0.0333 & 6.7117e-06 & 2.15 & 7.9193e-06 & 2.56 & 1.1213e-03 & 1.05 & 1.1350e-03 & 1.12 \\ \hline
		$35\times35\times35$ & 0.0286 & 4.9306e-06 & 2.00 & 5.1664e-06 & 2.77 & 9.5846e-04 & 1.02 & 9.6443e-04 & 1.06 \\ \hline
		$40\times40\times40$ & 0.0250 & 3.7156e-06 & 2.12 & 3.8764e-06 & 2.15 & 8.3634e-04 & 1.02 & 8.3980e-04 & 1.04 \\ \hline
		$45\times45\times45$ & 0.0222 & 2.9274e-06 & 2.02 & 3.0064e-06 & 2.16 & 7.4308e-04 & 1.00 & 7.4465e-04 & 1.02 \\ \hline
		$50\times50\times50$ & 0.0200 & 2.3615e-06 & 2.04 & 2.4013e-06 & 2.13 & 6.6794e-04 & 1.01 & 6.6915e-04 & 1.01 \\ \hline
		$55\times55\times55$ & 0.0182 & 1.9523e-06 & 2.00 & 1.9711e-06 & 2.07 & 6.0718e-04 & 1.00 & 6.0771e-04 & 1.01 \\ \hline
		$60\times60\times60$ & 0.0167 & 1.6363e-06 & 2.03 & 1.6473e-06 & 2.06 & 5.5622e-04 & 1.01 & 5.5673e-04 & 1.01 \\ \hline
	\end{tabular}
	\caption{Example~\ref{subsec:3D}: Errors and empirical orders of convergence for the
		BJK and BJK($\mathbb{Q}_1$) limiters.}
	\label{tab:3d_convergence}
\end{table}

\subsection{A Problem with an Interior Layer}\label{subsec:layer}

The following example goes back to Hughes et al.\ \cite{HMM86}. It is posed on
$\Omega = (0,1)^2$ with
$\boldsymbol{b} = (\cos(\pi/4), \sin(\pi/4))^{\top}$, $\sigma = 0$, $f = 0$,
$\varepsilon = 10^{-6}$, and the Dirichlet boundary condition
\begin{equation}
	u_{\mathrm{D}} =
	\begin{cases}
		1 & \text{if } (y = 1 \wedge x > 0) \text{ or } (x = 0 \wedge y > 0), \\
		0 & \text{otherwise.}
	\end{cases}
	\label{eq:layer_bc}
\end{equation}
The boundary datum is discontinuous at the origin, and the solution possesses a sharp
interior layer emanating from it along the direction of $\boldsymbol{b}$, together
with exponential layers at the outflow boundary. Fig.~\ref{fig:layer_profile} shows
the profile computed with the BJK($ \mathbb{Q}_1 $) limiter.

\begin{figure}[tbp]
	\centering
	\includegraphics[width=0.7\textwidth]{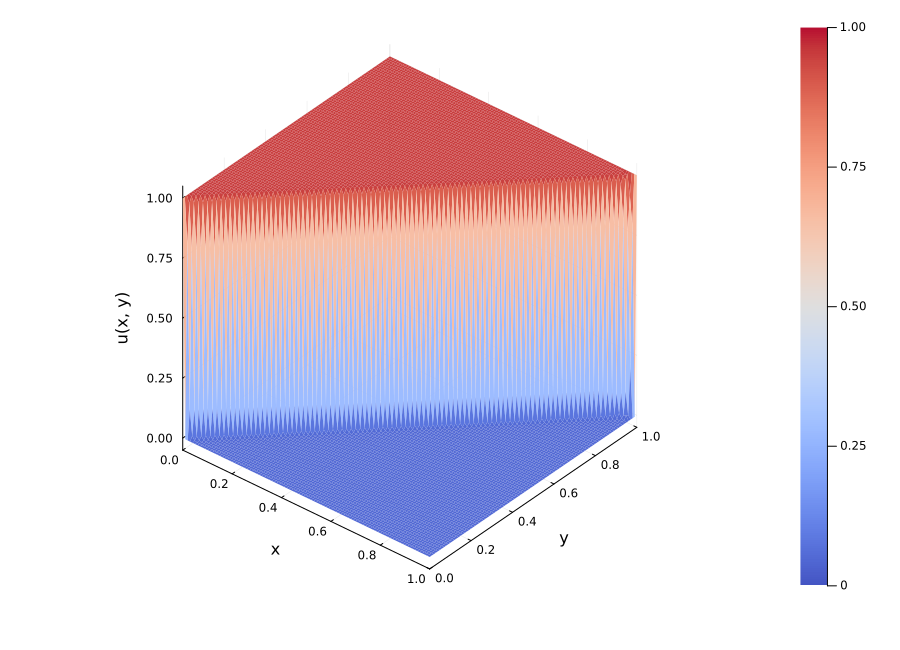}
	\caption{Example~\ref{subsec:layer}: Solution profile computed with the BJK($ \mathbb{Q}_1 $) limiter on a $100 \times 100$ grid.}
	\label{fig:layer_profile}
\end{figure}

The exact solution is not known, so the quality of a discretization is assessed by how
sharply it resolves the interior layer. First we will work on uniform grids. 

\subsubsection{Uniform Grids}

 Fig.~\ref{fig:cutline} (left) shows the cutline at
$y = 0.5$ for five schemes on a $65 \times 65$ grid, and Fig.~\ref{fig:cutline} (right)
magnifies the layer region. The Galerkin solution oscillates, as expected. The
remaining four curves are, in order of decreasing smearing, the low-order scheme
(obtained by setting $\alpha_{ij} \equiv 0$), the limiter of Kuzmin
\cite{Ku09}, the BJK limiter, and the BJK($ \mathbb{Q}_1 $) limiter.

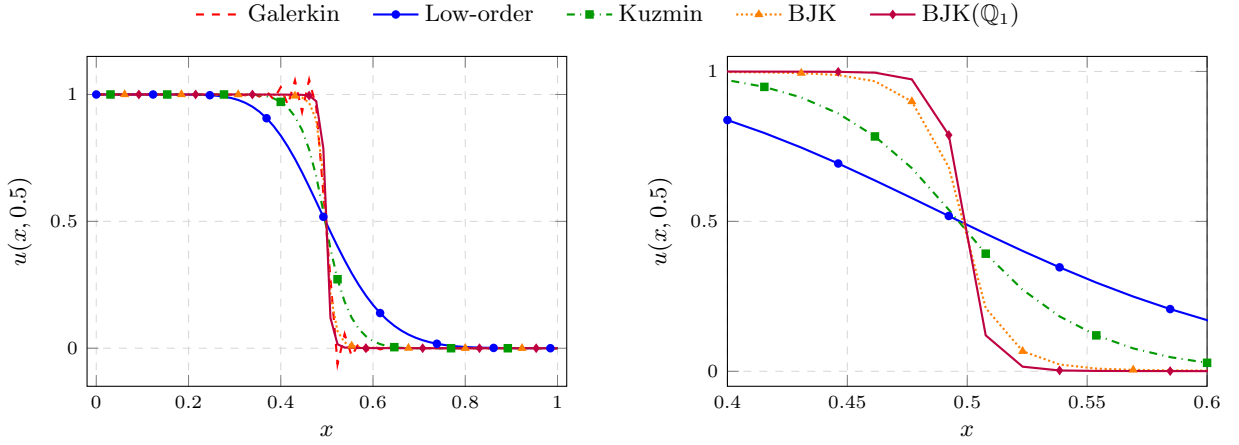
\begin{figure}[tbp]
\centering
\pgfplotslegendfromname{sharedlegend}\\[1.5ex]

\begin{minipage}[t]{0.48\textwidth}
	\centering
	\begin{tikzpicture}
		\begin{axis}[
			xlabel={$x$},
			ylabel={$u(x,0.5)$},
			xmin=-0.02, xmax=1.02,
			ymin=-0.15, ymax=1.15,
			plotpair,
			grid=major,
			grid style={dashed, gray!30},
			legend to name=sharedlegend,
			legend columns=-1,
			legend style={
				draw=none,
				fill=none,
				font=\small,
				/tikz/every even column/.append style={column sep=10pt}
			},
			tick label style={font=\scriptsize},
			label style={font=\small},
			set layers,
			mark layer=axis foreground
			]
			\addplot[color=red, dashed, thick] 
			table[x=x, y=Galerkin] {\cutlinedata};
			\addlegendentry{Galerkin}
			
			\addplot[
			color=blue, 
			solid, 
			thick, 
			mark=*, 
			mark size=1.2pt, 
			mark repeat=8, 
			mark phase=1,
			mark options={solid}
			] table[x=x, y=LowOrder] {\cutlinedata};
			\addlegendentry{Low-order}
			
			\addplot[
			color=green!60!black, 
			dashdotted, 
			thick, 
			mark=square*, 
			mark size=1.2pt, 
			mark repeat=8, 
			mark phase=3,
			mark options={solid}
			] table[x=x, y=Kuzmin] {\cutlinedata};
			\addlegendentry{Kuzmin}
			
			\addplot[
			color=orange, 
			densely dotted, 
			thick, 
			mark=triangle*, 
			mark size=1.2pt, 
			mark repeat=8, 
			mark phase=5,
			mark options={solid}
			] table[x=x, y=BJK]{\cutlinedata};
			\addlegendentry{BJK}
			
			\addplot[
			color=purple, 
			solid, 
			thick, 
			mark=diamond*, 
			mark size=1.2pt, 
			mark repeat=8, 
			mark phase=7,
			mark options={solid}
			] table[x=x, y=BJK()]{\cutlinedata};
			\addlegendentry{BJK($ \mathbb{Q}_1 $)}
		\end{axis}
	\end{tikzpicture}
\end{minipage}%
\hfill
\begin{minipage}[t]{0.48\textwidth}
	\centering
	\begin{tikzpicture}
		\begin{axis}[
			xlabel={$x$},
			ylabel={$u(x,0.5)$},
			xmin=0.40, xmax=0.60,
			ymin=-0.05, ymax=1.05,
			plotpair,
			grid=major,
			grid style={dashed, gray!30},
			tick label style={font=\scriptsize},
			label style={font=\small},
			set layers,
			mark layer=axis foreground
			]
			\addplot[
			color=blue, 
			solid, 
			thick, 
			mark=*, 
			mark size=1.2pt, 
			mark repeat=3, 
			mark phase=1,
			mark options={solid}
			] table[x=x, y=LowOrder]{\cutlinezoomeddata};
			
			\addplot[
			color=green!60!black, 
			dashdotted, 
			thick, 
			mark=square*, 
			mark size=1.2pt, 
			mark repeat=3, 
			mark phase=2,
			mark options={solid}
			] table[x=x, y=Kuzmin]{\cutlinezoomeddata};
			
			\addplot[
			color=orange, 
			densely dotted, 
			thick, 
			mark=triangle*, 
			mark size=1.2pt, 
			mark repeat=3, 
			mark phase=3,
			mark options={solid}
			] table[x=x, y=BJK] {\cutlinezoomeddata};
			
			\addplot[
			color=purple, 
			solid, 
			thick, 
			mark=diamond*, 
			mark size=1.2pt, 
			mark repeat=3, 
			mark phase=4,
			mark options={solid}
			] table[x=x, y=BJK()]{\cutlinezoomeddata};
		\end{axis}
	\end{tikzpicture}
\end{minipage}

	\caption{Example~\ref{subsec:layer}: Cutline at $y=0.5$ for the different schemes on a
	$65\times65$ grid (left), with a magnified view of the interior layer (right).
	The Galerkin solution is omitted from the magnified view for clarity.}
	\label{fig:cutline}
\end{figure}

Following \cite{JJK24}, we quantify the smearing of the layer by the width of the
transition region on the cutline,
\begin{equation}
	\mathrm{Smear} = x_2 - x_1,
	\label{eq:smear}
\end{equation}
where $x_1$ is the smallest $x$ at which the computed solution on the cutline
$y = 0.5$ falls below $0.9$ and $x_2$ is the smallest $x$ at which it falls below
$0.1$. Fig.~\ref{fig:smear} is a $\log$--$\log$ plot of $\mathrm{Smear}$ against the
number of degrees of freedom, for grids from $10 \times 10$ to $100 \times 100$.

The four schemes reduce the smearing in the order low-order, Kuzmin, BJK, BJK($ \mathbb{Q}_1 $), with the
BJK($ \mathbb{Q}_1 $) limiter giving the smallest transition width on every grid tested.

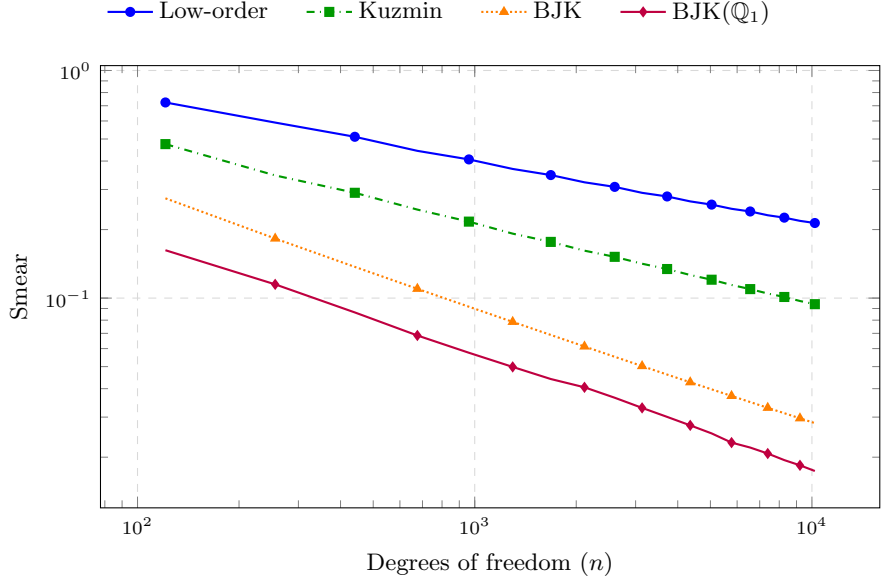
\begin{figure}[tbp]
\centering
\pgfplotslegendfromname{smearing_legend}\\[1.5ex]
\begin{tikzpicture}
	\begin{loglogaxis}[
		xlabel={Degrees of freedom ($n$)},
		ylabel={$\mathrm{Smear}$},
		plotsingle,
		grid=major,
		grid style={dashed, gray!30},
		legend to name=smearing_legend,
		legend columns=-1,
		legend style={
			draw=none,
			fill=none,
			font=\small,
			/tikz/every even column/.append style={column sep=12pt}
		},
		tick label style={font=\scriptsize},
		label style={font=\small},
		log basis x={10},
		log basis y={10},			
		set layers,
		mark layer=axis foreground
		]
		\addplot[
		color=blue, 
		solid, 
		thick, 
		mark=*, 
		mark size=1.5pt, 
		mark repeat=2, 
		mark phase=1,
		mark options={solid}
		] table[x=DOFs, y=LowOrder] {\convergencedata};
		\addlegendentry{Low-order}
		
		\addplot[
		color=green!60!black, 
		dashdotted, 
		thick, 
		mark=square*, 
		mark size=1.5pt, 
		mark repeat=2, 
		mark phase=1,
		mark options={solid}
		] table[x=DOFs, y=Kuzmin] {\convergencedata};
		\addlegendentry{Kuzmin}
		
		\addplot[
		color=orange, 
		densely dotted, 
		thick, 
		mark=triangle*, 
		mark size=1.5pt, 
		mark repeat=2, 
		mark phase=2,
		mark options={solid}
		] table[x=DOFs, y=BJK] {\convergencedata};
		\addlegendentry{BJK}
		
		\addplot[
		color=purple, 
		solid, 
		thick, 
		mark=diamond*, 
		mark size=1.5pt, 
		mark repeat=2, 
		mark phase=2,
		mark options={solid}
		] table[x=DOFs, y=BJK()] {\convergencedata};
		\addlegendentry{BJK($ \mathbb{Q}_1 $)}
	\end{loglogaxis}
\end{tikzpicture}
	\caption{Example~\ref{subsec:layer}: Width of the smeared interior layer, measured by Eq.~\eqref{eq:smear},
		against the number of degrees of freedom.}
	\label{fig:smear}
\end{figure}

\subsubsection{Non-Uniform Grids}
We now solve the same problem on the non-uniform axis-aligned grids. Fig.~\ref{fig:cutlinenu} shows the cutline at
$y = 0.5$ on a $65 \times 65$ graded grid, and Fig.~\ref{fig:smearnu} the
corresponding smearing against the number of degrees of freedom.

The ordering of the four schemes is the same as on the uniform grid, and the absolute smearing is smaller for all of them, the grading placing more mesh points near the layer. The two limiters BJK and BJK ($\mathbb{Q}_1$) are, however, much closer to each other here than might be expected, and the reason is the one recorded there: on a grid graded by $x_k = (k/n)^p$ the
parameter $\mu_i$ equals its uniform value $3$ at all but $O(n)$ nodes, and the
few nodes carrying the larger values $2^{p+1}-1$ and $2^{2p}-1$ lie along the
graded boundaries, away from the interior layer. A grading of this type therefore
changes the amount of limiting only in a region where the layer is absent. This
observation is pursued in Section~\ref{subsec:musens}.

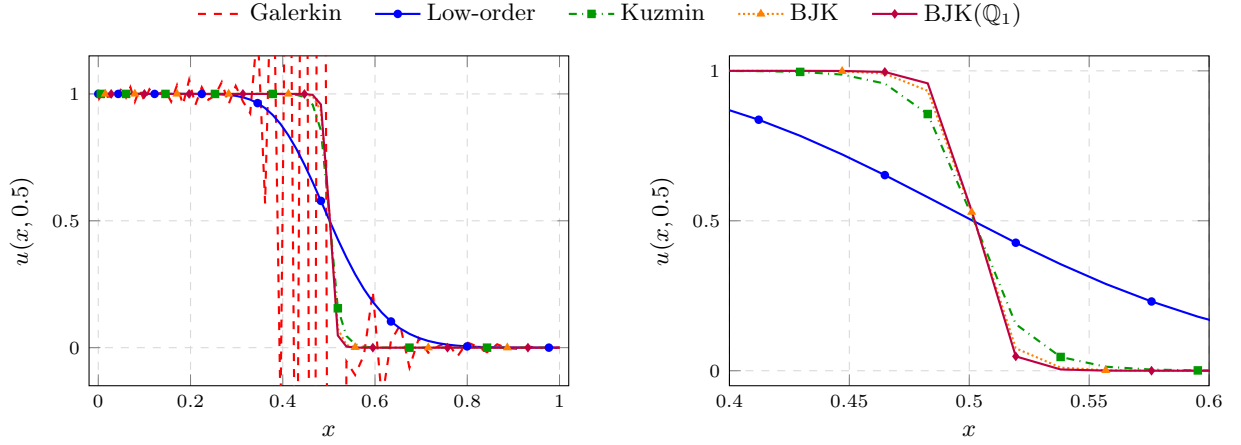
\begin{figure}[tbp]
	\centering
	\pgfplotslegendfromname{sharedlegend_nonuniform}\\[1.5ex]
	
	\begin{minipage}[t]{0.48\textwidth}
		\centering
		\begin{tikzpicture}
			\begin{axis}[
				xlabel={$x$},
				ylabel={$u(x,0.5)$},
				xmin=-0.02, xmax=1.02,
				ymin=-0.15, ymax=1.15,
				plotpair,
				grid=major,
				grid style={dashed, gray!30},
				legend to name=sharedlegend_nonuniform,
				legend columns=-1,
				legend style={
					draw=none,
					fill=none,
					font=\small,
					/tikz/every even column/.append style={column sep=10pt}
				},
				tick label style={font=\scriptsize},
				label style={font=\small},
				set layers,
				mark layer=axis foreground
				]
				\addplot[color=red, dashed, thick] 
				table[x=x, y=Galerkin] {\cutlinedatanu};
				\addlegendentry{Galerkin}
				
				\addplot[
				color=blue, 
				solid, 
				thick, 
				mark=*, 
				mark size=1.2pt, 
				mark repeat=8, 
				mark phase=1,
				mark options={solid}
				] table[x=x, y=LowOrder] {\cutlinedatanu};
				\addlegendentry{Low-order}
				
				\addplot[
				color=green!60!black, 
				dashdotted, 
				thick, 
				mark=square*, 
				mark size=1.2pt, 
				mark repeat=8, 
				mark phase=3,
				mark options={solid}
				] table[x=x, y=Kuzmin] {\cutlinedatanu};
				\addlegendentry{Kuzmin}
				
				\addplot[
				color=orange, 
				densely dotted, 
				thick, 
				mark=triangle*, 
				mark size=1.2pt, 
				mark repeat=8, 
				mark phase=5,
				mark options={solid}
				] table[x=x, y=BJK]{\cutlinedatanu};
				\addlegendentry{BJK}
				
				\addplot[
				color=purple, 
				solid, 
				thick, 
				mark=diamond*, 
				mark size=1.2pt, 
				mark repeat=8, 
				mark phase=7,
				mark options={solid}
				] table[x=x, y=BJK()]{\cutlinedatanu};
				\addlegendentry{BJK($ \mathbb{Q}_1 $)}
			\end{axis}
		\end{tikzpicture}
	\end{minipage}%
	\hfill
	\begin{minipage}[t]{0.48\textwidth}
		\centering
		\begin{tikzpicture}
			\begin{axis}[
				xlabel={$x$},
				ylabel={$u(x,0.5)$},
				xmin=0.40, xmax=0.60,
				ymin=-0.05, ymax=1.05,
				plotpair,
				grid=major,
				grid style={dashed, gray!30},
				tick label style={font=\scriptsize},
				label style={font=\small},
				set layers,
				mark layer=axis foreground
				]
				\addplot[
				color=blue, 
				solid, 
				thick, 
				mark=*, 
				mark size=1.2pt, 
				mark repeat=3, 
				mark phase=1,
				mark options={solid}
				] table[x=x, y=LowOrder]{\cutlinezoomeddatanu};
				
				\addplot[
				color=green!60!black, 
				dashdotted, 
				thick, 
				mark=square*, 
				mark size=1.2pt, 
				mark repeat=3, 
				mark phase=2,
				mark options={solid}
				] table[x=x, y=Kuzmin]{\cutlinezoomeddatanu};
				
				\addplot[
				color=orange, 
				densely dotted, 
				thick, 
				mark=triangle*, 
				mark size=1.2pt, 
				mark repeat=3, 
				mark phase=3,
				mark options={solid}
				] table[x=x, y=BJK] {\cutlinezoomeddatanu};
				
				\addplot[
				color=purple, 
				solid, 
				thick, 
				mark=diamond*, 
				mark size=1.2pt, 
				mark repeat=3, 
				mark phase=4,
				mark options={solid}
				] table[x=x, y=BJK()]{\cutlinezoomeddatanu};
			\end{axis}
		\end{tikzpicture}
	\end{minipage}
	
	\caption{Example~\ref{subsec:layer}: Cutline at $y=0.5$ for the different schemes on a
		$65\times65$ non-uniform axis-aligned grid (left), with a magnified view of the
		interior layer (right). The Galerkin solution is omitted from the magnified
		view for clarity.}
	\label{fig:cutlinenu}
\end{figure}
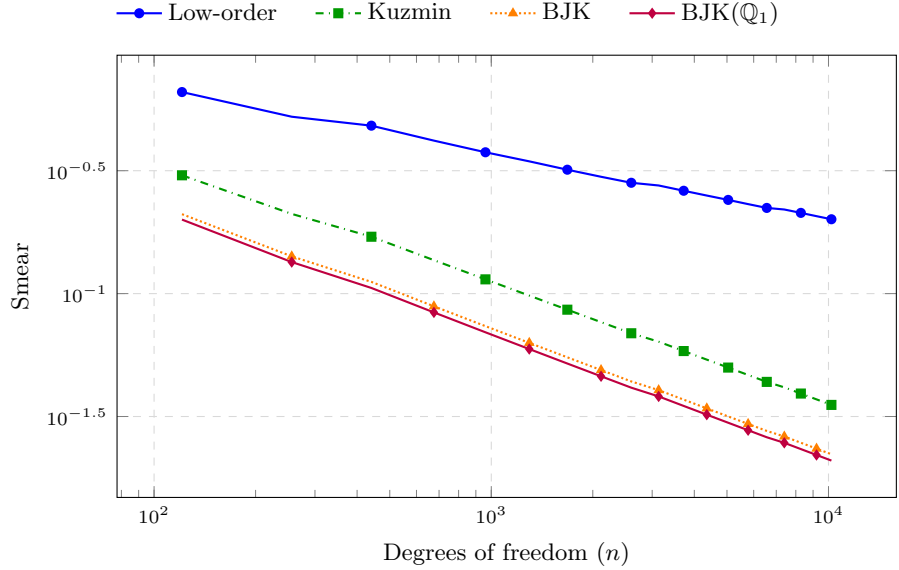
\begin{figure}[tbp]
	\centering
	\pgfplotslegendfromname{smearing_legend_uniform}\\[1.5ex]
	\begin{tikzpicture}
		\begin{loglogaxis}[
			xlabel={Degrees of freedom ($n$)},
			ylabel={$\mathrm{Smear}$},
			plotsingle,
			grid=major,
			grid style={dashed, gray!30},
			legend to name=smearing_legend_uniform,
			legend columns=-1,
			legend style={
				draw=none,
				fill=none,
				font=\small,
				/tikz/every even column/.append style={column sep=12pt}
			},
			tick label style={font=\scriptsize},
			label style={font=\small},
			log basis x={10},
			log basis y={10},			
			set layers,
			mark layer=axis foreground
			]
			\addplot[
			color=blue, 
			solid, 
			thick, 
			mark=*, 
			mark size=1.5pt, 
			mark repeat=2, 
			mark phase=1,
			mark options={solid}
			] table[x=DOFs, y=LowOrder] {\convergencedatanu};
			\addlegendentry{Low-order}
			
			\addplot[
			color=green!60!black, 
			dashdotted, 
			thick, 
			mark=square*, 
			mark size=1.5pt, 
			mark repeat=2, 
			mark phase=1,
			mark options={solid}
			] table[x=DOFs, y=Kuzmin] {\convergencedatanu};
			\addlegendentry{Kuzmin}
			
			\addplot[
			color=orange, 
			densely dotted, 
			thick, 
			mark=triangle*, 
			mark size=1.5pt, 
			mark repeat=2, 
			mark phase=2,
			mark options={solid}
			] table[x=DOFs, y=BJK] {\convergencedatanu};
			\addlegendentry{BJK}
			
			\addplot[
			color=purple, 
			solid, 
			thick, 
			mark=diamond*, 
			mark size=1.5pt, 
			mark repeat=2, 
			mark phase=2,
			mark options={solid}
			] table[x=DOFs, y=BJK()] {\convergencedatanu};
			\addlegendentry{BJK($ \mathbb{Q}_1 $)}
		\end{loglogaxis}
	\end{tikzpicture}
	\caption{Example~\ref{subsec:layer}: Width of the smeared interior layer, measured by Eq.~\eqref{eq:smear},
		against the number of degrees of freedom for the non-uniform axis-aligned grid.}
	\label{fig:smearnu}
\end{figure}

\subsubsection{$\mu_i$ Sensitivity}
\label{subsec:musens}
The proposed limiter BJK($\mathbb{Q}_1$) differs from the BJK limiter only in the
factors $\mu_i$, and to accommodate bilinearity preservation the proposed $\mu_i$
is larger than the one sufficient for linearity preservation \cite{BJK17}. It is
therefore necessary to ask how much of the behaviour observed above is a
consequence of bilinearity preservation and how much simply of a larger parameter.
In this part we run the BJK limiter with a range of constant values of $\mu_i$ and
compare the results.

Fig.~\ref{fig:cutline_mu} shows the cutline at $y = 0.5$ on the $65 \times 65$
uniform mesh. Larger values of $\mu_i$ resolve the layer slightly better, and
Fig.~\ref{fig:smear_mu} confirms that the smearing decreases monotonically with
$\mu_i$ on every grid tested. The differences are small, and they are so for a
structural reason: $\mu_i$ enters the limiter only through the products
$\mu_i Q_i^{\pm}$ in Eq.~\eqref{eq:Rpm}, and at a strict local extremum
$Q_i^{+} = q_i(u_i - u_i^{\max}) = 0$, so that $R_i^{+} = 0$ for every
$\mu_i > 0$. The parameter is thus inactive precisely at the nodes at which the
limiter acts in the neighbourhood of a layer, which is also the mechanism behind
Theorem~\ref{Thm:DMP_BJK}.

Increasing $\mu_i$ is, on the other hand, not free. Fig.~\ref{fig:iterations_mu}
reports the number of iterations of the damped fixed point method needed to reach
the tolerance $10^{-10}$. Up to $\mu_i = 5$ the iteration converges on every grid,
with the count growing moderately with $\mu_i$; at $\mu_i = 10$ it stagnates on the
finer grids, and at $\mu_i = 50$ and $\mu_i = 200$ it does not reach the tolerance at
all within the prescribed number of iterations. The range of parameters for which
the nonlinear problem is solvable by this iteration is therefore narrow, and the
value $\mu_i = 3$ furnished by Theorem~\ref{Thm:Explicit_bound} on a uniform grid
lies inside it, close to its upper end. In this sense the value required for
bilinearity preservation is not merely admissible but close to the largest value
that can be used in practice. We emphasize that this is a statement about the
fixed point iteration of \cite{JJ19} with dynamic damping, which is the
method recommended for this class of problems, and not about the solvability of
Eq.~\eqref{eq:AFC_Sys} itself, which is guaranteed for every $\mu_i > 0$.

\begin{figure}[tbp]
	\centering
	\pgfplotslegendfromname{sharedlegend_mu}\\[1.5ex]
	
	\begin{minipage}[t]{0.48\textwidth}
	\centering
	\begin{tikzpicture}
		\begin{axis}[
			xlabel={$x$},
			ylabel={$u(x,0.5)$},
			xmin=-0.02, xmax=1.02,
			ymin=-0.05, ymax=1.05,
			plotpair,
			grid=major,
			grid style={dashed, gray!30},
			legend to name=sharedlegend_mu,
			legend columns=4,
			legend style={
				draw=none,
				fill=none,
				font=\small,
				/tikz/every even column/.append style={column sep=10pt}
			},
			tick label style={font=\scriptsize},
			label style={font=\small},
			set layers,
			mark layer=axis foreground
			]
			\addplot[
			color=blue, 
			solid, 
			semithick, 
			mark=*, 
			mark size=1.0pt, 
			mark repeat=6, 
			mark phase=1,
			mark options={solid}
			] table[x=x, y=mu_1] {\cutlinedatamu};
			\addlegendentry{$\mu_i = 1$}
			
			\addplot[
			color=red, 
			dashed, 
			semithick, 
			mark=square*, 
			mark size=1.0pt, 
			mark repeat=6, 
			mark phase=2,
			mark options={solid}
			] table[x=x, y={mu_1.41421}] {\cutlinedatamu};
			\addlegendentry{$\mu_i = \sqrt{2}$}
			
			\addplot[
			color=green!60!black, 
			dashdotted, 
			semithick, 
			mark=triangle*, 
			mark size=1.0pt, 
			mark repeat=6, 
			mark phase=3,
			mark options={solid}
			] table[x=x, y=mu_3] {\cutlinedatamu};
			\addlegendentry{$\mu_i = 3$}
			
			\addplot[
			color=orange, 
			densely dotted, 
			semithick, 
			mark=diamond*, 
			mark size=1.0pt, 
			mark repeat=6, 
			mark phase=4,
			mark options={solid}
			] table[x=x, y=mu_5] {\cutlinedatamu};
			\addlegendentry{$\mu_i = 5$}
			
			\addplot[
			color=purple, 
			solid, 
			semithick, 
			mark=pentagon*, 
			mark size=1.0pt, 
			mark repeat=6, 
			mark phase=5,
			mark options={solid}
			] table[x=x, y=mu_10] {\cutlinedatamu};
			\addlegendentry{$\mu_i = 10$}
			
			\addplot[
			color=teal, 
			dashed, 
			semithick, 
			mark=star, 
			mark size=1.0pt, 
			mark repeat=6, 
			mark phase=6,
			mark options={solid}
			] table[x=x, y=mu_50] {\cutlinedatamu};
			\addlegendentry{$\mu_i = 50$}
			
			\addplot[
			color=magenta, 
			solid, 
			semithick, 
			mark=x, 
			mark size=1.0pt, 
			mark repeat=6, 
			mark phase=1,
			mark options={solid}
			] table[x=x, y=mu_200] {\cutlinedatamu};
			\addlegendentry{$\mu_i = 200$}
		\end{axis}
	\end{tikzpicture}
	\end{minipage}%
	\hfill
	\begin{minipage}[t]{0.48\textwidth}
	\centering
	\begin{tikzpicture}
		\begin{axis}[
			xlabel={$x$},
			ylabel={$u(x,0.5)$},
			xmin=0.40, xmax=0.60,
			ymin=-0.05, ymax=1.05,
			plotpair,
			grid=major,
			grid style={dashed, gray!30},
			tick label style={font=\scriptsize},
			label style={font=\small},
			set layers,
			mark layer=axis foreground
			]
			\addplot[
			color=blue, 
			solid, 
			semithick, 
			mark=*, 
			mark size=1.0pt, 
			mark repeat=2, 
			mark phase=1,
			mark options={solid}
			] table[x=x, y=mu_1] {\cutlinezoomeddatamu};
			
			\addplot[
			color=red, 
			dashed, 
			semithick, 
			mark=square*, 
			mark size=1.0pt, 
			mark repeat=2, 
			mark phase=1,
			mark options={solid}
			] table[x=x, y={mu_1.41421}] {\cutlinezoomeddatamu};
			
			\addplot[
			color=green!60!black, 
			dashdotted, 
			semithick, 
			mark=triangle*, 
			mark size=1.0pt, 
			mark repeat=2, 
			mark phase=2,
			mark options={solid}
			] table[x=x, y=mu_3] {\cutlinezoomeddatamu};
			
			\addplot[
			color=orange, 
			densely dotted, 
			semithick, 
			mark=diamond*, 
			mark size=1.0pt, 
			mark repeat=2, 
			mark phase=2,
			mark options={solid}
			] table[x=x, y=mu_5] {\cutlinezoomeddatamu};
			
			\addplot[
			color=purple, 
			solid, 
			semithick, 
			mark=pentagon*, 
			mark size=1.0pt, 
			mark repeat=2, 
			mark phase=1,
			mark options={solid}
			] table[x=x, y=mu_10] {\cutlinezoomeddatamu};
			
			\addplot[
			color=teal, 
			dashed, 
			semithick, 
			mark=star, 
			mark size=1.0pt, 
			mark repeat=2, 
			mark phase=2,
			mark options={solid}
			] table[x=x, y=mu_50] {\cutlinezoomeddatamu};
			
			\addplot[
			color=magenta, 
			solid, 
			semithick, 
			mark=x, 
			mark size=1.0pt, 
			mark repeat=2, 
			mark phase=1,
			mark options={solid}
			] table[x=x, y=mu_200] {\cutlinezoomeddatamu};
		\end{axis}
	\end{tikzpicture}
	\end{minipage}
	\caption{Example~\ref{subsec:layer}: Cutline profiles at $y=0.5$ for various values of $\mu_i$ over the full domain (left) and magnified interior layer region $x \in [0.40, 0.60]$ (right).}
	\label{fig:cutline_mu}
\end{figure}
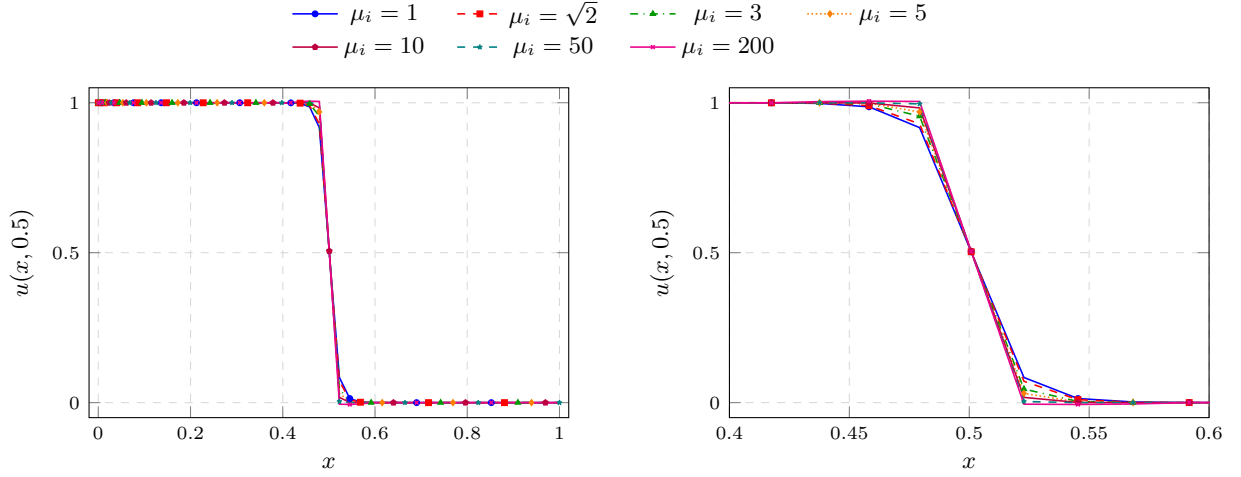

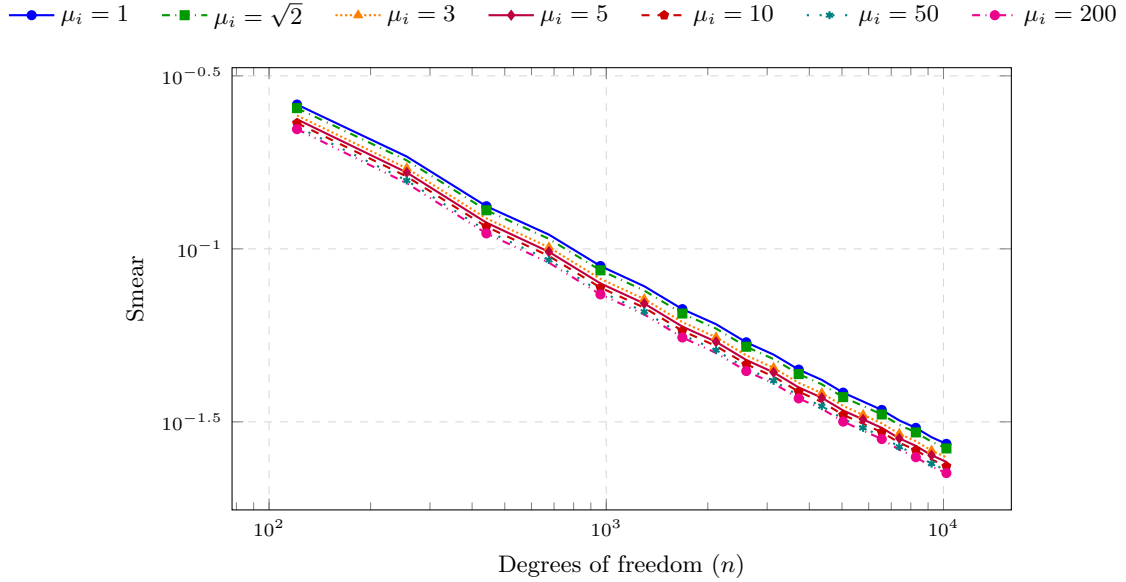
\begin{figure}[tbp]
	\centering
	\pgfplotslegendfromname{mu_dependency_legend}\\[1.5ex]
	\begin{tikzpicture}
		\begin{loglogaxis}[
			xlabel={Degrees of freedom ($n$)},
			ylabel={$\mathrm{Smear}$},
			plotsingle,
			grid=major,
			grid style={dashed, gray!30},
			legend to name=mu_dependency_legend,
			legend columns=-1,
			legend style={
				draw=none,
				fill=none,
				font=\small,
				/tikz/every even column/.append style={column sep=10pt}
			},
			tick label style={font=\scriptsize},
			label style={font=\small},
			log basis x={10},
			log basis y={10},			
			set layers,
			mark layer=axis foreground
			]
			\addplot[
			color=blue, 
			solid, 
			thick, 
			mark=*, 
			mark size=1.5pt, 
			mark repeat=2, 
			mark phase=1,
			mark options={solid}
			] table[x=DOFs, y=mu_1] {\convergencedatamudependency};
			\addlegendentry{$\mu_i=1$}
			
			\addplot[
			color=green!60!black, 
			dashdotted, 
			thick, 
			mark=square*, 
			mark size=1.5pt, 
			mark repeat=2, 
			mark phase=1,
			mark options={solid}
			] table[x=DOFs, y=mu_1_41421] {\convergencedatamudependency};
			\addlegendentry{$\mu_i=\sqrt{2}$}
			
			\addplot[
			color=orange, 
			densely dotted, 
			thick, 
			mark=triangle*, 
			mark size=1.5pt, 
			mark repeat=2, 
			mark phase=2,
			mark options={solid}
			] table[x=DOFs, y=mu_3] {\convergencedatamudependency};
			\addlegendentry{$\mu_i=3$}
			
			\addplot[
			color=purple, 
			solid, 
			thick, 
			mark=diamond*, 
			mark size=1.5pt, 
			mark repeat=2, 
			mark phase=2,
			mark options={solid}
			] table[x=DOFs, y=mu_5] {\convergencedatamudependency};
			\addlegendentry{$\mu_i=5$}
			
			\addplot[
			color=red!80!black, 
			dashed, 
			thick, 
			mark=pentagon*, 
			mark size=1.5pt, 
			mark repeat=2, 
			mark phase=1,
			mark options={solid}
			] table[x=DOFs, y=mu_10] {\convergencedatamudependency};
			\addlegendentry{$\mu_i=10$}
			
			\addplot[
			color=teal, 
			loosely dotted, 
			thick, 
			mark=asterisk, 
			mark size=1.5pt, 
			mark repeat=2, 
			mark phase=2,
			mark options={solid}
			] table[x=DOFs, y=mu_50] {\convergencedatamudependency};
			\addlegendentry{$\mu_i=50$}
			
			\addplot[
			color=magenta, 
			dashdotdotted, 
			thick, 
			mark=oplus*, 
			mark size=1.5pt, 
			mark repeat=2, 
			mark phase=1,
			mark options={solid}
			] table[x=DOFs, y=mu_200] {\convergencedatamudependency};
			\addlegendentry{$\mu_i=200$}
		\end{loglogaxis}
	\end{tikzpicture}
	\caption{Example~\ref{subsec:layer}: Width of the smeared interior layer plotted against the degrees of freedom for various values of $\mu_i$.}
	\label{fig:smear_mu}
\end{figure}

\begin{figure}[tbp]
	\centering
	\pgfplotslegendfromname{mu_dependency_legend}\\[1.5ex]
\begin{tikzpicture}
	\begin{semilogxaxis}[
		xlabel={Degrees of freedom ($n$)},
		ylabel={Iterations to convergence},
		plotsingle,
		grid=major,
		grid style={dashed, gray!30},
		log basis x={10},
		tick label style={font=\scriptsize},
		label style={font=\small},
		set layers,
		mark layer=axis foreground
		]
		\addplot[color=blue, solid, thick, mark=*, mark size=1.5pt,
		mark repeat=2, mark phase=1, mark options={solid}]
		table[x=DOFs, y=mu_1] {\dynamicdampingiterations};
		
		\addplot[color=green!60!black, dashdotted, thick, mark=square*, mark size=1.5pt,
		mark repeat=2, mark phase=1, mark options={solid}]
		table[x=DOFs, y=mu_sqrt2] {\dynamicdampingiterations};
		
		\addplot[color=orange, densely dotted, thick, mark=triangle*, mark size=1.5pt,
		mark repeat=2, mark phase=2, mark options={solid}]
		table[x=DOFs, y=mu_3] {\dynamicdampingiterations};
		
		\addplot[color=purple, solid, thick, mark=diamond*, mark size=1.5pt,
		mark repeat=2, mark phase=2, mark options={solid}]
		table[x=DOFs, y=mu_5] {\dynamicdampingiterations};
		
		\addplot[color=red!80!black, dashed, thick, mark=pentagon*, mark size=1.5pt,
		mark repeat=2, mark phase=1, mark options={solid}]
		table[x=DOFs, y=mu_10] {\dynamicdampingiterations};
		
		\addplot[color=teal, loosely dotted, thick, mark=asterisk, mark size=1.5pt,
		mark repeat=2, mark phase=2, mark options={solid}]
		table[x=DOFs, y=mu_50] {\dynamicdampingiterations};
		
		\addplot[color=magenta, dashdotdotted, thick, mark=oplus*, mark size=1.5pt,
		mark repeat=2, mark phase=1, mark options={solid}]
		table[x=DOFs, y=mu_200] {\dynamicdampingiterations};
	\end{semilogxaxis}
\end{tikzpicture}
	\caption{Example~\ref{subsec:layer}: Dynamic Damping Iterations to Convergence for various values of $\mu_i$.}
	\label{fig:iterations_mu}
\end{figure}
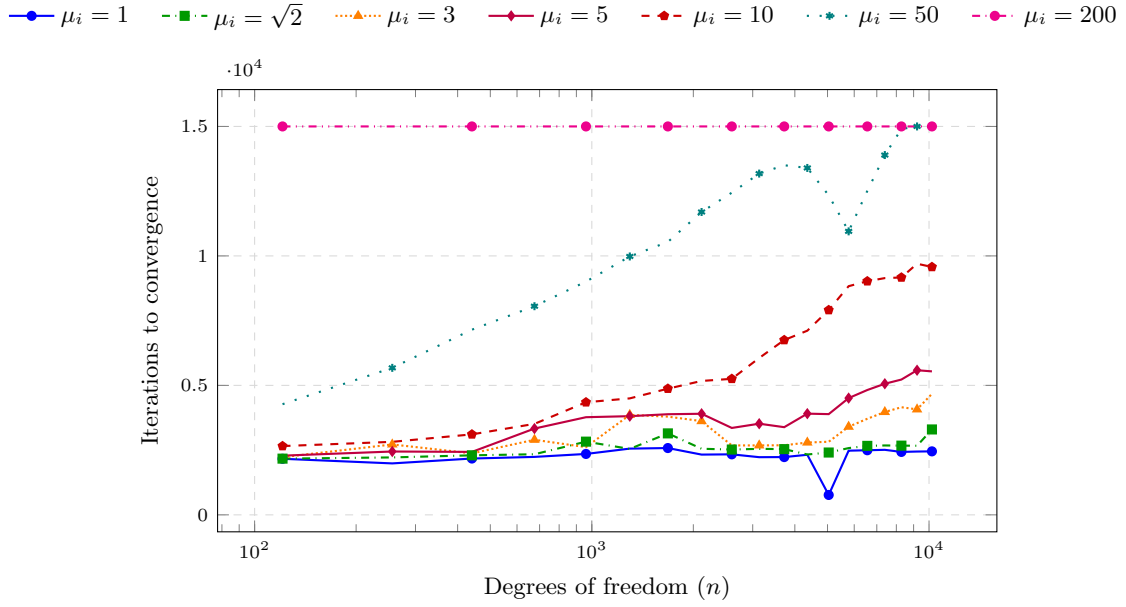

\section{Conclusions and Outlook}
\label{sec:conclusion}

The starting point of this work was the observation that linearity preservation, which
is the accuracy mechanism underlying the BJK limiter and its relatives, is not the
natural requirement on quadrilateral and hexahedral grids. On such grids the local
finite element space is $\mathbb{Q}_1$, which strictly contains $\mathbb{P}_1$, and the functions in
$\mathbb{Q}_1  $ are harmonic, satisfy the classical maximum principle, and on
axis-aligned grids lie in the finite element space. A limiter that switches off only
on affine functions therefore adds artificial diffusion on a class of functions that
the discretization reproduces exactly.

Removing this deficiency turned out to require a genuinely different argument from the
affine case. For affine functions, the value at an interior node lies between the
extreme values at the surrounding nodes as soon as the node lies in the convex hull of
its neighbours, so a finite parameter $\mu_i$ exists on every admissible patch. For
bilinear functions this is false, and the failure is not exotic: on a symmetric
4-Quad patch the function $-xy$ attains its largest value at the central
node. Bilinearity preservation is therefore attainable only on a restricted class of
grids, and identifying that class was a prerequisite for everything else.

Our answer is that the correct condition is a convexity condition, but one imposed
after lifting. A patch admits the bilinear maximum principle Eq.~\eqref{eq:BMP} if and only
if it is bilinearly convex, i.e.\ if the origin lies in the convex hull in $\R^3$ of
the lifted neighbouring nodes $(x_j - x_i,\, y_j - y_i,\, (x_j - x_i)(y_j - y_i))$
(Theorem~\ref{thm:BMPiffBC}). The condition is intrinsic to the geometry of the patch,
it involves no quantity from the discretization, and it can be verified by solving a
small linear feasibility problem. On any such patch a finite $\mu_i$ exists
(Proposition~\ref{Prop:Existence}) and is bounded by $\alpha_{\min}^{-1} - 1$, where
$\alpha_{\min}$ is the smallest of the convex weights that represent the origin. For
axis-aligned rectangular patches these weights are the ratios of the diametrically
opposite quadrant areas, and the bound collapses to the closed form
$\mu_i = (R_xR_y)^{-1} - 1$ in the linear eccentricities of the patch
(Theorem~\ref{Thm:Explicit_bound}); an explicit example shows that this value cannot be
lowered. The three-dimensional analogue $\mu_i = (R_xR_yR_z)^{-1} - 1$ follows with no
additional work (Remark~\ref{Rem:3D}). Since the resulting BJK($\mathbb{Q}_1$) limiter differs from the
BJK limiter only in the choice of $\mu_i$, it inherits the solvability, discrete
maximum principle and convergence results of \cite{BJK25} unchanged.

A separate result of the paper concerns representability rather than limiting. A mapped
$\mathbb{Q}_1$ space contains a prescribed physical bilinear function if and only if the
physical cells are rectangles with edges parallel to the coordinate axes
(Theorem~\ref{thm:exactrep}), and consequently the Galerkin solution of a problem with a
bilinear exact solution is exact precisely on axis-aligned grids
(Corollary~\ref{Cor:GalerkinExact}). This is an obstruction of the finite element
space, not of the stabilization, and it delimits what any limiter can achieve on
general quadrilateral grids; it is in the same spirit as the approximation-theoretic
limitations of mapped quadrilateral elements analysed in
\cite{ABF02}.

The numerical experiments confirm this picture on both sides. On axis-aligned grids,
uniform and graded alike, the BJK($\mathbb{Q}_1$) limiter reproduces the bilinear solution $u = 10xy$ to
the accuracy of the nonlinear iteration, while the BJK limiter produces an error of the
size of the discretization error; on randomly perturbed grids neither does, as
Corollary~\ref{Cor:Isoparametric} predicts. The three-dimensional example exhibits the
trilinear analogue and attains the optimal empirical orders $2$ in $L^2$ and $1$ in
$H^1$. On the interior-layer problem of
\cite{HMM86} the BJK($\mathbb{Q}_1$) limiter smears the layer less than the BJK limiter, the
limiter of \cite{Ku09} and the low-order scheme, on every grid tested.

It is worth being explicit about the mechanism behind the last observation. On a
uniform patch the BJK prescription gives $\mu_i = \sqrt{2}$ (and $1$ is sufficient
for locally symmetric meshes \cite{BJK17}) whereas ours gives $\mu_i = 3$, so the
BJK($\mathbb{Q}_1$) limiter is the \emph{less} restrictive of the two: a larger
$\mu_i$ reinstates more antidiffusion. The gain in sharpness is bought at no cost in
monotonicity only because the discrete maximum principles of
Theorem~\ref{Thm:DMP_BJK} hold for any $\mu_i > 0$.

That the two limiters differ only through this constant makes it necessary to ask
how much of the improvement is due to bilinearity preservation and how much simply
to a larger parameter. The sensitivity study of Section~\ref{subsec:musens} answers
this, and the answer is instructive in both directions. The smearing of the layer
decreases monotonically with $\mu_i$, so the gain at a layer is a consequence of the
value of the parameter rather than of the property it encodes; and the gain is
small, because $\mu_i$ multiplies $Q_i^{\pm}$, which vanishes identically at a
strict local extremum, so that the parameter is inactive precisely at the nodes at
which the limiter acts near a layer. It follows in particular that
$\mu_i \to \infty$ does not recover the unstabilized Galerkin method: the limit is
Galerkin away from local extrema, while full artificial diffusion is retained at
them, which is what Theorem~\ref{Thm:DMP_BJK} asserts for every $\mu_i > 0$.

Increasing $\mu_i$ is nevertheless not free. The damped fixed point iteration of
\cite{JJ19} converges for the parameters up to $\mu = 5$, stagnates at
$\mu = 10$ on the finer grids, and fails to reach the prescribed tolerance at
$\mu = 50$ and $\mu = 200$. The range of usable parameters is therefore narrow, and
the value $\mu_i = 3$ required by Theorem~\ref{Thm:Explicit_bound} on a uniform
grid lies inside it, close to its upper end: the constant needed for bilinearity
preservation is close to the largest one that the recommended solver tolerates.
This also indicates where the construction should be treated with care. On strongly
graded grids $R_x$ or $R_y$ becomes small and $\mu_i = (R_xR_y)^{-1} - 1$ becomes
large, so that the parameter demanded by the geometry may leave the range in which
the nonlinear problem can be solved by this iteration. Whether another iterative
scheme extends that range, and how the bound interacts with layer-adapted grids,
which are axis-aligned tensor-product grids so that our formula applies verbatim
but with eccentricities that degenerate as $\varepsilon \to 0$, is not addressed
here and deserves a dedicated study.

Several further directions suggest themselves. The most immediate is the extension to
time-dependent convection--diffusion--reaction problems, where algebraic stabilization
is combined with a time discretization and the limiter acts on each time step; the
question is whether the geometric constants derived here transfer unchanged or whether
the time discretization imposes an additional constraint on $\mu_i$.

A second direction concerns non-axis-aligned grids. The characterization of
Theorem~\ref{thm:BMPiffBC} and the general bound Eq.~\eqref{general_formula} are not
restricted to axis-aligned patches, so a bilinearity-preserving limiter can be
constructed on any bilinearly convex patch; what is lost there is not the limiter but
the exactness, by Theorem~\ref{thm:exactrep}. It would therefore be of interest to test
Eq.~\eqref{general_formula} on general quadrilateral grids in combination with a
discretization for which the representability obstruction does not arise.

This leads to the third and, in our view, most interesting direction. Bilinear
functions are a particular family of harmonic functions, and the property that makes
them the right target the classical maximum principle is shared by all harmonic
functions. One would like a scheme that introduces no artificial diffusion whenever the
local solution is harmonic, rather than merely multilinear. The virtual element method \cite{VEM}
appears to be a natural framework for such a construction, since its local spaces are
defined implicitly through a boundary value problem and contain harmonic functions by
design, and since it accommodates general polygonal elements on which the notion of
bilinear convexity introduced here remains meaningful.

Finally, the theoretical convergence theory continues to lag behind what is observed.
The hypothesis of \cite{BJK17} that linearity preservation is responsible for optimal
convergence rates is supported by our experiments in the multilinear setting, but a
proof of optimal rates for algebraically stabilized schemes for the BJK limiter as
much as for the BJK($\mathbb{Q}_1$) limiter remains open.

\noindent{\bf Code availability}.
The Julia implementation and the data generating the numerical results of this
paper are available at \url{https://doi.org/10.5281/zenodo.22707748}.

\noindent{\bf Acknowledgement}.
The work of the second author was supported partially by the Institute grant IP/52016 and by the INSPIRE Faculty Fellowship Research Grant DST/INSPIRE/04/2024/000202.

\bibliographystyle{plain}
\bibliography{Bilinearity_Preserving}

@book{RST08,
  author    = {Roos, Hans-G{\"o}rg and Stynes, Martin and Tobiska, Lutz},
  title     = {Robust Numerical Methods for Singularly Perturbed Differential Equations},
  subtitle  = {Convection-Diffusion-Reaction and Flow Problems},
  series    = {Springer Series in Computational Mathematics},
  volume    = {24},
  edition   = {2},
  publisher = {Springer},
  address   = {Berlin, Heidelberg},
  year      = {2008},
  pages     = {604}
}

@article {JR10,
    AUTHOR = {John, Volker and Roland, Michael},
     TITLE = {On the impact of the scheme for solving the higher dimensional
              equation in coupled population balance systems},
   JOURNAL = {Internat. J. Numer. Methods Engrg.},
  FJOURNAL = {International Journal for Numerical Methods in Engineering},
    VOLUME = {82},
      YEAR = {2010},
    NUMBER = {11},
     PAGES = {1450--1474},
      ISSN = {0029-5981,1097-0207},
   MRCLASS = {76M25 (65M60 76D05 76V05)},
  MRNUMBER = {2664451},
       DOI = {10.1002/nme.2830},
       URL = {https://doi.org/10.1002/nme.2830},
}

@software{Ferrite_jl,
  author  = {Kristoffer Carlsson and Fredrik Ekre and Ferrite.jl contributors},
  title   = {Ferrite.jl},
  publisher = {Zenodo},
  doi     = {10.5281/zenodo.18196628},
  year    = {2025}
}

@article {VEM,
    AUTHOR = {Beir\~ao da Veiga, Louren\c co and Brezzi, Franco and Marini,
              L. Donatella and Russo, Alessandro},
     TITLE = {The virtual element method},
   JOURNAL = {Acta Numer.},
  FJOURNAL = {Acta Numerica},
    VOLUME = {32},
      YEAR = {2023},
     PAGES = {123--202},
      ISSN = {0962-4929,1474-0508},
   MRCLASS = {65N30},
  MRNUMBER = {4586821},
MRREVIEWER = {Feng\ Wang},
       DOI = {10.1017/S0962492922000095},
       URL = {https://doi.org/10.1017/S0962492922000095},
}

@article {Kn10,
    AUTHOR = {Knobloch, Petr},
     TITLE = {A generalization of the local projection stabilization for
              convection-diffusion-reaction equations},
   JOURNAL = {SIAM J. Numer. Anal.},
  FJOURNAL = {SIAM Journal on Numerical Analysis},
    VOLUME = {48},
      YEAR = {2010},
    NUMBER = {2},
     PAGES = {659--680},
      ISSN = {0036-1429,1095-7170},
   MRCLASS = {65N30 (65N12 65N15)},
  MRNUMBER = {2670000},
MRREVIEWER = {Ricardo\ Ruiz Baier},
       DOI = {10.1137/090767807},
       URL = {https://doi.org/10.1137/090767807},
}

@article {Ku06,
    AUTHOR = {Kuzmin, D.},
     TITLE = {On the design of general-purpose flux limiters for finite
              element schemes. {I}. {S}calar convection},
   JOURNAL = {J. Comput. Phys.},
  FJOURNAL = {Journal of Computational Physics},
    VOLUME = {219},
      YEAR = {2006},
    NUMBER = {2},
     PAGES = {513--531},
      ISSN = {0021-9991,1090-2716},
   MRCLASS = {76M10 (65M60)},
  MRNUMBER = {2274948},
       DOI = {10.1016/j.jcp.2006.03.034},
       URL = {https://doi.org/10.1016/j.jcp.2006.03.034},
}

@incollection{KM05,
  author    = {Kuzmin, Dmitri and M\"oller, Matthias},
  title     = {Algebraic Flux Correction {I}. {S}calar Conservation Laws},
  booktitle = {Flux-Corrected Transport: Principles, Algorithms, and Applications},
  editor    = {Kuzmin, Dmitri and L{\"o}hner, Rainald and Turek, Stefan},
  series    = {Scientific Computation},
  pages     = {155--206},
  publisher = {Springer},
  address   = {Berlin, Heidelberg},
  year      = {2005},
  doi       = {10.1007/3-540-27206-2_6}
}

@article {JJ19,
    AUTHOR = {Jha, Abhinav and John, Volker},
     TITLE = {A study of solvers for nonlinear {AFC} discretizations of
              convection-diffusion equations},
   JOURNAL = {Comput. Math. Appl.},
  FJOURNAL = {Computers \& Mathematics with Applications. An International
              Journal},
    VOLUME = {78},
      YEAR = {2019},
    NUMBER = {9},
     PAGES = {3117--3138},
      ISSN = {0898-1221,1873-7668},
   MRCLASS = {65N30},
  MRNUMBER = {4015770},
MRREVIEWER = {Kathrin\ Smetana},
       DOI = {10.1016/j.camwa.2019.04.020},
       URL = {https://doi.org/10.1016/j.camwa.2019.04.020},
}

@article {BH82,
    AUTHOR = {Brooks, Alexander N. and Hughes, Thomas J. R.},
     TITLE = {Streamline upwind/{P}etrov-{G}alerkin formulations for
              convection dominated flows with particular emphasis on the
              incompressible {N}avier-{S}tokes equations},
      NOTE = {FENOMECH ''81, Part I (Stuttgart, 1981)},
   JOURNAL = {Comput. Methods Appl. Mech. Engrg.},
  FJOURNAL = {Computer Methods in Applied Mechanics and Engineering},
    VOLUME = {32},
      YEAR = {1982},
    NUMBER = {1-3},
     PAGES = {199--259},
      ISSN = {0045-7825,1879-2138},
   MRCLASS = {76-08 (65N30 76R99)},
  MRNUMBER = {679322},
       DOI = {10.1016/0045-7825(82)90071-8},
       URL = {https://doi.org/10.1016/0045-7825(82)90071-8},
}

@article {BE07,
    AUTHOR = {Burman, Erik and Ern, Alexandre},
     TITLE = {Continuous interior penalty {$hp$}-finite element methods for
              advection and advection-diffusion equations},
   JOURNAL = {Math. Comp.},
  FJOURNAL = {Mathematics of Computation},
    VOLUME = {76},
      YEAR = {2007},
    NUMBER = {259},
     PAGES = {1119--1140},
      ISSN = {0025-5718,1088-6842},
   MRCLASS = {76M10 (65N30)},
  MRNUMBER = {2299768},
MRREVIEWER = {Abdellatif\ Serghini Mounim},
       DOI = {10.1090/S0025-5718-07-01951-5},
       URL = {https://doi.org/10.1090/S0025-5718-07-01951-5},
}

@article {BJK24,
    AUTHOR = {Barrenechea, Gabriel R. and John, Volker and Knobloch, Petr},
     TITLE = {Finite element methods respecting the discrete maximum
              principle for convection-diffusion equations},
   JOURNAL = {SIAM Rev.},
  FJOURNAL = {SIAM Review},
    VOLUME = {66},
      YEAR = {2024},
    NUMBER = {1},
     PAGES = {3--88},
      ISSN = {1095-7200,0036-1445},
   MRCLASS = {65M60 (65M12)},
  MRNUMBER = {4704682},
MRREVIEWER = {Abhinav\ Jha},
       DOI = {10.1137/22M1488934},
       URL = {https://doi.org/10.1137/22M1488934},
}

@book{BJK25,
  author    = {Barrenechea, Gabriel R. and John, Volker and Knobloch, Petr},
  title     = {Monotone Discretizations for Elliptic Second Order Partial
               Differential Equations},
  series    = {Springer Series in Computational Mathematics},
  volume    = {61},
  publisher = {Springer},
  address   = {Cham},
  year      = {2025},
  doi       = {10.1007/978-3-031-80684-1}
}

@article {BJK16,
    AUTHOR = {Barrenechea, Gabriel R. and John, Volker and Knobloch, Petr},
     TITLE = {Analysis of algebraic flux correction schemes},
   JOURNAL = {SIAM J. Numer. Anal.},
  FJOURNAL = {SIAM Journal on Numerical Analysis},
    VOLUME = {54},
      YEAR = {2016},
    NUMBER = {4},
     PAGES = {2427--2451},
      ISSN = {0036-1429,1095-7170},
   MRCLASS = {65N30 (65N12)},
  MRNUMBER = {3537011},
MRREVIEWER = {Thomas\ Lee\ Lewis},
       DOI = {10.1137/15M1018216},
       URL = {https://doi.org/10.1137/15M1018216},
}

@article {BJK17,
    AUTHOR = {Barrenechea, Gabriel R. and John, Volker and Knobloch, Petr},
     TITLE = {An algebraic flux correction scheme satisfying the discrete
              maximum principle and linearity preservation on general
              meshes},
   JOURNAL = {Math. Models Methods Appl. Sci.},
  FJOURNAL = {Mathematical Models and Methods in Applied Sciences},
    VOLUME = {27},
      YEAR = {2017},
    NUMBER = {3},
     PAGES = {525--548},
      ISSN = {0218-2025,1793-6314},
   MRCLASS = {65N30 (65N12 65N15)},
  MRNUMBER = {3621336},
MRREVIEWER = {S\'ebastien\ Court},
       DOI = {10.1142/S0218202517500087},
       URL = {https://doi.org/10.1142/S0218202517500087},
}

@article {Ku09,
    AUTHOR = {Kuzmin, Dmitri},
     TITLE = {Explicit and implicit {FEM}-{FCT} algorithms with flux
              linearization},
   JOURNAL = {J. Comput. Phys.},
  FJOURNAL = {Journal of Computational Physics},
    VOLUME = {228},
      YEAR = {2009},
    NUMBER = {7},
     PAGES = {2517--2534},
      ISSN = {0021-9991,1090-2716},
   MRCLASS = {76M25 (65M60)},
  MRNUMBER = {2501695},
MRREVIEWER = {So-Hsiang\ Chou},
       DOI = {10.1016/j.jcp.2008.12.011},
       URL = {https://doi.org/10.1016/j.jcp.2008.12.011},
}

@article {Zal79,
    AUTHOR = {Zalesak, Steven T.},
     TITLE = {Fully multidimensional flux-corrected transport algorithms for
              fluids},
   JOURNAL = {J. Comput. Phys.},
  FJOURNAL = {Journal of Computational Physics},
    VOLUME = {31},
      YEAR = {1979},
    NUMBER = {3},
     PAGES = {335--362},
      ISSN = {0021-9991,1090-2716},
   MRCLASS = {76X05},
  MRNUMBER = {534786},
       DOI = {10.1016/0021-9991(79)90051-2},
       URL = {https://doi.org/10.1016/0021-9991(79)90051-2},
}

@article {HMM86,
    AUTHOR = {Hughes, T. J. R. and Franca, L. P. and Mallet, M.},
     TITLE = {A new finite element formulation for computational fluid
              dynamics. {I}. {S}ymmetric forms of the compressible {E}uler
              and {N}avier-{S}tokes equations and the second law of
              thermodynamics},
   JOURNAL = {Comput. Methods Appl. Mech. Engrg.},
  FJOURNAL = {Computer Methods in Applied Mechanics and Engineering},
    VOLUME = {54},
      YEAR = {1986},
    NUMBER = {2},
     PAGES = {223--234},
      ISSN = {0045-7825,1879-2138},
   MRCLASS = {76-08 (76N10)},
  MRNUMBER = {831553},
       DOI = {10.1016/0045-7825(86)90127-1},
       URL = {https://doi.org/10.1016/0045-7825(86)90127-1},
}

@article {JJK24,
    AUTHOR = {Jha, Abhinav and John, Volker and Knobloch, Petr},
     TITLE = {Adaptive grids in the context of algebraic stabilizations for
              convection-diffusion-reaction equations},
   JOURNAL = {SIAM J. Sci. Comput.},
  FJOURNAL = {SIAM Journal on Scientific Computing},
    VOLUME = {45},
      YEAR = {2023},
    NUMBER = {4},
     PAGES = {B564--B589},
      ISSN = {1064-8275,1095-7197},
   MRCLASS = {65N55 (65N30)},
  MRNUMBER = {4622977},
MRREVIEWER = {Jinru\ Chen},
       DOI = {10.1137/21M1466360},
       URL = {https://doi.org/10.1137/21M1466360},
}

@article {ABF02,
    AUTHOR = {Arnold, Douglas N. and Boffi, Daniele and Falk, Richard S.},
     TITLE = {Approximation by quadrilateral finite elements},
   JOURNAL = {Math. Comp.},
  FJOURNAL = {Mathematics of Computation},
    VOLUME = {71},
      YEAR = {2002},
    NUMBER = {239},
     PAGES = {909--922},
      ISSN = {0025-5718,1088-6842},
   MRCLASS = {65N30},
  MRNUMBER = {1898739},
       DOI = {10.1090/S0025-5718-02-01439-4},
       URL = {https://doi.org/10.1090/S0025-5718-02-01439-4},
}
\end{document}